\documentclass[reqno]{amsart}

\usepackage[utf8]{inputenc}
\usepackage[T1]{fontenc}
\usepackage[english]{babel}
\usepackage{amsmath}
\usepackage{amssymb}
\usepackage{amsthm} 
\usepackage{mathtools}
\usepackage{nomencl} 
\usepackage{gensymb}
\makenomenclature
\makeindex
\usepackage{array}
\usepackage{ifthen}
\usepackage{graphicx}
\usepackage{lmodern}
\usepackage{amssymb}
\usepackage{multirow}
\usepackage{appendix}
\usepackage{bbm}
\usepackage{ushort}
\usepackage{enumitem}
\usepackage{afterpage}
\usepackage{tikz-cd}
\usetikzlibrary{matrix,arrows}

\usepackage{here}

\usepackage[pdftex]{hyperref}
\usepackage{pifont}
\usetikzlibrary{shapes}
\usetikzlibrary{calc}
\usetikzlibrary{arrows}
\usetikzlibrary{decorations.markings}

\DeclareMathOperator{\Ad}{Ad}
\DeclareMathOperator{\ad}{ad}
\DeclareMathOperator{\Aut}{Aut}

\DeclareMathOperator{\cl}{cl}

\DeclareMathOperator{\diff}{d}
\DeclareMathOperator{\e}{e}
\DeclareMathOperator{\End}{End}

\DeclareMathOperator{\fppf}{fppf}

\DeclareMathOperator{\Frob}{Frob}

\DeclareMathOperator{\GL}{GL}
\DeclareMathOperator{\gl}{\mathfrak{gl}}

\DeclareMathOperator{\h}{h}

\DeclareMathOperator{\Hom}{Hom}

\DeclareMathOperator{\id}{id}

\DeclareMathOperator{\Ker}{Ker}
\DeclareMathOperator{\Lie}{Lie}

\DeclareMathOperator{\nil}{Nil}
\DeclareMathOperator{\N}{N}

\DeclareMathOperator{\PGL}{PGL}
\DeclareMathOperator{\pgl}{\mathfrak{pgl}}

\DeclareMathOperator{\rad}{rad}
\DeclareMathOperator{\radu}{\mathfrak{rad_u}}
\DeclareMathOperator{\Rad}{Rad}

\DeclareMathOperator{\reg}{reg}
\DeclareMathOperator{\SL}{SL}
\DeclareMathOperator{\ssl}{\mathfrak{sl}}
\DeclareMathOperator{\SO}{SO}

\DeclareMathOperator{\Spec}{Spec}
\DeclareMathOperator{\Ss}{ss}
\DeclareMathOperator{\Sc}{sc}
\DeclareMathOperator{\tr}{Tr}
\newcommand{\cat}[1]{\underline{\mathrm{#1}}}
\newcommand{\newcategory}[2]{
    \newcommand{#1}{\cat{#2}}
}

\newcategory{\Ab}        {Ab}
\newcategory{\Alg}        {Alg}
\newcategory{\Bij}        {Bij}
\newcategory{\Cat}        {Cat}
\newcategory{\Field}    {Field}
\newcategory{\Fin}        {Fin}
\newcategory{\Graph}    {Graph}
\newcategory{\Grp}        {Grp}
\newcategory{\Mod}        {Mod}
\newcategory{\Mon}        {Mon}
\newcategory{\Op}        {Op}
\newcategory{\Pos}        {Pos}
\newcategory{\Ring}        {Ring}
\newcategory{\Set}        {Set}
\newcategory{\Top}        {Top}

\newcommand{\red}{\mathrm{red}}
\newcolumntype{P}[1]{>{\centering\arraybackslash}p{#1}}

\DeclareMathAlphabet{\mathpzc}{OT1}{pzc}{m}{it}
\newtheorem{theorem}{Theorem}[section]

\newtheorem{proposition}[theorem]{Proposition}
\newtheorem{lemma}[theorem]{Lemma}
\newtheorem{corollary}[theorem]{Corollary}

\theoremstyle{definition}
\newtheorem{defn}[theorem]{Definition}

\theoremstyle{remark}
\newtheorem{remark}[theorem]{Remark}

\theoremstyle{remark}
\newtheorem{remarks}[theorem]{Remarks}

\theoremstyle{remark}
\newtheorem{ex}[theorem]{Example}
\theoremstyle{remark}

\tikzset{
  closed/.style = {decoration = {markings, mark = at position 0.5 with { \node[transform shape, xscale = .8, yscale=.4] {/}; } }, postaction = {decorate} },
  open/.style = {decoration = {markings, mark = at position 0.5 with { \node[transform shape, scale = .7] {$\circ$}; } }, postaction = {decorate} }
}
\begin{document}

      \title{Integration questions in separably good characteristics}
	\author{Marion Jeannin}
	\email{marion.jeannin@math.uu.se}
	\address{Mathematics Department\\ Uppsala University}

\begin{abstract}
Let $G$ be a reductive group over an algebraically closed field $k$ of separably good characteristic $p>0$ for $G$. Under these assumptions a Springer isomorphism \linebreak $\phi : \mathcal{N}_{\red}(\mathfrak{g}) \rightarrow \mathcal{V}_{\red}(G)$ from the nilpotent scheme of $\mathfrak{g}$ to the unipotent scheme of $G$ always exists and allows to integrate any $p$-nilpotent element of $\mathfrak{g}$ into a unipotent element of $G$. One should wonder whether such a punctual integration can lead to an integration of restricted $p$-nil $p$-subalgebras of $\mathfrak{g}= \Lie(G)$. We provide a counter-example of the existence of such an integration in general, as well as criteria to integrate some restricted $p$-nil $p$-subalgebras of $\mathfrak{g}$ (that are maximal in a certain sense). This requires to generalise the notion of infinitesimal saturation first introduced by P. Deligne and to extend one of his theorem on infinitesimally saturated subgroups of $G$ to the previously mentioned framework.

\center{\textsc{Résumé}}

Soit $G$ un groupe réductif défini au-dessus d'un corps algébriquement clos $k$ de caractéristique $p>0$ supposée séparablement bonne pour $G$. Sous ces hypothèses, il est possible d'associer à chaque élément $p$-nilpotent de $\Lie(G)$ un élément unipotent de $G$. L'existence de cette intégration ponctuelle conduit à s'interroger sur l'existence d'une intégration ``structurelle'', à savoir : étant donnée une sous-algèbre de Lie $p$-nil $\mathfrak{u} \subset \Lie(G)$ existe-t-il un sous-groupe unipotent lisse et connexe $U\subset G$ tel que $\Lie(U) = \mathfrak{u}$ ? Cet article fournit des contre-exemples à l'existence d'une intégration structurelle en général, ainsi que des critères sur les sous-algèbres de Lie $p$-nil de $\Lie(G)$ pour lesquels une telle intégration est possible.

\end{abstract}
\maketitle

\section{Introduction}
Let $k$ be an algebraically closed field and $G$ be a $k$-group. We denote:
\begin{itemize}
\item by $\mathfrak{g}$ its Lie algebra,
\item by $G^0$ the connected component of identity,
\item by $G_{\red}$ the reduced part of $G$.
\end{itemize} 

Assume that $G$ is a reductive group. When $k$ is of characteristic $0$, the classical theory comes with the well-defined exponential map which allows to integrate any nilpotent element $x \in \mathfrak{g}$ into a unipotent element $\exp(x) \in G$. This enables to define the Baker--Campbell--Hausdorff law which is useful to endow any nilpotent Lie subalgebra of $\mathfrak{g}$ with a group law. By this process, the aforementioned Lie subalgebra becomes a unipotent group isomorphic to a unipotent subgroup of $G$. To summarize, when $k$ is of characteristic $0$: 
\begin{enumerate}
\item any nilpotent subalgebra of $\mathfrak{g}$ can be integrated into a unipotent smooth connected subgroup $U \subseteq G$ (meaning that $\Lie(U) \cong \mathfrak{u}$ as Lie algebras),
\item the exponential map induces an equivalence of categories between the category of finite dimensional nilpotent $k$-Lie algebras and the category of unipotent algebraic $k$-groups (see for example \cite[IV, \S2, n\degree 4, Corollaire 4.5]{DG}).
\end{enumerate}

If now the field $k$ is of characteristic $p>0$, one should try to determine whether it is possible to define analogues of the previously mentioned tools in order to integrate $p$-nil subalgebras of $\mathfrak{g}$. As we will explain in section \ref{sous-annexe_restreinte}, these $p$-nil subalgebras are the adequate objects to consider in characteristic $p>0$ for integration questions. The first step would be to get a punctual integration, that is, to find a way to integrate $p$-nilpotent elements of $\mathfrak{g}$ into unipotent elements of $G$. This is ensured as soon as there exists a $G$-equivariant isomorphism of reduced schemes between the reduced nilpotent scheme of $\mathfrak{g}$ (denoted by $\mathcal{N}_{\red}(\mathfrak{g})$) and the reduced unipotent scheme of $G$ (denoted by $\mathcal{V}_{\red}(G)$). Such a map $\phi : \mathcal{N}_{\red}(\mathfrak{g}) \rightarrow \mathcal{V}_{\red}(G)$ is called a Springer isomorphism. There is a technical subtlety here, which is detailed in section \ref{N_red}. For the purpose of this introduction it is only required to have in mind that in separably good characteristics (which is the framework of this article), the nilpotent scheme is reduced, while in non separably good characteristics neither the nilpotent nor the unipotent schemes are reduced. Moreover a definition of separably good integers is provided in section \ref{hyp_car}. 

Using \cite{MCNOPT} and \cite{MT2}, one can show that such an isomorphism always exists in separably good characteristics for $G$. This has been observed by P. Sobaje in \cite{SOB1}. Furthermore, the non separably good characteristics case is addressed in \cite[\S 7]{SOB2}. The author explains there why Springer isomorphisms fail to exist without this assumption. Moreover, and even if this is actually not a requirement here, one might wonder whether Springer isomorphisms are compatible with the $p$-power of the restricted Lie subalgebra one considers. We will come back to this point later in the article (see the preamble of section \ref{Springer_iso} and Remark \ref{remark_p_compatibility} ii)) but let us briefly explain the situation: such a compatibility is not always satisfied by Springer isomorphisms. Nevertheless, under mild conditions on $p$ and $G$, there is always a Springer isomorphism compatible with the $p$-structure (see \cite[Appendix 7]{McNinch2003}).

Unfortunately, the existence of a punctual integration is not sufficient to ensure a priori that any restricted $p$-nil $p$-subalgebra can be integrated into a unipotent smooth connected subgroup of $G$. If one tries to mimic the characteristic zero framework, this would actually require the Springer isomorphism $\phi$ to come with a well-defined analogue of the Baker--Campbell--Hausdorff law. This analogue would allow to make any $p$-nil subalgebra into a unipotent algebraic group. In order to exists, such a law requires even stronger conditions on $p$: let us denote by $\h(G)$ the Coxeter number of $G$. In \cite{SER2} J-P. Serre shows that when $p\geq\h(G)$ the Baker--Campbell--Hausdorff law is well-defined. Note that under this assumption on $p$, the series that defines the classical exponential map stops at the $p$-power for any nilpotent element. This is a consequence of G. McNinch's article \cite{MAUS} in which the author shows that when $p>\h(G)$ any $p$-nilpotent element has $p$-nilpotency order $1$. Let us briefly remind the reader of the proof: one actually shows that any nilpotent element satisfies $\ad^{\h(G)}(x)=0$. As the regular nilpotent elements (those with centraliser of minimal dimension) are dense in $\mathcal{N}_{\red}(\mathfrak{g})$, it is enough to show this equality when the nilpotent element $x$ is regular. In this case the result can be obtained by looking at the weights of an associated cocharacter for which the corresponding weight spaces $\mathfrak{g}_m$ are non trivial. There are at most $\h(G)$ such weights $m$ and $\ad(x)(\mathfrak{g}_m) \subseteq \mathfrak{g}_{m+2}$, hence the result. This in particular implies that, when $p>\h(G)$, the $p$-power (thus the restricted $p$-algebra structure) is compatible with the exponential map. Otherwise stated, if $x \in \mathfrak{g}$ is a $p$-nilpotent element, one indeed has $\exp(x^{[p]}) = (\exp(x))^{[p]}$ (as $x^{[p]} = 0$).

Making use of this, V. Balaji, P. Deligne and A. J. Parameswaran detail in \cite[\S 6]{BDP} the proof of the existence of an isomorphism of algebraic groups induced by the exponential map between the Lie algebra of the unipotent radical of a Borel subgroup and this unipotent radical (there, the Lie algebra is endowed with an algebraic group structure induced by the Baker--Campbell--Hausdorff law). The existence of this isomorphism implies the existence of the desired integration when $p>\h(G)$. The authors attribute this result to J.-P. Serre (see \cite{SER2}). Note that Serre's result has been refined by Seitz in \cite[Proposition 5.3]{Sei}, when $G$ is semi-simple. There, the author establishes the existence of an isomorphism of algebraic groups induced by the exponential map between the Lie algebra of the unipotent radical of a parabolic subgroup and the corresponding unipotent radical $U$ when $p$ is greater than the nilpotent class of $U$ (which is smaller than $\h(G)$). 

Once the result of V. Balaji, P. Deligne and A. J. Parameswaran has been established, one could have expected that the existence of this integration would induce, as in the characteristic $0$ framework, an equivalence of categories (this time between the category of $p$-nil Lie algebras and the category of unipotent algebraic groups). This unfortunately breaks down and justify to introduce the notion of infinitesimal saturation as defined by P. Deligne in \cite{D1} and attributed to J-P. Serre. Actually, if $p>\h(G)$ the exponential map induces a bijective correspondence between the restricted $p$-nil $p$-Lie subalgebras of $\mathfrak{g}$ and the infinitesimally saturated unipotent algebraic subgroups of $G$. All this content is explained in more details in section \ref{Springer_iso}.

In this article we focus on integration of $p$-nil subalgebras of $\mathfrak{g}$ when the characteristic $p$ is separably good for $G$, which is a weaker assumption than the characteristic $p>\h(G)$ condition. As we will show in sections \ref{candidat} and \ref{section_Deligne}, the $\fppf$-formalism introduced by P. Deligne in \cite[VIB Proposition 7.1 and Remark 7.6.1]{SGA31} provides a way of associating a smooth connected unipotent subgroup $J_{\mathfrak{u}} \subset G$ to any restricted $p$-nil $p$-subalgebra $\mathfrak{u} \subseteq \mathfrak{g}$. Unfortunately even if this subgroup is a natural candidate to integrate $\mathfrak{u}$, it is in general too big. One can indeed only expect the inclusion $\mathfrak{u} \subseteq \mathfrak{j_u} := \Lie(J_{\mathfrak{u}})$ to hold true. We provide in section \ref{section_contre_ex_integration} a counter-example to the existence of a general integration of restricted $p$-Lie algebras under the separably good characteristic assumption. 

Notwithstanding this observation, and as we will show in section \ref{ss_section_intégration_nil}, this technique still allows to integrate some restricted $p$-nil $p$-Lie algebras, as for example the $p$-radicals of Lie algebras whose normalisers are $\phi$-infinitesimally saturated (for $\phi$ a Springer isomorphism for $G$). The notion of $\phi$-infinitesimal saturation here extends the notion of infinitesimal saturation when the punctual integration comes from a Springer isomorphism that is not necessarily the truncated exponential map (as it happens for instance for small separably good characteristics for $G$). 

In section \ref{section_Deligne} we introduce this extended notion and show how, together with the aforementioned $\fppf$-formalism, this allows us to obtain a variation, in separably good characteristics, of a theorem of P. Deligne on the reduced part of infinitesimally saturated subgroups. More precisely, we show the following statement:
\begin{theorem}
Let $G$ be a reductive group over an algebraically closed field $k$ of characteristic $p>0$ which is assumed to be separably good for $G$. Let $\phi: \mathcal{N}_{\red}(\mathfrak{g}) \rightarrow \mathcal{V}_{\red}(G)$ be a Springer isomorphism for $G$ and let $N\subseteq G$ be a $\phi$-infinitesimally saturated subgroup. Then:
\begin{enumerate}
\item the subgroups $N_{\red}^0$ and $\Rad_U(N^0_{\red})$ are normal in $N$. Moreover, the quotient $N/N^0_{\red}$ is a $k$-group of multiplicative type;
\item in addition, suppose that the connected reduced subgroup $N^0_{\red}$ is reductive. Then there exists in $N^0$ a central subscheme $M$ of multiplicative type such that $(M^0\times N_{\red}^0)/\mu \cong N^0$, where $\mu$ is the kernel of $M^0 \times N_{\red}^0 \rightarrow N^0$.
\end{enumerate}
\label{generalisation_Deligne}
\end{theorem}

Section \ref{appendice_Lie}, finally, is a miscellany of technical results used in the proofs of several statements of this paper.

Let us moreover stress out that even if after reading this introduction an integration seems to be possible only under very specific and restrictive conditions on the restricted $p$-nil $p$-subalgebras, the results presented in this article still allow to extend theorems classically known in characteristic zero to the characteristic $p$ framework. For instance analogues of Morozov Theorem can be obtained with these techniques (see \cite{JEAthese}, this will also be developed in more details in a future article). The latter states the following: let $G$ be a reductive group over a field $k$ of characteristic $0$, if $\mathfrak{u}\subset \mathfrak{g}$ is a nilpotent algebra which is the nilradical of its normaliser $N_{\mathfrak{g}}(\mathfrak{u})$, this normaliser is the Lie algebra of a parabolic subgroup of $G$. Obtaining analogues of this statement was the first motivation to study the questions raised in this paper. A subsidiary part of the content of this article comes from the author's Ph.D. manuscript \cite{JEAthese}.

\section{Context}
\subsection{Hypotheses on the characteristic}
\label{hyp_car}
Let $k$ be a field of characteristic $p>0$ and $G$ be a reductive $k$-group. This section is dedicated to discuss usual assumptions made on the characteristic of $k$. We refer the reader to \cite{Stei_prime} and \cite[\S 0.3]{SPR} for a definition and an exhaustive list of torsion characteristics for $G$. Good and very good characteristics are discussed for instance in the preamble of \cite{LMT} or in \cite[\S 2]{HER}. We only recall here some useful facts.

In what follows $k$ is assumed to be algebraically closed. When $G$ is a semisimple $k$-group the following statement is a consequence of \cite[Theorem 2.2 and Remark a)]{LMT} :

\begin{corollary}[(Corollary of {\cite[Theorem 2.2]{LMT}})]
Let $G$ be a semisimple group over an algebraically closed field $k$ of characteristic $p>0$ which is not of torsion for $G$. Let $\mathfrak{u} \subseteq \mathfrak{g}$ be a restricted $p$-nil $p$-subalgebra (see section \ref{sous-annexe_restreinte}). Then there exists a Borel subgroup $B\subset G$ such that $\mathfrak{u}$ is a subalgebra of $\mathfrak{b}:=\Lie(B)$.
\label{corollaire_LMT}
\end{corollary}

\begin{remarks}
The following remarks will be of main importance in the integration process described in this article:
\begin{enumerate}
	\item the subalgebra $\mathfrak{u}$ is actually contained in the Lie algebra of the unipotent radical of a Borel subgroup $B \subseteq G$. Indeed $\mathfrak{b}$ is nothing but the semidirect sum of the Lie algebra of the unipotent radical of $B$, denoted by $\radu(B)$ and the Lie algebra of a maximal torus of $G$, denoted by $\mathfrak{t}$. This last factor contains no $p$-nilpotent element (see the preamble of the subsection \ref{sous-annexe_restreinte}), whence the inclusion $\mathfrak{u} \subset \radu(B)$.
	\item The first point of this remark actually allows to generalise the corollary to any reductive $k$-group $G$, when $k$ is an algebraically closed field of characteristic $p>0$ that is not a torsion integer for $G$. Let $Z(G)$ be the center of $G$. Let also $\pi: G \rightarrow G' := G/Z^0_{\red}(G)$ be the quotient map and set $\mathfrak{u}':= \Lie(\pi)(\mathfrak{u})$. As $\Lie(Z^0_{\red}(G))$ is the Lie algebra of a torus it has no $p$-nilpotent element (this is detailed at the end of the proof of Lemma \ref{radical_reductif_centre}, note that the assumption made in the statement of this lemma is not necessary to prove this specific fact). Therefore one has $\mathfrak{u} \cong \mathfrak{u'}$. By what precedes there exists a Borel subgroup $B' \subset G'$ such that $\mathfrak{u'} \subseteq \radu(B') \subset \mathfrak{b'}$. Let $B= \pi^{-1}(B')$ be the preimage of $B'$. As $\radu(B') \cong \radu(B)$ one can always assume that $\mathfrak{u}$ is the subalgebra of the Lie algebra of the unipotent radical of a Borel subgroup of $\mathfrak{b} \subseteq \mathfrak{g}$.
	\end{enumerate}
\end{remarks}

Separably good characteristics are defined by J. Pevtsova and J. Stark in \cite[Definition 2.2]{PevSta}: 
\begin{enumerate}
	\item when $G$ is semisimple, let $G^{\Sc}$ be the simply connected cover of $G$. The characteristic $p$ is separably good for $G$ if $p$ is good for $G$ and if the isogeny $G^{\Sc} \rightarrow G$ is separable. 
	\item When $G$ is reductive, the characteristic $p$ is separably good for $G$ if it is separably good for its derived group $[G,G]$. 
	\end{enumerate}
\noindent As underlined by the two authors, if $p$ is very good for $G$, it is also separably good. Nevertheless, this last condition is only restrictive for type $A$, which is the only type for which very good and separably good characteristics do not coincide. As an example $p$ is separably good but not very good for $\SL_p$ or $\GL_p$. However, it is not separably good nor very good for $\PGL_p$. 

Moreover, let $G$ be a reductive algebraic group over an algebraically closed field $k= \bar{k}$ and consider a maximal torus $T \subsetneq G$. The tuple $\mathcal{R}(G) = (X(T), \Phi, Y(T), \Phi^{\vee})$ whose components are respectively the associated group of characters, the root system, the group of cocharacters and the coroot system, is a root datum for $G$. This root datum is unique up to isomorphism (see \cite[XXII, 2.6]{SGA33}). A prime number is pretty good for $G$ if, given any subset $\Phi'\subseteq \Phi$ both the groups $X(T)/\mathbb{Z}\Phi'$ and $Y(T)/\mathbb{Z}\Phi'^{\vee}$ have no $p$-torsion. Note that these definitions still make sense when $k$ is no longer algebraically closed but this goes beyond the framework of this article. Once again, this condition answers type $A$-phenomenon. It is studied by S. Herpel in \cite{HER}. In particular, pretty good and very good primes are the same when $G$ is semisimple (see \cite[Lemma 2.12]{HER}). For instance $p$ is not pretty good for $\SL_p$. However, if $G$ is an arbitrary group, being a very good prime is a more restrictive condition, indeed $p$ is pretty good but not very good for $\GL_p$ (see \cite[Example 2.13]{HER}). Finally, as explained in \cite[2.4]{Stei_prime}, if $p$ doesn't divide the order of $X(T)/\mathbb{Z}\Phi$ then $p$ is separably good for $G$. Hence any pretty good prime is separably good for $G$. In particular, as $p$ is not separably good for $\PGL_p$ it is not pretty good either. To summarize, one has the following chain of implications:
\[\text{Very good} \implies \text{pretty good} \implies \text{separably good} \implies \text{good} \implies \text{non torsion}.\]

\subsection{From characteristic zero to positive characteristics, defining the good analogues: sorites on restricted \texorpdfstring{$p$}{Lg}-Lie algebras}
	\label{sous-annexe_restreinte}

Before going any further, one needs to introduce the good analogues in characteristic $p>0$ for the objects involved in the characteristic zero setting. This is done in this section. The results presented below are stated in the most general way. In particular, we do not assume a priori (and unless explicitely stated) in this subsection that the field $k$ is algebraically closed.

Let $\mathfrak{g}$ be a finite dimensional restricted $p$-nil $p$-Lie algebra over $k$. In what follows we denote by $[p]$ the $p$-structure for $\mathfrak{g}$. Let us stress out that, in particular, the Lie algebra of any $k$-group scheme $G$ is endowed with such a $p$-structure (see \cite[II, \S7, n\degree 3.4]{DG}). Moreover, for any algebraic subgroup $H \subset G$, the $p$-structure on $\Lie(H) := \mathfrak{h}$ inherited from the group is compatible with the one on $\mathfrak{g}$. In other words $\mathfrak{h}$ is a restricted $p$-subalgebra of $\mathfrak{g}$. We refer the reader to \cite[\S 2 Définition]{FS} for general theory of restricted $p$-Lie algebras.

Let $k$ be a field and let $\mathfrak{g}$ be a $k$-Lie algebra. As a reminder:
\begin{enumerate}
	\item the solvable radical (or radical) of $\mathfrak{g}$, denoted by $\rad(\mathfrak{g})$, is the largest solvable ideal of $\mathfrak{g}$ (see \cite[\S 1.7, Definition]{FS}),
	\item the nilradical of $\mathfrak{g}$, denoted by $\nil(\mathfrak{g})$, is the largest nilpotent ideal of $\mathfrak{g}$. In particular all its elements are $\ad$-nilpotent, by a corollary of Engel Theorem (see for example \cite[\S 4 n\degree 2 Corollaire 1]{BOU1}). When $k$ is of characteristic $0$, the nilradical is nothing but the set of $\ad$-nilpotent elements of the radical of $\mathfrak{g}$ (see \cite[\S 1, Corollary 3.10]{FS} and \cite[\S 5, Corollaire 7]{BOU1}). Let us stress out that the equality $\nil(\mathfrak{g}/\nil(\mathfrak{g}))= 0$ is not always satisfied when $k$ is of characteristic $p>0$ (see \cite[p.~20]{FS} for a counter-example).
	\item A subalgebra $\mathfrak{h} \subseteq \mathfrak{g}$ is nil if any element of $\mathfrak{h}$ is $\ad$-nilpotent for the bracket on $\mathfrak{g}$. Any nil and finite dimensional $k$-Lie algebra is nilpotent.
\end{enumerate}

One may wonder whether these classical objects inherit of a $p$-structure compatible with the one of $\mathfrak{g}$:

\begin{lemma}
Let $\mathfrak{h}$ be a restricted $p$-Lie algebra over $k$. Then $\rad(\mathfrak{h})$ is a restricted $p$-subalgebra of $\mathfrak{h}$.
\label{radical_p_ss_alg}
\end{lemma}

\begin{proof}
Let us consider the morphism of Lie algebras $\mathfrak{h} \twoheadrightarrow \mathfrak{h/\rad(h)}$. According to \cite[1, \S7, Theorem 7.2]{FS} one has $\rad(\mathfrak{h/\rad(h)}) = 0$, thus the center $\mathfrak{z}_{\rad(\mathfrak{h/\rad(h)})}$ is trivial (because $\mathfrak{z_g} \subseteq \rad(\mathfrak{g})$, see for instance the first lines of the proof of Lemma \ref{radical_reductif_centre}). By \cite[2.3, Exercise 7]{FS}, the radical of $\mathfrak{h}$ is a $p$-Lie subalgebra.
\end{proof}

Assume the Lie algebra $\mathfrak{g}$ derives from an affine algebraic $k$-group. Let $\rho : G \rightarrow \GL(V)$ be a faithful representation of finite dimension. An element $x \in \mathfrak{g}$ is nilpotent, or $\mathfrak{g}$-nilpotent, if $\Lie(\rho)(x)$ is a nilpotent element of $\mathfrak{\gl}(V)$ (let us stress out that $\Lie(\rho)$ is still injective because the Lie functor is left exact (see \cite[II,\S4, 1.5]{DG})). On the same way, an element $x \in \mathfrak{g}$ is semisimple if $\Lie(\rho)(x)$ is a semisimple element of $\mathfrak{\gl}(V)$. These notions are independent from the choice of the faithful representation $\rho$ (see \cite[I.4.4, Theorem]{BORlag}). Let us emphasize that when $k$ is perfect any $x \in \mathfrak{g}$ has a Jordan decomposition in $\mathfrak{g}$ (see for example \cite[I.4.4, Theorem]{BORlag}). 

More generally, if one does no longer consider that $\mathfrak{g}$ is the Lie algebra of an algebraic group, then: 
\begin{itemize}
\item if $\mathfrak{g}$ is a semisimple Lie algebra over a field of characteristic $0$ (whatever the characteristic,  semisimple Lie algebras are those with trivial solvable radical), any element $x \in \mathfrak{g}$ has a unique Jordan decomposition (see for example \cite[\S 6 n\degree 3 Théorème 3]{BOU1}). 
\item Similarly, if $k$ is a perfect field of characteristic $p>0$ and $\mathfrak{g}$ is finitely generated restricted $p$-Lie algebra, a decomposition $x = x_s + x_n$ (with $x_s$ semisimple and $x_n$ nilpotent) always exists, with the additional condition for the nilpotent part to be $p$-nilpotent (see \cite[2.3 Theorem 3.5]{FS}). 
\end{itemize}
An element $x \in \mathfrak{g}$ is $p$-nilpotent if there exists an integer $m \in \mathbb{N}$ such that $x^{[p^m]}=0$. When it exists, the smallest $m \in \mathbb{N}$ such that $x^{[p^m]}=0$ is called the order of $p$-nilpotency of $x$. In this framework, an element $x \in \mathfrak{g}$ is $p$-semisimple if $x$ belongs to the restricted $p$-Lie algebra generated by $x^{[p]}$. Finally, an element $x \in \mathfrak{g}$ is toral if $x^{[p]} = x$. According to \cite[\S 2 Proposition 3.3]{FS} and the remark that follows this proposition, both definitions of semisimplicity are equivalent. In what follows an element is thus said to be $p$-semisimple (respectively $p$-nilpotent) if it is semisimple (respectively $\mathfrak{g}$-nilpotent). This equivalence of definitions is a consequence of Iwasawa Theorem (see \cite{IWA}) which ensures that any Lie subalgebra of finite dimension over a field of characteristic $p>0$ has a faithful representation. This result has afterwards been extended by N. Jacobson to the framework of finite dimensional restricted $p$-Lie algebras with the additional constraint that the involved representation is compatible with the $p$-structure (see \cite{JAC} and \cite[I, \S4, Theorem I.4.2]{SEL}). 

Let $k$ be a field of characteristic $p>0$. Let $\mathfrak{h}$ be a restricted $p$-algebra (as previously mentioned this is in particular the case if $\mathfrak{h}$ derives from a subgroup $H \subset G$). The restricted $p$-subalgebra $\mathfrak{h}$ is $p$-nilpotent if there exists an integer $n \in \mathbb{N}$ such that $\mathfrak{h}^{[p^n]} = 0$. When $\mathfrak{g}$ is of finite dimension any restricted $p$-subalgebra which is $p$-nilpotent is also $p$-nil (that is, any of its elements are $p$-nilpotent). 

It is worth noting that the study of ideals of $\mathfrak{g}$ that consist only in semisimple elements can also be very instructive. Let us recall the following result as an illustration (see \cite[Proposition 2.13]{BorT2}): let $\mathfrak{g}$ be the Lie algebra of a reductive $k$-group $G$. We consider the action of $G$ on $\mathfrak{g}$ by conjugation. Let $\mathfrak{j} \subseteq \mathfrak{g}$ be an ideal which in $G$-stable. Then $\mathfrak{j}$ consists only in semisimple elements if and only if $\mathfrak{j} \subseteq  \mathfrak{z_g}$. 


Let us finally underline that, although in positive characteristic the nilradical of a restricted $p$-algebra is well-defined, it does no longer satisfy the properties it had in characteristic $0$. Hence the necessity of introducing the following object which appears to be, under some additional hypotheses, the good analogue to consider in characteristic $p>0$:

\begin{defn} Let $\mathfrak{h}$ be a restricted $p$-algebra. The $p$-radical of $\mathfrak{h}$, denoted by $\rad_p(\mathfrak{h})$, is the maximal $p$-nilpotent $p$-ideal of $ \mathfrak{h}$ (such an object exists, see for instance \cite[2.1, Corollary 1.6]{FS}).
\label{def_p-rad} 
\end{defn}

Let us also stress out that the Lie algebra of the unipotent radical of a connected algebraic group $H$, denoted by $\mathfrak{\radu}(H)$, is an ideal of $\nil(\mathfrak{h})$ (as $U$ is a unipotent normal subgroup of $\Rad(H)$). We aim to compare these different objects: 

\begin{lemma}
Let $\mathfrak{h}$ be a restricted $p$-algebra. Then:
\begin{enumerate}
	\item the inclusions $\rad_p(\mathfrak{h}) \subseteq \nil(\mathfrak{h}) \subseteq \rad(\mathfrak{h})$ are satisfied,
	\item the $p$-radical of $\mathfrak{h}$ is a subset of the set of all $p$-nilpotent elements of $\rad(\mathfrak{h})$.
	\item Let us denote by $\mathfrak{z_h}$ the center of $\mathfrak{h}$. The equality $\rad_p(\mathfrak{h}) = \nil(\mathfrak{h})$ holds true if and only if the inclusion $\mathfrak{z_h} \subseteq \rad_p(\mathfrak{h})$ is satisfied.
\end{enumerate} 
\label{inclusion_p-struct}
\end{lemma}

\begin{proof}
We show each point of the lemma separately:
\begin{enumerate}
\item the inclusion $\rad_p(\mathfrak{h}) \subseteq \nil(\mathfrak{h})$ is clear as $\rad_p(\mathfrak{h})$ is a nil ideal of $\mathfrak{h}$ (because it is $p$-nil). Hence it is a nilpotent ideal of $\mathfrak{h}$ because the Lie algebras involved here are of finite dimension. 

The second inclusion is also direct as any nilpotent ideal is in particular solvable (see for example \cite[\S1.5 Remark]{FS}). Hence the first point of the lemma is shown.

\item This last inclusion being satisfied and $\rad_p(\mathfrak{h})$ being $p$-nil, the restricted $p$-ideal is necessarily contained in the set of all $p$-nilpotent elements of $\rad(\mathfrak{h})$. This ends the proof of (ii).

\item The center of $\mathfrak{h}$ is an abelian ideal of $\mathfrak{h}$. It is therefore contained in the nilradical of $\mathfrak{h}$. Thus if one has the equality $\nil(\mathfrak{h})=\rad_p(\mathfrak{h})$, one also has the inclusion $\mathfrak{z_h}\subseteq \rad_p(\mathfrak{h})$. 

Reciprocally, assume the inclusion $\mathfrak{z_h} \subseteq \rad_p(\mathfrak{h})$ to be satisfied and let us show that any $x \in \nil(\mathfrak{h})$ is $p$-nilpotent. First, it is $\ad$-nilpotent according to Corollary \cite[\S4 n\degree 2 Corollaire 1]{BOU1} because the ideal $\nil(\mathfrak{h})$ is nilpotent. Moreover, as the Lie algebra $\mathfrak{h}$ is endowed with a $p$-structure, there exists an integer $n$ such that $\ad(x)^{p^n}= 0 = \ad(x^{[p^n]}).$ In other words $x^{[p^n]}$ belongs to the center of $\mathfrak{h}$. As we assumed the inclusion $\mathfrak{z_h} \subseteq \rad_p(\mathfrak{h})$ to hold true, the element $x^{[p^n]}$  is actually $p$-nilpotent (the $p$-radical being $p$-nil). Hence there exists an integer $m$ such that $(x^{[p^n]})^{[p^m]} = (x ^{[p^{n+m}]}) =0$, whence the $p$-nilpotency of any element of $\nil(\mathfrak{h})$. This implies that $\nil(\mathfrak{h})$ is a restricted $p$-ideal $p$-nil of $\mathfrak{h}$, since the nilradical of $\mathfrak{h}$ is a restricted $p$-ideal according to Lemma \cite[2.3, Exercise 5d]{FS}. This leads to the desired equality. Thus we have shown (iii).
\end{enumerate}
\end{proof}

When $\mathfrak{g}$ derives from a smooth connected algebraic $k$-group $G$ these objects should be compared with the Lie algebra of the radical (respectively of the unipotent radical) of $G$.

\begin{lemma}
Let $k$ be a field of characteristic $p \geq 3$ and $G$ be a reductive $k$-group. Then the equalities $\mathfrak{z_g} = \rad(\mathfrak{g}) = \nil(\mathfrak{g})$ hold true.
 \label{radical_reductif_centre}
\end{lemma}

\begin{remark}
The assumption on the characteristic allows a uniform proof of the above lemma. Notwithstanding this point, it is worth noting that the characteristic $2$ case can be handled by a case-by-case analysis (by making use of \cite[table 1]{Ho}). Moreover, Lemma \ref{nil_red} below provides the equality $\mathfrak{z_g} = \nil(\mathfrak{g})$ (which is a weaker result) in any characteristic $p>0$. This last statement appears as a Corollary of \cite[Lemma 2.1]{VAS}.
\label{rem_car_2}
\end{remark}

The following lemma is useful in the proof of Lemma \ref{radical_reductif_centre}:
\begin{lemma}
Let $\widetilde G$ and $G$ be two reductive $k$-groups and let us consider the following central exact sequence of algebraic groups:
\begin{figure}[H]
\begin{center}
\[\begin{tikzpicture} 

 \matrix (m) [matrix of math nodes,row sep=2em,column sep=4.8em,minimum width=2em,  text height = 1.5ex, text depth = 0.25ex]
  {
    1 & S & \widetilde{G} & G & 1.\\
  };
  \path[-stealth]
  	(m-1-1) edge (m-1-2)
    (m-1-2) edge node[above] {$\iota$}(m-1-3) 
    (m-1-3) edge node[above] {$\pi$} (m-1-4) 
    (m-1-4) edge (m-1-5)
    ;  	
\end{tikzpicture}\]
\end{center}
\end{figure}
\noindent Let also $\widetilde T \subseteq \widetilde{G}$ be a maximal $k$-torus and set $T:=\widetilde{T} /S$.
Then
$\Lie(\pi)(\widetilde{\mathfrak{g}})$ is an ideal of $\mathfrak{g}$ and the quotient $\mathfrak{g}/ \Lie(\pi)(\widetilde{\mathfrak{g}})$ is isomorphic to $\mathfrak{t}/ \Lie(\pi)(\widetilde{\mathfrak{t}})$ as a $k$-Lie algebra.
In particular if $k$ is of characteristic $p>0$, the restricted $p$-Lie algebra $\mathfrak{g}/ \Lie(\pi)(\widetilde{\mathfrak{g}})$ is toral. 
\label{quotient_central}
\end{lemma}

\begin{proof}
The center of a reductive group is a diagonalisable subgroup (see for instance \cite[XXII, Corollaire 4.1.6]{SGA33}). The exact sequence of the lemma being central, the $k$-group $S$ is diagonalisable. Indeed any subgroup of a diagonalisable group defined over a field is diagonalisable (see \cite[IX, Proposition 8.1]{SGA32}). Let $E$ be a $k$-torus such that $S^0 \subseteq E$. Let us stress out that such an object always exists because the maximal connected subgroups of multiplicative type of a reductive group over a field are the maximal tori (see Corollary \ref{tm_connexe_max_cas_red}). Consider the following commutative diagram of algebraic $k$-groups:
\begin{figure}[H]
\begin{center}
\[\begin{tikzpicture} 

 \matrix (m) [matrix of math nodes,row sep=2em,column sep=4.8em,minimum width=2em,  text height = 1.5ex, text depth = 0.25ex]
  {	  & 1 & 1\\
	  & \mathbb{G}_m^r & \mathbb{G}_m^r\\
  	1 & E & G' & G & 1\\
    1 & S^0 & \widetilde{G} & G & 1,\\
    & 1 & 1\\
  };
  \path[-stealth]
  	(m-2-2) edge (m-2-3)
    
    (m-2-2) edge (m-1-2) 
    (m-2-3) edge (m-1-3) 
    
    (m-3-2) edge (m-2-2) 
    (m-3-3) edge node[left] {$q$} (m-2-3) 
    
	(m-4-2) edge (m-3-2) 
    (m-4-3) edge node[left] {$i$} (m-3-3)     
    (m-4-4) edge node[left] {$=$}(m-3-4)

  	(m-3-1) edge (m-3-2)
    (m-3-2) edge (m-3-3) 
    (m-3-3) edge node[above] {$\pi'$} (m-3-4) 
    (m-3-4) edge (m-3-5) 
    
    (m-4-1) edge (m-4-2)
    (m-4-2) edge (m-4-3) 
    (m-4-3) edge node[above] {$\pi$} (m-4-4) 
    (m-4-4) edge (m-4-5) 
    
    (m-5-2) edge (m-4-2) 
    (m-5-3) edge (m-4-3) 
    ;  	
\end{tikzpicture}\]
\end{center}
\end{figure}
\noindent where $G'$ is defined for the lower left square to be commutative. It induces by derivation a commutative diagram of Lie algebras: 
\begin{figure}[H]
\begin{center}
\[\begin{tikzpicture} 

 \matrix (m) [matrix of math nodes,row sep=2em,column sep=4.8em,minimum width=2em,  text height = 1.5ex, text depth = 0.25ex]
  {	  &		& 0 \\
  	  & k^r & k^r\\
  	0 & k^r & \mathfrak{g}' & \mathfrak{g} & 0.\\
    0 & \mathfrak{s} & \widetilde{\mathfrak{g}} & \mathfrak{g}\\
      &	0	& 0 \\
  };
  
  \path[-stealth]
  
  	(m-2-3) edge (m-1-3)
  	(m-5-3) edge (m-4-3)
  	(m-5-2) edge (m-4-2)
  	
  	(m-2-2) edge (m-2-3)
    
    (m-3-2) edge (m-2-2) 
    (m-3-3) edge  node[left] {$\Lie(q)$}(m-2-3) 
    
	(m-4-2) edge (m-3-2) 
    (m-4-3) edge node[left] {$\Lie(i)$} (m-3-3)     
    (m-4-4) edge node[left] {$=$}(m-3-4)
 	   
  	(m-3-1) edge (m-3-2)
    (m-3-2) edge (m-3-3) 
    (m-3-3) edge node[above] {$\Lie(\pi')$} (m-3-4) 
    (m-3-4) edge (m-3-5) 
    
    (m-4-1) edge (m-4-2)
    (m-4-2) edge (m-4-3) 
    (m-4-3) edge node[above] {$\Lie(\pi)$} (m-4-4) 
    ;  	
\end{tikzpicture}\]
\end{center}
\end{figure}
\noindent Note that the right-exactness of the second line comes from the smoothness of $\Ker(\pi')$ (see \cite[II, \S5, n\degree5, Proposition 5.3]{DG}). 

We show that $\Lie(\pi)(\widetilde{\mathfrak{g}})$ is an ideal of $\mathfrak{g}$: let $y \in \Lie(\pi)(\widetilde{\mathfrak{g}}) \subseteq \mathfrak{g}$ and and pick $g \in \mathfrak{g}$. Let also $x \in \widetilde{\mathfrak{g}}$ be such that $\Lie(\pi)(x) = y$. As $\Lie(\pi')$ is surjective there exists $g' \in\mathfrak{g}'$ such that $\Lie(\pi)(g') = g$. This provides the equality: \[[y,g] = [\Lie(\pi)(x), \Lie(\pi')(g')] = [\Lie(\pi')\circ \Lie(i)(x), \Lie(\pi')(g')] = \Lie(\pi')([\Lie(i)(x),g']),\] 
\noindent The Lie algebra $\widetilde{\mathfrak{g}}$ is isomorphic to the kernel of $\Lie(q) : \mathfrak{g}'\rightarrow k^r$ which is an ideal of $\mathfrak{g'}$. The commutativity of the diagram thus allows us to conclude that $[y,g] \in \Lie(\pi)(\widetilde{\mathfrak{g}})$. Therefore $\Lie(\pi)(\widetilde{\mathfrak{g}})$ is an ideal of $\mathfrak{g}$.

It remains to prove that the inclusion $\Lie(\pi)(\widetilde{\mathfrak{t}}) \subseteq \Lie(\pi)(\widetilde{\mathfrak{g}}) \cap \mathfrak{t}$ is actually an equality. This being established, one will only need to apply \cite[Corollaire 2.17]{BorT2} to end the proof (as this corollary states that $\mathfrak{t} \twoheadrightarrow \mathfrak{g}/\Lie(\pi)(\widetilde{\mathfrak{g}})$ is surjective). Let us thus show the equality $\Lie(\pi)(\widetilde{\mathfrak{t}}) = \Lie(\pi)(\widetilde{\mathfrak{g}}) \cap \mathfrak{t}$. It comes from the study of the right lower square of the above commutative diagram of groups: the morphism $\pi'$ being surjective with toric kernel $E$, the group $T$ is the image of a torus $T' \subseteq G'$ (by \cite[IX, Proposition 8.2 (ii)]{SGA32}). Hence the equalities $T = T'/E = \widetilde{T}/S$ hold true. The following square 

\begin{figure}[H]
\begin{center}
\[\begin{tikzpicture} 

 \matrix (m) [matrix of math nodes,row sep=2em,column sep=4.8em,minimum width=2em,  text height = 1.5ex, text depth = 0.25ex]
  {	 
  	G' & G\\
    \widetilde{G} & G,\\
  };
  \path[-stealth]

  	(m-2-1) edge  node[left]{$i$} (m-1-1)
    (m-2-2) edge (m-1-2) 
    (m-1-1) edge node[above] {$\pi'$} (m-1-2) 
    (m-2-1) edge node[above] {$\pi$} (m-2-2) 
    ;  	
\end{tikzpicture}\]
\end{center}
\end{figure}
\noindent is commutative. The image $i(\widetilde{T})$ is thus contained in $T'$. Hence the exact sequence:
\begin{figure}[H]
\begin{center}
\[\begin{tikzpicture} 

 \matrix (m) [matrix of math nodes,row sep=2em,column sep=4.8em,minimum width=2em,  text height = 1.5ex, text depth = 0.25ex]
  {
    1 & \widetilde{G} & G' & \mathbb{G}_m^r & 1\\
  };
  \path[-stealth]
  	(m-1-1) edge (m-1-2)
    (m-1-2) edge node[above] {$i$}(m-1-3) 
    (m-1-3) edge (m-1-4) 
    (m-1-4) edge (m-1-5)
    ;  	
\end{tikzpicture}\]
\end{center}
\end{figure}
\noindent induces an exact sequence of tori:
\begin{figure}[H]
\begin{center}
\[\begin{tikzpicture} 

 \matrix (m) [matrix of math nodes,row sep=2em,column sep=4.8em,minimum width=2em,  text height = 1.5ex, text depth = 0.25ex]
  {
    1 & \widetilde{T} & T' & T'' & 1.\\
  };
  \path[-stealth]
  	(m-1-1) edge (m-1-2)
    (m-1-2) edge node[above] {$i$}(m-1-3) 
    (m-1-3) edge (m-1-4) 
    (m-1-4) edge (m-1-5)
    ;  	
\end{tikzpicture}\]
\end{center}
\end{figure}
\noindent Note that the subgroup $T''$ is indeed a torus as it is:
\begin{itemize}
\item diagonalisable according to \cite[IX, Proposition 8.1]{SGA32},
\item smooth by \cite[II, \S5, n\degree 5, Proposition 5.3 (ii)]{BDP}).
\end{itemize} 
\noindent The exactness is here preserved by derivation as $\widetilde{T}$ is smooth. 

Let us now consider the right lower square of the above commutative diagram of Lie algebras:

\begin{figure}[H]
\begin{center}
\[\begin{tikzpicture} 

 \matrix (m) [matrix of math nodes,row sep=2em,column sep=4.8em,minimum width=2em,  text height = 1.5ex, text depth = 0.25ex]
  {&&0 &  & \\
 && k^r &  & \\
  0&  k^r & \mathfrak{g}' & \mathfrak{g} & 0.\\
  &&  \widetilde{\mathfrak{g}} & \mathfrak{g} & \\
  && 0\\};
  \path[-stealth]
  	(m-2-3) edge (m-1-3)
  	
    (m-3-3) edge (m-2-3) 
    (m-4-3) edge node[left] {$\Lie(i)$}(m-3-3) 
(m-4-4) edge node[left] {$=$} (m-3-4)     
    
    (m-3-1) edge (m-3-2)
    (m-3-2) edge (m-3-3) 
    (m-3-3) edge node[above] {$\Lie(\pi')$} (m-3-4) 
    (m-3-4) edge (m-3-5)
    
    (m-4-3) edge node[above] {$\Lie(\pi)$} (m-4-4)  
    
    (m-5-3) edge (m-4-3)    
    ;  	
\end{tikzpicture}\]
\end{center}
\end{figure}
\noindent The kernel $E$ being smooth, the derived morphism $\Lie(\pi')$ is still surjective. Hence one still has $\mathfrak{t} = \mathfrak{t'}/k^r$. According to what precedes any $y \in \Lie(\pi)(\widetilde{\mathfrak{g}}) \cap \mathfrak{t}$ is the image of a certain $x \in \widetilde{\mathfrak{g}}$ such that $\Lie(i)(x) \in \mathfrak{t}'$. This, combined with the exactness of the following derived exact sequence: 
\begin{figure}[H]
\begin{center}
\[\begin{tikzpicture} 

 \matrix (m) [matrix of math nodes,row sep=2em,column sep=4.8em,minimum width=2em,  text height = 1.5ex, text depth = 0.25ex]
  {
    0 & \widetilde{\mathfrak{t}} & \mathfrak{t}' & \mathfrak{t}'' & 0,\\
  };
  \path[-stealth]
  	(m-1-1) edge (m-1-2)
    (m-1-2) edge node[above] {$\Lie(i)$}(m-1-3) 
    (m-1-3) edge (m-1-4) 
    (m-1-4) edge (m-1-5)
    ;  	
\end{tikzpicture}\]
\end{center}
\end{figure} 
\noindent allows us to conclude. The exactness indeed ensures that $x \in \widetilde{\mathfrak{t}}$. Moreover, since one has that \linebreak $y = \Lie(\pi)(x)= \Lie(\pi')(i(x)) \in \Lie(\pi)(\widetilde{\mathfrak{t}})$, the expected inclusion, thus the equality, are obtained.
\end{proof}

\begin{proof}[Proof of Lemma \ref{radical_reductif_centre}]
The center $\mathfrak{z_g}$ is a nilpotent ideal of $\mathfrak{g}$, it is therefore solvable. The inclusions $\mathfrak{z_g}\subseteq \nil(\mathfrak{g}) \subseteq \rad(\mathfrak{g})$ follow. One thus only needs to show that $\rad(\mathfrak{g}) \subseteq \mathfrak{z_g}$. The involved objects being all compatible with base change we can without loss of generality assume $k$ to be algebraically closed.

A dévissage argument allows us to reduce ourselves to prove the statement for $G$ connected and semisimple: the reductive case can be deduced from the semisimple one, while the latter is ruled by the semisimple and simply connected case.
	\begin{enumerate}
		\item Assume the $k$-group $G$ to be semisimple and simply connected. It thus decomposes into a product of almost simple groups (see \cite[3.1.1, p.~55]{Tit}) and one can assume without loss of generality that $G$ is almost simple. There are two options:
		\begin{enumerate}
		 \item either $G$ is not of type $G_2$ when $p= 3$, then according to \cite[Haupsatz]{Hi}, the quotient $\mathfrak{g/z_g}$ is a simple $G$-module. Hence the radical $\rad(\mathfrak{g/z_g})$ is trivial;
		 \item or $G$ is a $k$-group of type $G_2$ and $k$ is of characteristic $3$. According to \cite[table 1]{Ho} there are then only two possibilities for $\rad(\mathfrak{g})$: it is either trivial or the Lie algebra of a $\PGL_3$ factor. This last option cannot occur because the Lie algebra $\mathfrak{pgl}_3$ is not solvable, so one can conclude that $\rad(\mathfrak{g})=0$.
\end{enumerate}		
		
		\item Assume now that $G$ is semisimple. It then admits a universal covering, denoted by $G^{\Sc}$ (see for example \cite[1.1.2, Theorem 1, p.~43]{Tit}), and one can consider the following associated central extension:
\begin{figure}[H]
\begin{center}
\[\begin{tikzpicture} 

 \matrix (m) [matrix of math nodes,row sep=2em,column sep=4.8em,minimum width=2em,  text height = 1.5ex, text depth = 0.25ex]
  {	  
  	1 & \mu & G^{\Sc} & G & 1.\\
  };
  \path[-stealth]

  	(m-1-1) edge (m-1-2)
    (m-1-2) edge (m-1-3) 
    (m-1-3) edge node[above] {$\pi$} (m-1-4) 
    (m-1-4) edge (m-1-5) 
  
    ;  	
\end{tikzpicture}\]
\end{center}
\end{figure}		
\noindent Let $T^{\Sc}$ be a maximal $k$-torus of
$G^{\Sc}$ and set $T=T^{\Sc}/\mu$ (the corresponding Lie algebras will be denoted by $\mathfrak{t}^{\Sc}$, respectively $\mathfrak{t}$). The above lemma ensures that $\Lie(\pi)(\Lie(G^{\Sc}))$ is an ideal of $\mathfrak{g}$ and one has the following exact sequence of restricted $p$-Lie algebras:
\begin{figure}[H]
\begin{center}
\[\begin{tikzpicture} 

 \matrix (m) [matrix of math nodes,row sep=2em,column sep=3.8em,minimum width=2em,  text height = 1.5ex, text depth = 0.25ex]
  {	  
  	0 &  \Lie(\mu) & \Lie(G^{\Sc}) & \mathfrak{g}& \mathfrak{g}/\Lie(\pi)(\Lie(G^{sc})) & 0.\\
  	& & & & \mathfrak{t}/\Lie(\pi)(\mathfrak{t^{\Sc}})& \\
  };
  \path[-stealth]

  	(m-1-1) edge (m-1-2)
    (m-1-2) edge (m-1-3) 
    (m-1-3) edge node[above] {$\Lie(\pi)$} (m-1-4) 
    (m-1-4) edge (m-1-5) 
    (m-1-5) edge (m-1-6);
    
    (m-2-5) \node[xshift=10.3em,yshift = -0.25em]{\rotatebox{90}{$\cong$}}(m-1-5)

    ;  	
\end{tikzpicture}\]
\end{center}
\end{figure}
\noindent The extension being central, the preimage of $\rad(\mathfrak{g})$ is a solvable ideal of $\Lie(G^{\Sc})$ (this is a consequence of \cite[1.5, Theorem 5.1 (2)]{FS}). Hence it is contained in $\rad(\Lie(G^{\Sc})) = \mathfrak{z}_{\Lie(G^{\Sc})}$. Composing with $\Lie(\pi)$, one can then deduce that the inclusion 
\[\rad(\mathfrak{g}) \cap  \Lie(\pi)(\Lie(G^{\Sc} )) \subseteq  \mathfrak{z_g}\]
is satisfied, whence the desired equality
\[\rad(\mathfrak{g}) \cap  \Lie(\pi)( \Lie(  G^{\Sc} )) =  \mathfrak{z_g}  \cap  \Lie(\pi)( \Lie(G^{\Sc})).\]

The above exact sequence thus induces the following one:
\begin{figure}[H]
\begin{center}
\[\begin{tikzpicture} 

 \matrix (m) [matrix of math nodes,row sep=2em,column sep=4.8em,minimum width=2em,  text height = 1.5ex, text depth = 0.25ex]
  {	  
  	0 &  \mathfrak{z_g}\cap \Lie(\pi)(\Lie(G^{\Sc})) & \rad(\mathfrak{g}) & \mathfrak{h} & 0.\\
  };
  \path[-stealth]

  	(m-1-1) edge (m-1-2)
    (m-1-2) edge (m-1-3) 
    (m-1-3) edge node[above] {$\Lie(\pi)$} (m-1-4) 
    (m-1-4) edge (m-1-5) 
    ;  	
\end{tikzpicture}\]
\end{center}
\end{figure}
\noindent where $\mathfrak{h}$ is a restricted $p$-subalgebra of $\mathfrak{t}/ \Lie(\pi)( \mathfrak{t}^{\Sc})$, which is toral so has no $p$-nilpotent elements. In other words, the $p$-nilpotent elements of $\rad(\mathfrak{g})$ are trivial. Hence $\rad(\mathfrak{g})$ only has semisimple elements. According to \cite[Proposition 2.13]{BorT2}, it only remains to show the equality $N_{G}(\rad(\mathfrak{g})) =G$ to get the desired inclusion $\rad(\mathfrak{g}) \subseteq \mathfrak{z_g}$. Note also that all the other assumptions of the Proposition are trivially satisfied as $\rad(\mathfrak{g})$ is a proper ideal of $\mathfrak{g}$ (because $G$ is a reductive $k$-group). 

Let us thus show the equality $N_{G}(\rad(\mathfrak{g})) =G$. According to \cite[II,\S5, n\degree 3.2, Proposition]{DG} this can be shown on $\bar{k}$-points (as the group $G$ is smooth and of finite presentation and the Lie algebra $\rad(\mathfrak{g})$ is reduced and closed in $\mathfrak{g}$). This is clear as $\rad(\mathfrak{g})(\bar{k})$ is stable under conjugation: the image of $\rad(\mathfrak{g})(\bar{k})$ by $G(\bar{k})$-conjugation is a solvable ideal of $\mathfrak{g}(\bar{k})$, its maximality can be deduced by applying the inverse morphism.

		\item If $G$ is any reductive $k$-group, the following exact sequence allows to reduce ourselves to the preceding cases (see for example \cite[XXII Définition 4.3.6]{SGA33}):
		\begin{figure}[H]
\begin{center}
\[\begin{tikzpicture} 

 \matrix (m) [matrix of math nodes,row sep=2em,column sep=3.8em,minimum width=2em,  text height = 1.5ex, text depth = 0.25ex]
  {	  
  	1 &  (Z_G^0)_{\red}=\Rad(G) & G & G/(Z_G^0)_{\red} := G^{\Ss}& 1.\\
  };
  \path[-stealth]

  	(m-1-1) edge (m-1-2)
    (m-1-2) edge (m-1-3) 
    (m-1-3) edge node[above] {$\pi$} (m-1-4) 
    (m-1-4) edge (m-1-5) 
    ;  	
\end{tikzpicture}\]
\end{center}
\end{figure}
\noindent Indeed, as the subgroup $\Rad(G)$ is smooth, this exact sequence induces after derivation an exact sequence of Lie algebras (see \cite[\S5, n\degree 5, Proposition 5.3]{DG})
\begin{figure}[H]
\begin{center}
\[\begin{tikzpicture} 

 \matrix (m) [matrix of math nodes,row sep=2em,column sep=4.8em,minimum width=2em,  text height = 1.5ex, text depth = 0.25ex]
  {	  
  	0 &  \Lie(\Rad(G)) & \mathfrak{g} & \Lie(G^{\Ss}) & 0.\\
  };
  \path[-stealth]

  	(m-1-1) edge (m-1-2)
    (m-1-2) edge (m-1-3) 
    (m-1-3) edge node[above] {$\Lie(\pi)$} (m-1-4) 
    (m-1-4) edge (m-1-5) 
    ;  	
\end{tikzpicture}\]
\end{center}
\end{figure}
\noindent The morphism $\Lie(\pi)$ is surjective, its image $\Lie(\pi)(\rad(\mathfrak{g}))$ is therefore a solvable ideal of $\Lie(G^{\Ss})$. By what precedes it is then contained in the center of $\Lie(G^{\Ss})$. Let $x \in \rad(\mathfrak{g})$. As $k$ may be assumed to be algebraically closed, the element $x$ admits a Jordan decomposition, say $x = x_{s} + x_n$, with $x_{s}$ semisimple and $x_n$ a $p$-nilpotent element of $\rad(\mathfrak{g})$ (for the existence of such see for example \cite[2.3 Theorem 3.5]{FS}). As $\pi(x) \in \mathfrak{z_g}$ one necessarily has $\pi(x_n) = 0$, meaning that $x_n \in \Lie((Z^0_G)_{\red})$ which is toral. Hence $x_n = 0$. So $\rad(\mathfrak{g})$ only has semisimple elements. According to \cite[Proposition 2.13]{BorT2} we just have shown that $\rad(\mathfrak{g}) \subseteq \mathfrak{z_g}$ because $\rad(\mathfrak{g})$ is a proper $G$-sub-module of $\mathfrak{g}$. 
	\end{enumerate}
\end{proof}

\begin{remarks}
It is worth mentioning the following points:
\begin{enumerate}
\item Lemma \ref{radical_reductif_centre} in particular allows to measure the potential lack of smoothness of the center of $G$. More precisely one has:
\[\Lie(Z_G)/\Lie((Z_G)_{\red}) \cong \mathfrak{z_g}/\Lie{(Z_G)_{\red}}\cong \rad(\mathfrak{g})/\Lie(\Rad(G)), \]
\noindent where the first isomorphism comes from Remark \ref{centre_noyau_ad} ii). According to the proof of Lemma \ref{radical_reductif_centre} this quotient is a restricted toral $p$-algebra.

\item A careful study of the proof shows that the only difficulty one would have when trying to extend the above result to the characteristic $2$ framework relies on the fact that the $G$-module $\mathfrak{g/z_g}$ might not be simple. According to \cite[Haupsatz]{Hi} for an algebraically closed field $k$ of characteristic $2$ this is not an issue if the root system of $G$ only has irreducible components of $A_n$-type. This is always satisfied in this article.
\end{enumerate}
\label{remark_autorise_car2}
\end{remarks}

As mentioned in Remark \ref{rem_car_2} the following result allows to slightly refine the hypotheses on $p$ in the study of the nilradical of the Lie algebra of a reductive group.

\begin{lemma}[(Corollary of {\cite[Lemma 2.1]{VAS}})]
Let $G$ be a reductive $k$-group. If $k$ is of characteristic $2$ assume that $G^{\Ad}_{k^s}$ has no direct factor $G_1$ isomorphic to $\SO_{2n+1}$ for an integer $n>0$.
Under these assumptions $\nil(\mathfrak{g})$ is the center of $\mathfrak{g}$.
\label{nil_red}
\end{lemma}

\begin{proof}
One inclusion is clear and does not require any additional assumption on the characteristic of $k$: the center of $\mathfrak{g}$ is a nilpotent ideal of $\mathfrak{g}$ so it is contained in the nilradical of $\mathfrak{g}$. 

To show the reverse inclusion one only needs to prove that $\nil(\mathfrak{g})/\mathfrak{z_g} = 0$. The inclusion $\mathfrak{g/z_g} \subseteq \Lie(G^{\Ad})$ is provided by the exact sequence of Lie algebras of Remark \ref{centre_noyau_ad} ii):
\begin{figure}[H]
\begin{center}
\[\begin{tikzpicture} 

 \matrix (m) [matrix of math nodes,row sep=2em,column sep=4.8em,minimum width=2em,  text height = 1.5ex, text depth = 0.25ex]
  {
    0 & \mathfrak{z_g} = \Lie(Z_G) & \mathfrak{g} & \End(\mathfrak{g})\\
  };
  \path[-stealth]
  	(m-1-1) edge (m-1-2)
    (m-1-2) edge (m-1-3) 
    (m-1-3) edge node[above] {$\ad$} (m-1-4) 
    ;  	
\end{tikzpicture}\]
\end{center}
\end{figure}
\noindent Assume that $\nil(\mathfrak{g})/\mathfrak{z_g} \neq 0$. We show that this implies \cite[Lemma 2.1]{VAS} to hold true, leading to a contradiction (as it would imply $p = 2$ and $G$ to be such as excluded in the assumptions). 

We therefore have to check that:
\begin{enumerate}
	\item the quotient $\nil(\mathfrak{g})/\mathfrak{z_g}$ is a $G^{\Ad}$-sub-module of $\Lie(G^{\Ad})$, 
	\item for any maximal torus $T^{\Ad}\subseteq G^{\Ad}$ the intersection $\mathfrak{\nil(g)/z_g} \cap \Lie(T^{\Ad})$ is trivial.
\end{enumerate} 

To check that condition (i) is satisfied we first show that $\nil(\mathfrak{g/z_g}) = \nil(\mathfrak{g})/\mathfrak{z_g}$. The preimage of $\nil(\mathfrak{g/z_g})$ is a nilpotent ideal of $\mathfrak{g}$, as the considered extension of Lie algebras: 
\begin{figure}[H]
\begin{center}
\[\begin{tikzpicture} 

 \matrix (m) [matrix of math nodes,row sep=2em,column sep=4.8em,minimum width=2em,  text height = 1.5ex, text depth = 0.25ex]
  {
    0 & \mathfrak{z_g} = \Lie(Z_G) & \mathfrak{g} & \mathfrak{g/z_g} &0\\
  };
  \path[-stealth]
  	(m-1-1) edge (m-1-2)
    (m-1-2) edge (m-1-3) 
    (m-1-3) edge (m-1-4) 
    (m-1-4) edge (m-1-5) 
    ;  	
\end{tikzpicture}\]
\end{center}
\end{figure}
\noindent is central. It is thus contained in $\nil(\mathfrak{g}/\mathfrak{z_g})$ because the quotient $\nil(\mathfrak{g})/\mathfrak{z_g}$ is a nilpotent ideal of $\mathfrak{g/z_g}$. Hence we have shown the desired equality. So we are reduced to show that $\nil(\mathfrak{g}/\mathfrak{z_g})$ is a $G^{\Ad}$-sub-module of $\Lie(G^{\Ad})$, or in other words that $N_{G^{\Ad}}(\nil(\mathfrak{g/z_g})) =G^{\Ad}$. Once again as:
\begin{itemize}
	\item the group $G^{\Ad}$ is smooth of finite presentation, 
	\item and $\mathfrak{g/z_g}$ is reduced and closed in $\Lie(G^{\Ad})$, 
\end{itemize}
\noindent one only needs to check this equality on $\bar{k}$-points (see \cite[II,\S5, n\degree 3.2, Proposition]{DG}). Remark that the quotient $\nil(\mathfrak{g/z_g}(\bar{k}))$ is stable for the adjoint action as the image of $\nil(\mathfrak{g/z_g})(\bar{k})$ under the $G^{\Ad}(\bar{k})$-conjugation is a nilpotent ideal of $\mathfrak{g/z_g}(\bar{k})$. Its maximality follows by considering the reverse morphism. Thus we have shown the equality.

To check that condition (ii) is indeed satisfied, first notice that any maximal torus $T^{\Ad} \subset G^{\Ad}$ comes from a maximal torus $T\subset G$. At the Lie algebras level one can summarize the situation with the following commutative diagram:
\begin{figure}[H]
\begin{center}
\begin{tikzpicture}
    \matrix(m)[matrix of math nodes, row sep=2em, column sep=2em, text height=1.5ex, text depth=0.25ex]
    { & \nil(\mathfrak{g})/\mathfrak{z_g} &  \\
    0 & \mathfrak{g/z_g} & \Lie(G^{\Ad}), \\
    0 & \Lie(T)/\mathfrak{z_g} & \Lie(T^{\Ad}). \\
  };
 (m-1-2) \node[xshift=-1.8em,yshift = 1.8em]{\rotatebox{-90}{$\subseteq$}}(m-2-2) ;
 (m-3-2) \node[xshift=-1.8em,yshift = -1.8em]{\rotatebox{90}{$\subseteq$}}(m-2-2)
 ;
  (m-3-3) \node[xshift=4.4em,yshift = -1.8em]{\rotatebox{90}{$\subseteq$}}(m-2-3)
 ; 
    \path[->,font=\scriptsize]

    (m-2-1) edge (m-2-2)
    (m-2-2) edge (m-2-3)
    (m-3-1) edge (m-3-2)
    (m-3-2) edge (m-3-3)

    ;
\end{tikzpicture}
\end{center}
\end{figure} 
\noindent Assume that the intersection $\mathfrak{\nil(g)/z_g} \cap \Lie(T^{\Ad})$ is not trivial. This in particular implies that neither is the intersection $\mathfrak{\nil(g)/z_g} \cap \Lie(T)/\mathfrak{z_g}$ as any element of the first intersection occurs as an element of the image of $\mathfrak{g} \rightarrow \mathfrak{g/z_g}$. Remember that we have already shown that $\mathfrak{z_g}$ is contained in $\nil(\mathfrak{g})$. According to Remark \ref{centre_noyau_ad} ii), it is nothing but the Lie algebra of $Z_G$, whence the inclusion $\mathfrak{z_g} \subseteq \Lie(T)$. 

The non-triviality of the intersection $\mathfrak{\nil(g)/z_g} \cap \Lie(T)/\mathfrak{z_g}$ is therefore equivalent to suppose that the inclusion $\mathfrak{z_g} \subsetneq \Lie(T)\cap \nil(\mathfrak{g})$ is strict. This leads to a contradiction. Indeed any element of the nilradical is $\ad$-nilpotent (according to the second point of the preamble of this section) and any $\ad$-nilpotent element of the Lie algebra of a torus is central. This can be shown as follows: let $n$ be the order of $\ad$-nilpotency of $x \in \Lie(T)$ and $y \in \mathfrak{g}$. Passing to the algebraic closure of $k$ if necessary, the Lie algebra $\mathfrak{g}$ has a weight space decomposition for the action of the maximal torus $T$ (which is locally splittable). Let $R$ be an associated root system. One has:  $\mathfrak{g} = \mathfrak{t} \oplus \bigoplus_{\alpha \in R} \mathfrak{g}^{\alpha}$. 
Thus $y$ writes $y = y_0 + \sum_{\alpha \in R} y^{\alpha}$ for $y^0 \in \mathfrak{t}$ and $y^{\alpha} \in \mathfrak{g}^{\alpha}$, with $\alpha \in R$. This leads to:
 \begin{alignat*}{3}
0 = \ad^n(x)(y) \: & = \ad^n(x)(y^0 + \sum_{\alpha \in R} y^{\alpha}) \: & = \sum_{\alpha \in R} \alpha^n(x)y^{\alpha},\\
 \end{alignat*}
where we have made use of the vanishing condition $\ad(x)(y^0)=0$ as $x \in \mathfrak{t}$. This equality being satisfied if and only if $\ad(x)(y^{\alpha}) = 0$ for any $\alpha \in R$, this implies that $x \in \mathfrak{z_g}$.

\end{proof}

\begin{remark}
In this article we always require that $p$ is not of torsion for $G$. This in particular implies that $p$ is strictly greater than $2$ if $G$ has any factor of $B_n$ type. The above lemma then tells us that the equality $\nil(\mathfrak{g}) = \mathfrak{z_g}$ is always satisfied here.
\end{remark}

\begin{lemma}
Let $U$ be a unipotent algebraic $k$-group then its Lie algebra $\mathfrak{u}$ is $p$-nil. In particular, the Lie algebra of the unipotent radical of a smooth connected $k$-group $G$ is a restricted $p$-nil $p$-ideal of $\mathfrak{g}$.
\label{rad_unip_p_nil}
\end{lemma}

\begin{proof}
As $k$ is a field, it follows from \cite[IV, \S2, n\degree 2 Proposition 2.5 vi)]{DG} that the unipotent $k$-group $U$ is embeddable into the subgroup $U_{n,k}$ of upper triangular matrices of $\GL_n$ for a certain $n \in \mathbb{N}$. This leads to the following inclusion of restricted $p$-Lie algebras (all of them coming from algebraic $k$-groups) $\mathfrak{u} \subseteq \mathfrak{u}_{n,k}$. Note that the $p$-structure on $\mathfrak{u}_{n,k}$ is given by taking the $p$-power of matrices. This makes $\mathfrak{u}_{n,k}$ into a restricted $p$-nil $p$-subalgebra, so is $\mathfrak{u}$. 

If now $U$ is the unipotent radical of a smooth connected $k$-group $G$ its Lie algebra is an ideal of $\mathfrak{g}$ because it is the Lie algebra of a normal subgroup of $G$. As it derives from an algebraic $k$-subgroup of $G$ it is endowed with a $p$-structure compatible with the $p$-structure of $G$. Hence it is a restricted $p$-ideal of $\mathfrak{g}$. It is $p$-nil by what precedes. 
\end{proof}

\begin{lemma}
Let $k$ be a perfect field and $H$ be a smooth connected algebraic $k$-group. Then:
\begin{enumerate}
\item if the reductive $k$-group $\underline{H} := H/\Rad_U(H)$ satisfies the conditions of Lemma \ref{nil_red} the Lie algebra of the unipotent radical of $H$ is the $p$-radical of $\mathfrak{h}$. In other words the equality $\mathfrak{\radu}(H) = \rad_p(\mathfrak{h})$ holds true, 
\item  if $k$ is of characteristic $p\geq 3$ the $p$-radical of $\mathfrak{h}$ is the set of $p$-nilpotent elements of $\rad(\mathfrak{h})$.
\end{enumerate}
\label{p_rad_gpe_lisse_connexe}
\end{lemma}

\begin{remark}
In particular let $k$ be a perfect field. Consider a reductive $k$-group $G$ and a parabolic subgroup $P \subseteq G$. If $P$ is such that the reductive quotient $P/\Rad_U(P)$ it defines satisfies the assumptions of Lemma \ref{nil_red} then: 
\begin{itemize}
\item the Lie algebra of its unipotent radical is the $p$-radical of $\mathfrak{p} := \Lie(P)$,
\item it is the set of all $p$-nilpotent elements of $\rad(\mathfrak{p})$.
\end{itemize}
\noindent As a reminder (see \cite[XXVI, Proposition 1.21 ii)]{SGA33}), if $L \subseteq P$ is a Levi subgroup, the solvable radical $\Rad(P)$ is the semi-direct product of the unipotent radical of $P$ with the radical of $L$. This in particular implies that $\Lie(Z_L^0)\subseteq \rad(\mathfrak{p})$ for $Z_L^0$ being the center of $L$.
\label{p-rad_parab}
\end{remark}

\begin{proof}
We show each point separately.
\begin{enumerate}
	\item We start by showing (i). An implication is clear: according to Lemma \ref{rad_unip_p_nil} the Lie algebra $\mathfrak{\radu}(H)$ is a restricted $p$-nil $p$-ideal. In particular the inclusion
$\mathfrak{\radu}(H)\subseteq \rad_p(\mathfrak{h})$ holds true. 

Let us show the reverse inclusion. The radical of $H$ being a smooth subgroup the following exact sequence of algebraic $k$-groups:
\begin{figure}[H]
\begin{center}
\[\begin{tikzpicture} 

 \matrix (m) [matrix of math nodes,row sep=2em,column sep=4.8em,minimum width=2em,  text height = 1.5ex, text depth = 0.25ex]
  {
    1 & \Rad_u(H) & H & H/\Rad_U(H)=: \underline{H} & 1,\\
  };
  \path[-stealth]
  	(m-1-1) edge (m-1-2)
    (m-1-2) edge (m-1-3) 
    (m-1-3) edge node[above] {$\pi$} (m-1-4) 
    (m-1-4) edge (m-1-5)
    
    ;
    	
\end{tikzpicture}\]
\end{center}
\end{figure}
\noindent induces an exact sequence of $k$-Lie algebras (see \cite[II, \S5, n\degree 5 Proposition 5.3]{DG}):
\begin{figure}[H]
\begin{center}
\[\begin{tikzpicture} 

 \matrix (m) [matrix of math nodes,row sep=2em,column sep=4.8em,minimum width=2em,  text height = 1.5ex, text depth = 0.25ex]
  {
    0 & \mathfrak{\radu}(H) & \mathfrak{h} & \mathfrak{h/\radu}(H)=:\underline{\mathfrak{h}} & 0. \\};
  \path[-stealth]
  	(m-1-1) edge (m-1-2)
    (m-1-2) edge (m-1-3) 
    (m-1-3) edge node[above] {$\Lie(\pi)$} (m-1-4) 
    (m-1-4) edge (m-1-5)
     ;
    	
\end{tikzpicture}\]
\end{center}
\end{figure}

\noindent This is an exact sequence of restricted $p$-Lie algebras (see \cite[II, \S7 n\degree 2.1 and n\degree 3.4]{DG} for the compatibility with the $p$-structure). The derived morphism $\Lie(\pi)$ being surjective the image of $\nil(\mathfrak{h})$ under $\Lie(\pi)$ is still an ideal. It is nilpotent as $\Lie(\pi)$ is a morphism of restricted $p$-Lie algebras, whence the inclusion $\Lie(\pi)(\nil(\mathfrak{h})) \subseteq \nil(\underline{\mathfrak{h}})$. As $\underline{\mathfrak{h}}$ derives from a reductive $k$-group which does not fit into the pathological case raised by A. Vasiu in \cite[Lemma 2.1]{VAS}, Lemma \ref{nil_red} applies. This leads to the equality $ \nil(\underline{\mathfrak{h}}) = \mathfrak{z}_{\underline{\mathfrak{h}}}$, thus one has $\Lie(\pi)(\nil(\mathfrak{h})) = \mathfrak{z}_{\underline{\mathfrak{h}}}$.

According to Lemma \ref{inclusion_p-struct} (i) one has $\rad_p(\mathfrak{h}) \subseteq \nil(\mathfrak{h})$, hence any $x \in \rad_p(\mathfrak{h})$ is mapped to the center of $\mathfrak{z}_{\underline{\mathfrak{h}}}$. The restricted $p$-ideal $\rad_p(\mathfrak{h})$ being $p$-nil, the element $x$ is $p$-nilpotent, so is $\Lie(\pi)(x)$ (as $\Lie(\pi)$ is compatible with the $p$-structures on $\mathfrak{h}$ and $\underline{\mathfrak{h}}$). In other words $\Lie(\pi)(x)$ is a $p$-nilpotent element of $\mathfrak{z}_{\underline{\mathfrak{h}}} $, which is, according to Remark \ref{centre_noyau_ad} ii), the Lie algebra of the center of the reductive $k$-group $\underline{H}$. This center is thus a toral restricted $p$-subalgebra (see for example \cite[XXII, Corollaire 4.1.7]{SGA33}) hence the equality $\Lie(\pi)(x)= 0$. In other words we just have shown that $x \in \radu(H)$, whence the equality $\rad_p(\mathfrak{h}) = \radu(H)$. This concludes the proof of (i). 

\item Let us then prove the second point of the statement. Once again an inclusion is clear: the $p$-radical $\rad_p(\mathfrak{h})$ is a restricted $p$-nil $p$-ideal of $\mathfrak{h}$ and is therefore contained in the set of all $p$-nilpotent elements of $\mathfrak{h}$.
Let us show the converse inclusion: let $x \in \rad(\mathfrak{h})$. The morphism $\Lie(\pi)$ being surjective, the image $\Lie(\pi(x))$ belongs to $\rad(\underline{\mathfrak{h}})$, which is the center of $\underline{\mathfrak{h}}$ according to Lemma \ref{radical_reductif_centre} (which holds true as $p\geq 3$). It necessarily vanishes because the center of $\underline{\mathfrak{h}}$ is toral and $\Lie(\pi)(x)$ is also a $p$-nilpotent element. In other words, we have shown that $x\in \radu(\mathfrak{h})$ which is the $p$-radical of $\mathfrak{h}$ according to the first point of the lemma. Hence any $p$-nilpotent of $\rad(\mathfrak{h})$ belongs to $\rad_p(\mathfrak{h})$, whence the desired equality.
\end{enumerate}
\end{proof}

\section{Springer isomorphisms and $\fppf$-formalism}
\label{candidat}
	\subsection{Integrating $p$-nilpotent elements: a starting point}
	
In what follows, the field $k$ is algebraically closed of characteristic $p>0$ and $G$ is a reductive $k$-group. The notations used here are those of \cite[II, Définition 4.6.1]{SGA31}: let $S$ be an affine scheme with ring of coordinates $\mathcal{O}_S$. For any $\mathcal{O}_S$-module $F$, denote by $W(F)$ the following contravariant functor over the category of $S$-schemes:
\[W(F)(S'):= \Gamma(S', F\otimes_{\mathcal{O}_S} \mathcal{O}_{S'}),\]
\noindent where $\Gamma$ identifies with the set of $F\otimes_{\mathcal{O}_S} \mathcal{O}_{S'}$-sections over $S'$. Moreover let $H$ be a $k$-algebraic group, and denote by $\mathfrak{h}$ its Lie algebra. By \cite[II, Lemme 4.11.7]{SGA31} the equality $\mathfrak{h} = W(\mathfrak{h})$ is satisfied. In particular $\mathfrak{h}$ is smooth and connected. 

\subsubsection{Reduced part of the nilpotent and unipotent schemes}\phantom{ooo}

\label{N_red}
The reductive group $G$ acts on $\mathfrak{g}$ via the adjoint action. Let us denote by $\mathcal{O}_{\mathfrak{g}}$ the coordinate ring of $\mathfrak{g}$, and by $\mathcal{O}_{\mathfrak{g}}^G$ the fixed points under the induced action. When $k$ is a field the affine quotient $[\mathfrak{g}/G]:= \Spec(\mathcal{O}_{\mathfrak{g}}^G)$ is universal and the nilpotent scheme $\mathcal{N}(\mathfrak{g})$ is the fibre $\pi^{-1}\pi(0)$, where $\pi: \mathfrak{g}\rightarrow [\mathfrak{g}/G]$ is the quotient morphism and $0 \in \mathfrak{g}(k)$ is the zero section. This is explained in details in a more general framework in \cite[(2.4), (2.5) and (2.6)]{HESS1}. The reduced part of the nilpotent scheme, denoted by $\mathcal{N}_{\red}(\mathfrak{g})$, coincides with the reduced subscheme of $\mathfrak{g}$ whose $k$-points are the $p$-nilpotent elements of $\mathfrak{g}$ (see for instance \cite[9.2.1]{BarRic}).

Similarly $G$ acts on itself via the adjoint action. When $k$ is a field the affine quotient $[G/\Ad(G)]$ is universal and the unipotent scheme $\mathcal{V}(G)$ is the fibre $\pi^{-1}\pi(e)$, where $\pi: G\rightarrow [G/\Ad(G)]$ is the quotient morphism and $e \in G(k)$ is the neutral element. The reduced part of the unipotent scheme, denoted by $\mathcal{V}_{\red}(G)$, coincides with the reduced subscheme of $G$ whose $k$-points are the unipotent elements of $G$ (see for instance \cite[9.1]{BarRic}).

The literature on Springer isomorphisms mainly considers the so-called nilpotent and unipotent varieties (under the convention that varieties are reduced). When the terminology of schemes is adopted, as in the article of V. Balaji, P. Deligne and A. J. Parameswaran \cite{BDP}, the authors insist on the necessity of considering the reduced part of both nilpotent and unipotent schemes (as the proof of existence of Springer isomorphisms is constructive and based on a reasoning on points). One might then wonder whether under the framework of this article these schemes are reduced or not. For the nilpotent scheme this is given by \ref{nilpotent_variety_reduced} which is a corollary of S. Riche's work \cite[Lemma 3.3.3]{Riche}. The following notions will be necessary in what follows:
\begin{itemize}
\item an element $x \in \mathfrak{g}$ is regular if its centraliser $C_G(x)$ is of minimal dimension. This lower bound exists, equals the rank of the reductive group $G$ and is attained. This is detailed for instance in section 2.3 of S. Riche's paper. We denote by $\mathfrak{g}_{\reg}$ the subset of $\mathfrak{g}$ consisting of all regular nilpotent elements;
\item the intersection $\mathcal{N}(\mathfrak{g})\cap \mathfrak{g}_{\reg} \in W(\mathfrak{g})$ is denoted $\mathcal{N}(\mathfrak{g}_{\reg})$ and is the scheme whose points are regular nilpotent elements of $\mathfrak{g}$.
\end{itemize}
 
\begin{proposition}[(corollary of {\cite[Lemma 3.3.3]{Riche}})]
Let $G$ be a connected reductive group defined over a field $k$ of characteristic $p>0$ which is assumed to be pretty good for $G$. Then the nilpotent scheme $\mathcal{N}(\mathfrak{g})$ is reduced.
\label{nilpotent_variety_reduced}
\end{proposition}

\begin{remark}
Pretty good characteristics are those that satisfy S. Riche's (C3) condition in \cite[2.2, p.~227]{Riche}. Integration and related questions often come with the assumption that the reductive group $G$ is standard (see for instance \cite[2.9]{Jantzen2004}). Namely this means that: 
\begin{enumerate}
\item the derived group $G$ is simply connected,
\item the characteristic is good for $G$,
\item there exists a 
$G$-equivariant nondegenerate bilinear form on $\mathfrak{g}$. 
\end{enumerate}
As underlined in \cite[2.2]{Riche}, this is a stronger condition than the pretty good characteristic assumption. So in particular Proposition \ref{nilpotent_variety_reduced} ensures that the nilpotent scheme of a reductive group that satisfies the standard hypotheses is reduced.
\label{standard_hyp}
\end{remark}

\begin{proof}
Let us first remind the reader that the nilpotent scheme is irreducible, as already shown for instance in \cite[Lemma 6.2]{Jantzen2004}. Note that in the article the proof is made for the reduced part of the nilpotent scheme (by definition of the nilpotent variety); this is not an issue here, as being irreducible is a property of the underlying topological space. 

The subset of regular elements of $\mathfrak{g}$ is open (this is shown for instance in \cite[1.4 Corollary]{HUM2}) and non-empty (\cite[Lemma 3.3.1]{Riche}), therefore $\mathcal{N}(\mathfrak{g}_{\reg})$ is also open (and non-empty) in $\mathcal{N}(\mathfrak{g})$. The situation is that of the following diagram:
\begin{figure}[H]
\begin{center}
\[
\begin{tikzcd}[sep=large]
\mathcal{N}(\mathfrak{g}_{\reg})   \arrow[r] \arrow[d] &  W(\mathfrak{g}_{\reg}) \arrow[d, "\chi_{\reg}"],  \\
\Spec(k) \arrow[r] &  \left(W(\mathfrak{g})//G\right) \cong \mathbb{A}^n_k, 
\end{tikzcd}
\]
\end{center}
\end{figure}
\noindent which is cartesian and where the morphism $\chi_{\reg}$ is smooth according to \cite[Lemma 3.3.3]{Riche} (as we have assumed the characteristic to be pretty good for $G$). Thus the morphism $\mathcal{N}(\mathfrak{g}_{\reg}) \rightarrow \Spec(k)$ is smooth (see for instance \cite[I, \S 4, n\degree 4.1]{DG}). As any smooth scheme over a field is geometrically reduced (see \cite[\href{https://stacks.math.columbia.edu/tag/056T}{Tag 056T}]{stacks-project}), what precedes tells us that $\mathcal{N}(\mathfrak{g}_{\reg})$ is indeed reduced. As this is an open subset of $\mathcal{N}(\mathfrak{g})$ it is dense in $\mathcal{N}(\mathfrak{g})$ (the latter being irreducible) and its scheme theoretic closure is the nilpotent scheme itself. Hence $\mathcal{N}(\mathfrak{g}_{\reg})$ is a reduced open subset which is scheme theoretically dense in $\mathcal{N}(\mathfrak{g})$ so the nilpotent scheme itself is reduced (see \cite[\href{https://stacks.math.columbia.edu/tag/056E}{Tag 056E}]{stacks-project}).
\end{proof}

\begin{remark}
To conclude that the nilpotent scheme is reduced in separably good characteristics, it remains to investigate what happens for $\SL_{mp}$ in characteristic $p>0$, with $m\in \mathbb{N}^{*}$. This case is actually governed by the case of $\GL_{mp}$. Indeed, a matrix $M \in \gl_{mp}$ is nilpotent if and only if its characteristic polynomial is of the form $t^{mp}$, thus if and only if $\tr(\bigwedge^i M) = 0$ for $0\leq i \leq mp$ where $\bigwedge^i M$ is the $i^{\text{th}}$-exterior power of $M$. In particular they have trace $0$, whence the isomorphism of coordinate rings 
\[\mathcal{O}_{\mathcal{N}(\gl_{mp})} = \mathcal{O}_{\gl_{mp}}/\langle \tr(\bigwedge^i M)\rangle_{i=1, \cdots, mp} \cong \mathcal{O}_{\ssl_{mp}}/\langle \tr(\bigwedge^i M)\rangle_{i=1,\cdots mp}=\mathcal{O}_{\mathcal{N}(\ssl_{mp})}.\] 
As $\mathcal{N}(\gl_{mp})$ is reduced in characteristic $p$ (according to \ref{nilpotent_variety_reduced} as $p$ is pretty good for $\GL_{mp}$) so is its coordinate ring $\mathcal{O}_{\mathcal{N}(\gl_{mp})}$, hence the reducedness of $\mathcal{N}(\ssl_{mp})$. Therefore $\mathcal{N}(\mathfrak{g})$ is reduced when the characteristic is separably good for the reductive group $G$.
\end{remark}

\begin{remark}
In non separably good characteristics both unipotent and nilpotent schemes might be non-reduced as underlined for instance in \cite[3.9, Remark]{SLO} in which the author studies the unipotent scheme of $\PGL_2$ in characteristic $2$. From this, one can derive the same counter-example for the nilpotent scheme of $\pgl_2$:
\[\mathcal{N} (\pgl_2) = \left\lbrace \begin{pmatrix}
x_1 & 0 & 0 \\
x_2 & 0 & x_3\\
0 & 0 & x_1
\end{pmatrix}\right\rbrace \in \gl_3,\]
and a matrix $M$ of this Lie algebra (as described above) is nilpotent if and only if its characteristic polynomial is $t^3$. Hence the nilpotent scheme of $\pgl_2$ has coordinate ring $k[x_1, x_2, x_3]/\langle (x_1^2), (2x_1)\rangle$, which is not reduced in characteristic $2$.
\end{remark}

\subsubsection{Springer isomorphisms}\phantom{ooo}\\

\label{Springer_iso}
As explained in the introduction, the existence of a $G$-equivariant isomorphism $\phi:\mathcal{N}_{\red}(\mathfrak{g}) \rightarrow \mathcal{V}_{\red}(G)$ is necessary to obtain a punctual integration. When $G$ is a simply connected $k$-group and $p$ is good for $G$, T. A. Springer establishes in \cite[Theorem 3.1]{SPR} the existence of homeomorphisms between these schemes. As pointed out by the author himself, these homeomorphisms would be isomorphisms of varieties (with the convention of the article, hence of reduced schemes) if the reduced part of the nilpotent scheme were known to be normal (which had not been shown at the time of the paper). This result has been studied and refined by many mathematicians among those P. Bardsley and R. W. Richardson in \cite[9.3.2]{BarRic} who established the normality of $\mathcal{N}(\mathfrak{g})$ in pretty good characteristic (under this assumption the nilpotent scheme is reduced according to Proposition \ref{nilpotent_variety_reduced}) and extended the existence of such isomorphisms to any reductive $k$-group which satisfies the standard hypotheses (as defined by J. C. Jantzen, see \ref{standard_hyp}). Let us also mention here the work of S. Herpel who shows in \cite{HER} the existence of Springer isomorphisms for any reductive $k$-groups in pretty good characteristic. This mainly uses previous results from G. McNinch and D. Testerman (see \cite[Theorem 3.3]{MT2}). 

From now on, and unless otherwise stated, the characteristic of $k$ is separably good for $G$. In particular nilpotent and unipotent schemes are reduced, hence the subscripts are removed everywhere Springer isomorphisms are considered.

Let us insist on the following point: there exist several Springer isomorphisms but they all induce the same bijection between the $G$-orbits of $\mathcal{N}_{\red}(\mathfrak{g})$ and those of $\mathcal{V}_{\red}(G)$, as shown by J-P. Serre in \cite[10, Appendix]{MCNOPT}. To fix better the reader's idea on such variety of Springer isomorphisms one might have a look at the preamble of the aforementioned appendix. There J.-P. Serre considers the example $G = \SL_n$ and picks a nilpotent element $e \in \ssl_n$ of order $n$. He then explains that in this case a Springer isomorphism $\phi$ is of the form:
\begin{alignat*}{5}
\phi : \: & \mathcal{N}_{\red}(\mathfrak{g}) \: &  \rightarrow \: & \mathcal{V}_{\red}(G) \: & \\
	   \: & 1+e \:& \mapsto \: & a_1e + \: &\cdots +a_{n-1}e^{n-1},\\
\end{alignat*}
where the $a_i$'s are elements of $k$ such that $a_1 \neq 0$. Moreover, any $n$-tuple $(a_1, \hdots, a_{n-1})$ with $a_1\neq 0$ defines a unique Springer isomorphism.

When $G$ is simple, P. Sobaje reminds that Springer isomorphisms exist if $k$ is of separably good characteristic for $G$ (see \cite[Theorem 1.1 and Remark 2]{SOB1}). Moreover, in \cite[\S 7]{SOB2} the author investigates the non separably good characteristic case. He also emphasises that in separably good characteristics one can always find an isomorphism $\phi : \mathcal{N}_{\red}(\mathfrak{g}) \rightarrow \mathcal{V}_{\red}(G)$ that restricts to an isomorphism of reduced schemes $W(\radu(B)) \cong \radu(B) \rightarrow \Rad_U(B)$ for any Borel subgroup $B\subset G$. The author then stresses that the differential of this restriction at $0$ is a scalar multiple of the identity (this does not depend on the considered Borel subgroup). More precisely, the situation is the following (note that the two vertical arrows are closed immersions):

\begin{figure}[H]
\begin{center}
\[
\begin{tikzcd}[sep=large]
\mathcal{N}_{\red}(\mathfrak{g})   \arrow[r, "\phi"] &  \mathcal{V}_{\red}(G),  \\
W(\radu(B)) \arrow[u, hook, "/" marking] \arrow[r, "\tilde{\phi}" = k\cdot \id"]& \Rad_U(B),\arrow[u, hook, "/" marking] & 
\end{tikzcd}
\]
\end{center}
\end{figure}

\noindent and $\tilde{\phi}$ is such that $(\diff\tilde{\phi})_0 = \lambda \id$ with $\lambda \in k^{*}$. Note that the fact that the restriction $\tilde{\phi}$ maps to $\Lie(\Rad_U(B))$ is actually a consequence of the $\phi$-infinitesimal saturation of the unipotent radical of Borel subgroups (see Definition \ref{def_phi_sat} for a definition of this notion as well as Remark \ref{parab_phi_inf_sat} for a proof of this fact). P. Sobaje attributes the existence of such specific Springer isomorphisms to G. McNinch and D. Testerman (see \cite[Theorem E]{MT2}). This can be generalised to any reductive $k$-group in separably good characteristic. Indeed the properties required for $\phi$ are preserved under separable isogenies for $G$ (see \cite[Corollary 5.5]{HER} and \cite[Proposition 9]{MCNOPT}). This allows us to consider semisimple groups rather than simple ones. The reductive case follows because ``Springer isomorphisms are insensitive to the center''. Indeed the radical of $\mathfrak{g}$ is the Lie algebra of the center of $G$, thus is toral and does not contain any $p$-nilpotent elements (see Lemma \ref{radical_reductif_centre} and Remark \ref{centre_noyau_ad} ii)).

Let $P$ be a so-called restricted parabolic subgroup, that is, a parabolic subgroup for which the Lie algebra of the unipotent radical has $p$-nilpotency order equals to $1$. For instance, when $p>\h(G)$ (where $\h(G)$ is the Coxeter number of $G$) any Borel subgroup of $G$ is restricted (because then, as explained in the introduction, any $p$-nilpotent element of $G$ has $p$-nilpotent order equals to $1$, see \cite{MAUS} for more details). We denote:
\begin{itemize}
	\item by $\radu(P)$ the Lie algebra of the unipotent radical of $P$, 
	\item and by $\Rad_U(P)$ the unipotent radical of $P$.
\end{itemize}  
In \cite[Proposition 5.3]{Sei} (credited by G. M. Seiz to J-P. Serre) the author explains how to obtain a $P$-equivariant isomorphism of algabraic groups $\exp_P: \mathfrak{rad_u(p)} \rightarrow \Rad_U(P)$ by base-changing the usual exponential map in characteristic $0$. Note that here $\mathfrak{rad_u(p)}$ is endowed with the group law induced by the Baker--Campbell--Hausdorff law (which is well defined, see the preamble of \cite[Section 5]{Sei}). When $p>\h(G)$, P. Sobaje explains in \cite[Theorem 6.0.2]{SOB2} that there exists a unique Springer isomorphism $\phi$ that restricts to $\exp_P$, whose tangent map is the identity and that is compatible with the $p$-power. Note that when $p<\h(G)$, maps satisfying the three aforementioned requirements still exist, but this time, there are many of such. P. Sobaje study and classify in \cite[Theorem 6.0.2]{SOB2} a specific class of such maps, the so-called generalised exponential maps.

Let us first consider the case $p>\h(G)$ to remind the reader of the construction of such group isomorphisms $\exp_B$ where $B \subset G$ is a Borel subgroup, and how this leads to integration results. In this settings, the reader might also be referred to a recent article of V. Balaji, P. Deligne and A. J. Parameswaran (see \cite[\S 6]{BDP}) for a detailed construction of the group isomorphism $\exp_B : \mathfrak{u_b}\rightarrow U_B$. This, combined with Corollary \ref{corollaire_LMT}, actually allows to integrate restricted $p$-nil $p$-Lie subalgebras of $\mathfrak{g}$:

\begin{proposition}
Let $G$ be a reductive $k$-group over an algebraically closed field of characteristic  $p> \h(G)$. Let $\mathfrak{u} \subset \mathfrak{g}$ be a restricted $p$-nil $p$-Lie subalgebra. Then $\mathfrak{u}$ can be integrated into a smooth connected unipotent subgroup of $G$. Namely there exists a smooth connected unipotent subgroup $U \subset G$ such that $\Lie(U) \cong \mathfrak{u}$ as Lie algebras. 
\label{exp_car_p}
\end{proposition}

Proposition \ref{système_de_coordonnées} will be useful to show the above statement. In the aforementioned framework and as underlined by J-P. Serre in \cite[2.2]{SER2}, if $B\subset G$ is a Borel subgroup the group law on $\mathfrak{\radu}(B)$ comes from the characteristic zero framework by lifting and specialisation.
	
	More precisely, if we denote:
	\begin{enumerate}
	\item by $G_{\mathbb{Z}}$ a reductive $\mathbb{Z}$-group and $B_{\mathbb{Z}} \subset G_{\mathbb{Z}}$ a Borel subgroup such that $G=G_{\mathbb{Z}}\otimes_{\mathbb{Z}} k$ and $B= B_{\mathbb{Z}}\otimes_{\mathbb{Z}} k$ (such groups both exist according to \cite[XXV, Corollaire 1.3]{SGA33}),
	\item by $G_\mathbb{Q}$ and $B_{\mathbb{Q}}$ the groups obtained from $G_{\mathbb{Z}}$ and $B_{\mathbb{Z}}$ by base change from $\mathbb{Z}$ to $\mathbb{Q}$,
\end{enumerate}	  
\noindent then:

	\begin{proposition}[({\cite[2.2]{SER2}})]
	 The law making $\radu(B)$ into an algebraic $k$-group comes from the one on $\radu(B)_{\mathbb{Q}}$: it is defined over $\mathbb{Q}$, extends on $\radu(B)_{\mathbb{Z}_{(p)}}$ and induces a group law on $\radu(B)_{\mathbb{F}_p}$ then on $\radu(B)$ by specialisation. Namely, the situation can be read on the following diagram, the point being the existence of the dotted arrow. In other words the Baker--Campbell--Hausdorff law has $\mathbb{Z}_{(p)}$-integral coefficients:
	 
\begin{figure}[H]
\begin{center}
\[
\begin{tikzcd}[sep=small]
(B , G ) \ar[d]& & & (\radu(B), \circ) \ar[d]  \\
(B_{\mathbb{F}_p} ,G_{\mathbb{F}_p} )\ar[r]\ar[dr,gray] & (B_{\mathbb{Z}_{(p)}}, G_{\mathbb{Z}_{(p)}}) \ar[d,gray] &( B_{\mathbb{Q}}, G_{\mathbb{Q}})\ar[l]\ar[dl,gray] & (\radu(B)_{\mathbb{F}_p}, \circ) \ar[dr,gray]\ar[r] &(\radu(B)_{\mathbb{Z}_{(p)}}, \circ) \ar[d,gray] & (\radu(B)_{\mathbb{Q}}, \circ). \ar[l, dashed] \ar[dl,gray] \\
& \textcolor{gray}{(B_{\mathbb{Z}} , G_{\mathbb{Z}})}& & & \textcolor{gray}{\radu(B)_{\mathbb{Z}}} & 
\end{tikzcd}
\]
\end{center}
\end{figure}
	\label{système_de_coordonnées}
	\end{proposition}

We refer the reader to \cite[Annexe D]{JEAthese} for a proof of Proposition \ref{système_de_coordonnées} stated in these terms. We are now able to show Proposition \ref{exp_car_p}:

\begin{proof}[Proof of Proposition \ref{exp_car_p}]

According to \cite[II, Lecture 2, Theorem 3]{SER1}, when $p>\h(G)$ the Lie algebra of the unipotent radical of any Borel subgroup $B \subset G$ is endowed with a group structure induced by the Baker--Campbell--Hausdorff law (which has $p$-integral coefficients as shown for example in \cite[2.2, Propositon 1]{SER2}). This law being defined with iterated Lie brackets, it reduces to any subalgebra of $\radu(B)$, endowing it with a group structure. As by assumption one has $p>\h(G)$, the characteristic of $k$ is not of torsion for $G$. Thus there exists a Borel subgroup $B \subset G$ such that $\mathfrak{u}$ is a Lie subalgebra of $\radu(B)$ (according to Corollary \ref{corollaire_LMT}). Hence what precedes in particular implies that $\mathfrak{u}$ is an algebraic group for the Baker--Campbell--Hausdorff law.

The isomorphism of groups $\exp_{\mathfrak{b}} : \radu(B) \rightarrow \Rad_U(B)$ defined by J-P. Serre in \cite[Part II, Lecture 2, Theorem 3]{SER1} thus restricts to $\mathfrak{u}$. Denote by $U$ the image of the restricted morphism. It is a smooth connected unipotent subgroup of $G$. 

It remains to show that $\Lie(U) \cong \mathfrak{u}$. 

The algebraic groups $\mathfrak{u}$ and $U$ being smooth, the isomorphism of algebraic groups $(\exp_{\mathfrak{b}})_{\mid \mathfrak{u}}$ induces an isomorphism $\Lie(U) \cong \Lie(\mathfrak{u})$ (see \cite[VIIA Proposition 8.2]{SGA31}). As $\mathfrak{u}$ is a vector space over a field one has $\Lie(\mathfrak{u}) \cong \mathfrak{u}$, hence $\Lie(U) \cong \mathfrak{u}$ as Lie algebras. In other words, the map $\exp_{\mathfrak{b}}$ induces the identity on tangent spaces. Therefore the restricted $p$-nil $p$-subalgebra $\mathfrak{u}$ integrates into a smooth connected  unipotent subgroup $U$ of $G$.
\end{proof}

\begin{remarks}
The following points should help the reader to better understand the issues that are specific to the characteristic $p>0$ framework.
\begin{enumerate}
\item What precedes ensures that when $p>\h(G)$ any restricted $p$-nil $p$-subalgebra of $\mathfrak{g}$ can be integrated into a smooth unipotent connected subgroup of $G$. In particular, under this assumption on $p$ any restricted $p$-nil $p$-subalgebra of an integrable $p$-nil subalgebra of $\mathfrak{g}$ can be integrated. This is not true in general, as shown in section \ref{contre_ex_int} (see in particular Remark \ref{remark_counter_example}).

\item Note that the work of G. Seitz mentioned in the preamble of this section (see \cite[section 5]{Sei}, in particular the Proposition 5.3) allows to relax assumptions on $p$ and to still obtain an integration when $G$ is semi-simple (this should extend easily to arbitrary reductive groups): let $P \subsetneq G$ be a proper parabolic subgroup and let $\mathfrak{p}$ be its Lie algebra. Denote by $\cl(\mathfrak{u_p})$ the nilpotent index of the Lie algebra of $U_P$, the unipotent radical of $P$. Results of G. Seitz ensure that there exists an isomorphism of algebraic groups between $\exp_P : \mathfrak{u_p} \rightarrow U_P$, where the Lie algebra is endowed with an algebraic group structure given by the Baker--Campbell--Hausdorff law  (which is here well defined as shown by the author). This result, coupled with works of G. McNinch (\cite{M2}) allows to integrate several nil Lie subalgebras  in characteristic $p<\h(G)$. Indeed, when $p$ is not of torsion for $G$, given a $p$-nil subalgebra $\mathfrak{u} \subsetneq \mathfrak{g}$ there exists a so-called optimal parabolic subgroup $P_{\mathfrak{u}}(G)$  such that $\mathfrak{u}$ is a Lie subalgebra of $\mathfrak{u_{p_u}}$ (see also \cite[II.3]{JEAthese} for more details on this paper and the aforementioned construction, with the same notations as the one used here). Therefore a precise bound for \ref{exp_car_p} would be to consider $p> \cl(\mathfrak{u_{p_u}})$. However this doesn't help to get rid of the difficulties that occur for small separably good characteristics $p$, hence we preferred a ``rough'' statement which avoid to get lost in technical details here.

\item The situation might seem to be quite similar to the characteristic zero framework. Unfortunately, and contrary to what happens in characteristic zero, the adjoint representation is not compatible with this integration in general. Namely it is not always true that $\exp(t\ad(x))= \Ad(\exp(tx))$ for any $x \in \mathfrak{g}$ (where we denote by $\ad$ the derived representation $\Lie(\Ad)$). Nevertheless this is automatically satisfied if the adjoint representation $\Ad$ is of low height according to \cite[4.6]{BDP}. The authors mention this result as a corollary of \cite[Lecture 4, Theorem 5]{SER}.

\item Let us emphasize another property that is lost in characteristic $p>0$ and already mentioned in the introduction (see \cite[IV,\S 2, 4.5]{DG} for more details on the characteristic $0$ case): even in characteristic $p>\h(G)$ the integration process does no longer induces an equivalence of categories between the category of restricted $p$-nil $p$-Lie algebras and the category of smooth connected unipotent algebraic groups. More precisely: 
\begin{itemize}
	\item let $U$ be a unipotent subgroup of $G$ and denote by $\mathfrak{u} := \Lie(U)$ its Lie algebra. The field $k$ is algebraically closed, thus perfect. The subgroup $U$ is therefore $k$-embeddable into the unipotent radical of a Borel subgroup $B \subset G$. 
	\item Let $\log_B : \Rad_U(B) \rightarrow \radu(B)$ be the inverse isomorphism of algebraic groups of the morphism $\exp_{\mathfrak{b}}$, see \cite[\S 6]{BDP} for an explicit construction. As for $\exp_{\mathfrak{b}}$ it is induced by an isomorphism of reduced $k$-schemes $\log: \mathcal{V}_{\red}(G) \rightarrow\mathcal{N}_{\red}(\mathfrak{g})$. 
\end{itemize}
\noindent In general one cannot expect the integrated group $\exp(\mathfrak{u})$ to be the starting group $U$. Equivalently the  equality $\log_B(U) \neq \mathfrak{u}$ needs not being satisfied a priori. 

For example let $G = \SL_3$ and $p>3$. We consider the unipotent connected smooth subgroups of $G$ generated respectively by the matrices $\begin{pmatrix}
 1 & t & 0 \\
 0 & 1 & 0 \\
 0 & 0 & 1
 \end{pmatrix}$ and $\begin{pmatrix}
 1 & t & t^p \\
 0 & 1 & 0 \\
 0 & 0 & 1
 \end{pmatrix}$. As $U_1 \neq U_2$, the restricted $p$-nil $p$-algebras $\log(U_1)$ and $\log(U_2)$ do not coincide, even though the Lie algebras $\mathfrak{u_1} = \mathfrak{u_2}$ are the same (namely it is the restricted $p$-nil $p$-Lie algebra generated by $\begin{pmatrix}
 0 & 1 & 0\\
 0 & 0 & 0 \\
 0 & 0 & 0
\end{pmatrix}$).
 
\end{enumerate}
\label{pas_corresp_bij} 
\end{remarks}
	
\subsection{From Lie algebras to groups: a natural candidate}
\label{From_Lie_algebras_to_groups_a_natural_candidate}
Let $G$ be a reductive group over an algebraically closed field $k$ of characteristic $p>0$ which is assumed to be separably good for $G$. Let $\phi : \mathcal{N}_{\red}(\mathfrak{g}) \rightarrow \mathcal{V}_{\red}(G)$ be a Springer isomorphism for $G$ such that for any Borel subgroup $B\subset G$ the differential of $\phi$ restricted to $\radu(B)$ is the identity at $0$ (this last assumption is allowed by \cite[Theorem 1.1 and Remark 2]{SOB1} as explained in section \ref{Springer_iso}). It defines for any $p$-nilpotent element of $\mathfrak{g}$ a $t$-power map:
\begin{alignat*}{5}
\phi_{x} : \: & \mathbb{G}_a\: & \rightarrow \: & G  \\
	   \: & t \:& \mapsto \: & \phi_x(t).\\
\end{alignat*}
\noindent Let $\mathfrak{u}$ be a restricted $p$-nil $p$-Lie subalgebra of $\mathfrak{g}$. The $t$-power map $\phi_x$ induces the following morphism:
\begin{alignat*}{5}
\psi_{\mathfrak{u}} : W\: &(\mathfrak{u}) \times \: & \mathbb{G}_a\: & \: & \rightarrow \: & G  \\
	   \: & (x, \: & t)\: & \:& \mapsto \: & \phi_x(t),\\
\end{alignat*}
\noindent where the notations are those of \cite[I, 4.6]{SGA31} (see also \cite[II, Lemme 4.11.7]{SGA31}). Denote by $J_\mathfrak{u}$ the subgroup of $G$ generated by $\psi_{\mathfrak{u}}$ as a $\fppf$-sheaf (see \cite[VIB, Proposition 7.1 and Remark 7.6.1]{SGA31}). This is:
\begin{enumerate}
\item a connected subgroup by \cite[VIB, Corollaire 7.2.1]{SGA31} as $W(\mathfrak{u})$ is geometrically reduced and geometrically connected; 
\item smooth (according to \cite[VIB, Proposition 7.1 (i)]{SGA31} as $G$ is locally of finite type over the field $k$);
\item unipotent as we will see in section \ref{section_Deligne} (see Lemma \ref{alg_lie_groupe_engendre}).
\end{enumerate} 
\noindent One thus needs to compare $\mathfrak{u}$ with the Lie algebra of $J_{\mathfrak{u}}$, denoted by $\mathfrak{j_u}$. We will show that when the restricted $p$-nil $p$-Lie subalgebra $\mathfrak{u}\subseteq \mathfrak{g}$ satisfies some maximality properties (as the one required in the statements of Lemmas \ref{p_nil_integrable} and \ref{p_rad_integrable}), it is integrable by $J_{\mathfrak{u}}$. Before going any further let us stress out that when this integration holds true the normalisers $N_G(J_{\mathfrak{u}})$ and $N_G(\mathfrak{u})$ turn out to be the same. More precisely:

	\begin{lemma}
Let $\mathfrak{u}$ be a restricted $p$-nil $p$-Lie subalgebra of $\mathfrak{g}$. The subgroup $N_G(\mathfrak{u})$ normalises $J_\mathfrak{u}$.
\label{inclusion_normalisateurs_J_u_et_u}
	\end{lemma}

	\begin{proof}
	First notice that $\phi$ is $G$-equivariant (because it is a Springer isomorphism), so is $\phi_x$. Hence the morphism $\psi_{\mathfrak{u}}$ is compatible with the $G$-action on $\mathfrak{u}$. In other words, for any $g \in G$ and any $(x,t) \in \mathfrak{u} \times \mathbb{G}_a$ the equality $\Ad(g)\psi_{\mathfrak{u}}(x,t) = \psi_{\mathfrak{u}} \left(\Ad(g)x,t\right)$ is satisfied.
	
Let $R$ be a $k$-algebra, and consider $g \in N_G(\mathfrak{u})(R)$ and $h \in J_{\mathfrak{u}}(R)$. By definition of $J_{\mathfrak{u}}$ there exists an $\fppf$-covering $S \rightarrow R$ such that $h_S =\psi_{\mathfrak{u}}(x_1,s_1) \cdots \psi_{\mathfrak{u}}(x_n,s_n)$ for $x_i \in \mathfrak{u}_R \otimes_R S $ and $s_i \in S$ (so that $\psi_{\mathfrak{u}}(x_i,s_i) \in J_{\mathfrak{u}}(S)$). But then one has $ (\Ad(g)h)_S = \prod_{i=1}^n \Ad(g_S)\psi_{\mathfrak{u}}(x_i,s_i)$. The morphism $\psi_{\mathfrak{u}}$ being compatible with the $G$-action, this can be rewritten as follows:
	\[(\Ad(g)h)_S = \prod_{i=1}^n \Ad(g_S)\psi_{\mathfrak{u}}(x_i,s_i) = \prod_{i=1}^n \psi_{\mathfrak{u}}\left(\Ad(g_S)x_i,s_i\right) \in J_{\mathfrak{u}}(S)\cap G(R) = J_{\mathfrak{u}}(R),\] 
\noindent where the equality $J_{\mathfrak{u}}(S)\cap G(R) = J_{\mathfrak{u}}(R)$ follows from the fact that $J_{\mathfrak{u}}$ is generated by $\psi_{\mathfrak{u}}$ as a $\fppf$-sheaf. 

We thus have shown that $\Ad(g)h \in J_{\mathfrak{u}}(R)$ for all $g \in N_{\mathfrak{g}}(\mathfrak{u})(R)$. In other words we have shown the inclusion $N_G(\mathfrak{u})(R) \subseteq N_G(J_\mathfrak{u})(R)$ for any $k$-algebra $R$. Yoneda's Lemma then leads to the desired inclusion $N_G(\mathfrak{u}) \subseteq N_G(J_\mathfrak{u})$.
	\end{proof}

\begin{lemma}
When $J_{\mathfrak{u}}$ integrates $\mathfrak{u}$, the equality $N_G(J_{\mathfrak{u}}) = N_G(\mathfrak{u})$ is satisfied.
\label{normalisateurs_p_inf}
\end{lemma}

\begin{proof}
By Lemma \ref{inclusion_normalisateurs_J_u_et_u} one only needs to show the inclusion $N_G(J_{\mathfrak{u}}) \subseteq N_G(\mathfrak{u})$. This is direct according to Lemma \ref{normalisateurs_inclusion} as the equality $\Lie(J_{\mathfrak{u}}) = \mathfrak{u}$ is satisfied by assumption.
\end{proof}	
	
	\begin{remarks}
	Let us emphasize the following points:
	\begin{enumerate}
\item The assumptions on normalisers in Lemma \ref{normalisateurs_p_inf} in particular hold true when $\mathfrak{u}$ is a subalgebra of $\mathfrak{g}$ made of all the $p$-nilpotent elements of the radical of $N_{\mathfrak{g}}(\mathfrak{u})$. This will be shown in Lemma \ref{p_nil_integrable} below.
\item As mentioned in the introduction, not any Springer isomorphism is compatible with the $p$-structure of the restricted $p$-nil $p$-algebra. This is actually not a requirement here. Indeed one considers the image of all $p$-nilpotent elements of $\mathfrak{u}$ so $x$ and all its $p$-powers are taken into account in the integration process. Moreover, one only needs to compare the subalgebras $\mathfrak{u}$ and $\mathfrak{j_u}$ and both of them inherit their $p$-structure from that of $\mathfrak{g}$.
	\end{enumerate}
	\label{remark_p_compatibility}
	\end{remarks}

\subsection{Obstructions to the existence of an integration for embedded restricted $p$-nil $p$-subalgebras}
	\label{section_contre_ex_integration}
	\subsubsection{Witt vectors and family of counter-examples}\phantom{ooo}\\
	In what follows we make use of some general results on Witt vectors to construct a family of counter-examples to the existence of an integration of morphisms and restricted $p$-nil $p$-Lie algebras in general:
	
\begin{ex}
\label{contre_ex_int}
Let $k$ be a perfect field of characteristic $p>0$ and let us consider the following commutative diagram of algebraic groups:

\begin{figure}[H]
\begin{center}
\begin{tikzpicture}
    \matrix(m)[matrix of math nodes, row sep=2em, column sep=2em, text height=1.5ex, text depth=0.25ex]
    {0       & \mathbb{G}_a       & W_2         & \mathbb{G}_a        & 0,  \\
    0        & \mathbb{G}_a       & W'_2 = \left(W_2 \times \mathbb{G}_a\right)/\langle i(x)-i(x^p) \mid x \in \mathbb{G}_a\rangle       & \mathbb{G}_a        & 0.    \\};
     \path[->,font=\scriptsize]

    (m-1-1) edge (m-1-2)
    (m-1-2) edge node[above]{$i$}(m-1-3)
    (m-1-3) edge node[above]{$r$}(m-1-4)
    (m-1-4) edge (m-1-5)

    (m-2-1) edge (m-2-2)
    (m-2-2) edge  node[above]{$i$}(m-2-3)
    (m-2-3) edge node[above]{$r$}(m-2-4)
    (m-2-4) edge (m-2-5)

    (m-1-2) edge node[left] {$\Frob$} (m-2-2)
    (m-1-3) edge (m-2-3)
    (m-1-4) edge node[left]{$=$}(m-2-4)
    ;
\end{tikzpicture}
\end{center}
\end{figure}
\noindent where:
\begin{itemize}
\item we denote by $\Frob$ the absolute Frobenius automorphism,
\item the central term of the lower sequence is the pushout of the morphisms $i$ and $\Frob$.
\end{itemize}  
\noindent The group $\mathbb{G}_a$ being smooth, the exactness of the two horizontal sequences is preserved by derivation (see \cite[II, \S5, Proposition 5.3]{SGA31} and \cite[II, \S7, n\degree 3, Proposition 3.4]{DG}). This leads to the following commutative diagram of restricted $p$-Lie algebras and $p$-morphisms:
\begin{figure}[H]
\begin{center}
\begin{tikzpicture}
    \matrix(m)[matrix of math nodes, row sep=2em, column sep=2em, text height=1.5ex, text depth=0.25ex]
    {0       & \Lie(\mathbb{G}_a)      & \mathfrak{w}_2 := \Lie(W_2)        & \Lie(\mathbb{G}_a)        & 0   \\
    0        &  \Lie(\mathbb{G}_a)      &  \mathfrak{w}'_2       &  \Lie(\mathbb{G}_a)
          & 0.    \\};
     \path[->,font=\scriptsize]

    (m-1-1) edge (m-1-2)
    (m-1-2) edge node[above] {$\Lie(i)$} (m-1-3)
    (m-1-3) edge node[above] {$\Lie(r)$}(m-1-4)
    (m-1-4) edge (m-1-5)

    (m-2-1) edge (m-2-2)
    (m-2-2) edge node[above] {$\Lie(i)$} (m-2-3)
    (m-2-3) edge node[above] {$\Lie(r)$} (m-2-4)
    (m-2-4) edge (m-2-5)

    (m-1-2) edge node[left] {$\Lie(\Frob) = 0$} (m-2-2)
    (m-1-3) edge node[left]{$\pi$}(m-2-3)
    (m-1-4) edge node[left]{$=$}(m-2-4)
    ;
\end{tikzpicture}
\end{center}
\end{figure}
\noindent As $\Lie(\Frob) = 0$ the $p$-morphism $\mathfrak{w'_2} \rightarrow k$ is split (as a $p$-morphism). Let $s :\Lie(\mathbb{G}_a) \rightarrow \mathfrak{w}_2'$ be the resulting splitting.
						    
Even though $\Lie(\mathbb{G}_a)$ and $\mathfrak{w}_2$ are integrable, this splitting does not lift into a morphism of algebraic $k$-groups. The field $k$ being perfect, one only needs to check this on $k$ points. As the vertical morphisms induce the identity morphism on $k$-points if the lifting $s : \Lie(\mathbb{G}_a) \rightarrow \mathfrak{w_2}'$ were integrable into  a morphism of algebraic groups $\sigma : \mathbb{G}_a(k) \rightarrow W'_2(k)$ such that $\Lie(\sigma)_k = s_k$, the lower exact sequence of the above commutative diagram of algebraic groups would be split (because the $\Lie$-functor is left exact). According to the previous remark on $k$-points, the following exact sequence would then also be split: 
\begin{figure}[H]
\begin{center}
\begin{tikzpicture}
    \matrix(m)[matrix of math nodes, row sep=2em, column sep=2em, text height=1.5ex, text depth=0.25ex]
    {0  & \underbrace{\mathbb{G}_a(k)}_{=k} & W_2(k)    & k  & 0   \\};

    \path[->,font=\scriptsize]

    (m-1-1) edge (m-1-2)
    (m-1-2) edge (m-1-3)
    (m-1-3) edge (m-1-4)
    (m-1-4) edge[bend right, dashed] (m-1-3)
    (m-1-4) edge (m-1-5)
    ;
\end{tikzpicture}
\end{center}
\end{figure}
\noindent This leads to a contradiction as $W_2$ would appear as a vector group, while it has $p^2$-torsion (see for example \cite[VII, \S 2, n\degree 10, Proposition 9]{Ser1988}: the construction of $W_2'$ is explained in the proof of the proposition, see also \cite{RosM} for a reminder of vector groups).
 
\end{ex}

\begin{remark}
As pointed out by the referee, connected abelian unipotent subgroups of dimension $2$ over a field are classified, up to isomorphism, by a pair of invariants (see the paragraph after \cite[VII, \S 2, 11, Proposition 11]{Ser1988}). Let $U$ be such a group and denote by $U''$ the subgroup of $U$ whose elements have order dividing $p$. The group law is denoted additively. As exaplained in the aforementioned reference, the purely inseparable isogeny $U/pU \rightarrow U'' : x \mapsto px$ has degree $p^h$ and $h$ is the second invariant of the above pair. This invariant could also be used to show that the splitting $s$ is not integrable in Example \ref{contre_ex_int} above (as this would imply that $W_2$ has  $h=2$ while it is actually equal to $1$).
\end{remark}

\subsubsection{Obstructions in the reductive framework}\phantom{ooo}\\

Let us go back to the framework we are interested in: let $G$ be a reductive group over an algebraically closed field of characteristic $p>0$ which is assumed to be separably good for $G$. Let $U$ and $V \subset G$ be two subgroups. What precedes in particular tells us that if $f : \mathfrak{u}=\Lie(U) \rightarrow \mathfrak{v}=\Lie(V)$ is a morphism of restricted $p$-Lie algebras, it is not true in general that there exists a morphism of groups $U \rightarrow V$ such that $\Lie(\phi) = f$. Namely the map $\Hom(G, \mathbb{G}_a) \rightarrow  \Hom_{p-\Lie}(\mathfrak{g}, k)$ is not surjective. In what follows a morphism of restricted $p$-Lie algebras is integrable if it lifts into a morphism of algebraic groups with smooth kernel.

\begin{remarks}
 One can make the following important remarks on integration of $p$-nil subalgebras:
 \begin{enumerate}
  \item let $\mathfrak{u}\subset \mathfrak{g}$ be a restricted $p$-nil $p$-subalgebra which is integrable into a unipotent smooth connected subgroup $U\subseteq G$. Example \ref{contre_ex_int} together with Lemma \ref{integration_alg_vs_morphism} also shows that not any restricted $p$-nil $p$-subalgebra $\mathfrak{v} \subseteq \mathfrak{u}$ of a restricted $p$-nil $p$-Lie algebra is integrable into a smooth connected unipotent group $V$ such that $\Lie(V) = \mathfrak{v}$. Nevertheless, if we require the inclusion $\mathfrak{v} \subseteq \mathfrak{u}$ to be integrable into a morphism of algebraic groups with smooth kernel, then $\mathfrak{v}$ is integrable into a smooth unipotent subgroup of $U$ (by virtue of \cite[II,\S 5, Proposition 5.3]{DG}). Hence, in what follows a morphism of $p$-Lie algebras $f : \mathfrak{h} \rightarrow \mathfrak{h}'$ is said to be integrable if there exists a morphism of algebraic groups $\phi : H  \rightarrow H'$ with smooth kernel such that $\Lie(\phi)=f$.
 \item Let us stress out that only the restricted $p$-subalgebras $\mathfrak{h} \subseteq \mathfrak{g}$ can pretend to derive from an algebraic group (as this last property automatically implies that $\mathfrak{h}$ is endowed with a $p$-structure inherited from the group, see for example \cite[II, \S7, n\degree 3, Proposition 3.4]{DG}). Moreover, as underlined by the example presented in Remark \ref{pas_corresp_bij} (iii) the integration of restricted $p$-nil $p$-subalgebras of $\mathfrak{g}$ does no longer induce a bijective correspondence with unipotent subgroups of $G$. This in particular implies that the integration of morphisms of restricted $p$-Lie algebras depends on the integration of the Lie algebra one starts with.
   
 \end{enumerate}
\label{remark_counter_example}
\end{remarks} 

The following lemma makes a connection between integration of morphisms and integration of subalgebras:

\begin{lemma}
Let $G$ and $H$ be two algebraic smooth $k$-groups with Lie algebras $\mathfrak{g}:= \Lie(G)$, respectively $\mathfrak{h}:=\Lie(H)$. Assume that $f : \mathfrak{g} \rightarrow \mathfrak{h}$ is a morphism of restricted $p$-nil $p$-Lie algebras which is integrable into a morphism of groups with smooth kernel. Let us denote by $\phi : G \rightarrow H$ the resulting integrated morphism, then $f(\mathfrak{g})$ is integrable into an algebraic smooth connected $k$-group.
\label{integration_alg_vs_morphism}
\end{lemma}

\begin{proof}
Denote by $\mathfrak{v} := f(\mathfrak{g})$ the image of the morphism $f$, which is assumed to be integrable into a morphism $\phi : G \rightarrow H$ with smooth kernel. One can a priori only expect the inclusions  $f(\mathfrak{g}) \subseteq \Lie(\phi(G)) \subseteq \mathfrak{h}$ to hold true. However, as $k$ is a field and $\ker(\phi)$ and $G$ are smooth so is $\phi(G)$. As a consequence, the restricted morphism $f = \Lie(\phi) : \mathfrak{g} \rightarrow \Lie(\phi(G))$ is surjective (see
\cite[II, \S 5, Proposition 5.3]{DG}), whence the equality $\mathfrak{v} = \Lie(\phi(G))$. In particular the restricted $p$-Lie algebra $\mathfrak{v}$ is integrable into an algebraic smooth connected $k$-group.
\end{proof}

\begin{remark}
Let $\phi : G\rightarrow H$ be a smooth morphism of algebraic $k$-groups. Assume that the derived morphism $\Lie(\phi) : \mathfrak{g} \rightarrow \mathfrak{h}$ has a splitting $s:\mathfrak{h} \rightarrow \mathfrak{g}$ which is also a morphism of restricted $p$-Lie algebras. It is worth noting that
this splitting does not necessarily lift into a splitting of algebraic groups: consider for instance the Artin--Schreier covering of $\mathbb{G}_a \rightarrow \mathbb{G}_a : t \mapsto t^p-t$, its derived morphism is nothing but the identity, whence it admits a splitting that does not lift to a splitting of algebraic groups.
%
%
%
%
%
\end{remark}

\section{\texorpdfstring{$\phi$}{Lg}--infinitesimal saturation and proof of Theorem \ref{generalisation_Deligne}}
\label{section_Deligne}

In what follows $G$ is a reductive group over an algebraically closed field $k$ of characteristic $p>0$ which is assumed to be separably good for $G$. Let $\phi : \mathcal{N}_{\red}(\mathfrak{g}) \rightarrow \mathcal{V}_{\red}(G)$ be a Springer isomorphism for $G$ such that for any Borel subgroup $B\subseteq G$ the differential at $0$ of $\phi$ restricted to $\rad_U(B)$ is the identity. 

	\subsection{\texorpdfstring{$\phi$}{Lg}-infinitesimal saturation}
	\label{phi_saturation_infinitesimale}
The following definition extends the notion of infinitesimal saturation to the separably good characteristics.
\begin{defn}
A subgroup $G' \subseteq G$ is $\phi$-infinitesimally saturated if for any $p$-nilpotent element $x \in \mathfrak{g'} := \Lie(G')$ the $t$-power map:
	\begin{alignat*}{3}
		\phi_x : \: & \mathbb{G}_a \: & \rightarrow \: & G,\\
			\:& t \: & \mapsto \: & \phi(tx),
	\end{alignat*} 
\noindent factorises through $G'$. In other words we ask for the following diagram to commute:
\begin{figure}[H]
\begin{center}
\[\begin{tikzpicture} 

 \matrix (m) [matrix of math nodes,row sep=2em,column sep=1.8em,minimum width=2em,  text height = 1.5ex, text depth = 0.25ex]
  {
    \mathbb{G}_a & G.\\
    G' & \\};
  \path[-stealth]
  	(m-1-1) edge
  			node [above] {$\phi_x$} (m-1-2)
    (m-1-1) edge[dashed] 
    		node [left] {$\exists$} (m-2-1) 
    (m-2-1) edge (m-1-2) ;
    	
\end{tikzpicture}\]
\end{center}
\end{figure}
\label{def_phi_sat}
\end{defn}
	
It follows from the definition that the group $G$ is itself $\phi$-infinitesimally saturated. Let us stress out that there are non trivial examples of $\phi$-infinitesimally saturated subgroups of $G$, namely:

\begin{lemma}
Any parabolic subgroup of $G$ is $\phi$-infinitesimally saturated, so are the Levi subgroups and the unipotent radical of any parabolic subgroup $P \subset G$.
\label{parab_phi_inf_sat}
\end{lemma}

\begin{proof}
In order to show this result we make use of the dynamic method introduced in \cite[4]{C} and \cite[\S2.1]{CGP}. Let $T \subset P \subset G$ be respectively a maximal torus and a parabolic subgroup of $G$. As $k$ is a field there exists a non-necessarily unique cocharacter of $T$, denoted here by $\lambda : \mathbb{G}_m \rightarrow G$ such that $P = P_G(\lambda)$ (see \cite[Proposition 2.2.9]{CGP}). We aim to show that for any $p$-nilpotent element $x \in \mathfrak{p}$ the image of the $t$-power map $\phi_x$ belongs to $P= P_G(\lambda)$. The field $k$ being algebraically closed, this is enough to show it on $k$-points. As a reminder when $P$ is of the form $P_G(\lambda)$ the $k$-points of $P$ are nothing but the set:
\[P_G(\lambda)(k) = \{g \in G(k) \mid \lim_{s\to 0} \lambda(s)\cdot g \text{ exists}\}.\]
\noindent Hence one only needs to prove that $\lim_{s\to 0} \lambda(s)\cdot \phi(tx)$ exists. This can be done by making use of the $G$-equivariance of $\phi$. This leads to the equality $\lambda(s)\cdot \phi(tx) = \phi(\lambda(s) \cdot tx)$. Moreover as $x \in \mathfrak{p_g(\lambda)}:= \Lie(P_G(\lambda))$ the limit $\lim_{s\to 0} \lambda(s) \cdot x$ exists by definition. We deduce from the above equality that $\lim_{s\to 0} \lambda(s)\cdot \phi(tx)$ exists, meaning that $\phi(tx)\in P_G(\lambda)= P$, whence the result.

The same reasoning as before, together with \cite[Lemma 2.1.5]{CGP}, allows us to show that: 
\begin{itemize}
\item the unipotent radical of any parabolic subgroup $P \subseteq G$ is $\phi$-infinitesimally saturated as 
\[\Rad_U(P_G(\lambda))(k) = \{g \in G(k) \mid \lim_{s\to 0} \lambda(s)\cdot g =1\},\]
\item the Levi subgroups of any parabolic subgroup $P \subseteq G$ are $\phi$-infinitesimally saturated as 
\[Z_G(P_G(\lambda))(k) = P_G(\lambda)\cap P_G(-\lambda).\]
\end{itemize}
\end{proof}

\begin{remark}
As mentioned in the preamble of section \ref{Springer_iso} when $p>\h(G)$ the only Springer isomorphism for $G$ that restricts to $\exp_B: \mathfrak{u_b} \rightarrow U_B$, whose tangent map is the identity and that is compatible with the $p$-power is nothing but the classical exponential map truncated at the power $p$ (this last condition is not necessary here). In this framework, being $\phi$-infinitesimally saturated is nothing but being $\exp$-saturated, i.e. being infinitesimally saturated as defined by P. Deligne in \cite[Définition 1.5]{D1}). 
\label{sat_inf_coincide}
\end{remark}

This being introduced we can show the following lemma which states that the generated subgroup $J_{\mathfrak{u}}$ seems to be the good candidate to integrate $\mathfrak{u}$ in general.

\begin{lemma}
Let $\mathfrak{u} \subseteq \mathfrak{g}$ be a restricted $p$-nil $p$-subalgebra. Then: 
\begin{enumerate}
\item the generated subgroup $J_{\mathfrak{u}}$ is unipotent,
\item the inclusion $\mathfrak{u} \subseteq \Lie(J_{\mathfrak{u}}): = \mathfrak{j_u}$ is satisfied.
\end{enumerate}  
\label{alg_lie_groupe_engendre}
\end{lemma}

\begin{proof}
The Lie algebra $\mathfrak{u}$ being a restricted $p$-nil $p$-subalgebra of $\mathfrak{g}$ and $p$ being separably good for $G$ (thus not of torsion) Corollary \ref{corollaire_LMT} allows to embed $\mathfrak{u}$ into the Lie algebra of the unipotent radical of a Borel subgroup $B\subset G$. Let us remind the reader of the notation introduced in section \ref{From_Lie_algebras_to_groups_a_natural_candidate}:
\begin{itemize}
\item the Springer isomorphism $\phi$ being fixed, we define
\begin{alignat*}{5}
\psi_{\mathfrak{u}} : W\: &(\mathfrak{u}) \times \: & \mathbb{G}_a\: & \: & \rightarrow \: & G  \\
	   \: & (x, \: & t)\: & \:& \mapsto \: & \phi_x(t),\\
\end{alignat*}
\noindent where $\phi_x(t):= \phi(tx)$ is the $t$-power map.
\item We then denote by $J_{\mathfrak{u}}$ the subgroup of $G$ obtained by considering the $\fppf$-sheaf generated by the image of $\psi_{\mathfrak{u}}$.
\end{itemize} 
\noindent What precedes in particular tells us that $J_{\mathfrak{u}}$ is $k$-embeddable into the unipotent radical of a Borel subgroup. This is because $B$ is $\phi$-infinitesimally saturated according to Lemma \ref{parab_phi_inf_sat}. In other words $J_{\mathfrak{u}}$ is unipotent (see for example \cite[IV, \S2, n\degree2, Proposition 2.5 (vi)]{DG}). 

We still denote by $\phi$ the restriction of the Springer isomorphism to $\radu(B)$. Recall that: 
\begin{itemize}
\item this restriction maps to $\Rad_U(B)$,
\item its differential satisfies $(\diff\phi)_0 = \id$ by assumption.
\end{itemize} 
\noindent The subgroup $J_{\mathfrak{u}}$ being generated by the images of the $t$-power maps $\phi_x$ for all $x\in \mathfrak{u}$, the Lie algebra $\mathfrak{j_u}$ contains the differential at $0$ of all such maps, hence the expected inclusion. 
\end{proof}

It is worth noting that the inclusion $\mathfrak{u}\subseteq \mathfrak{j_u}$ is strict in general, as underlined by the following lemma which is a variation of \cite[VIB, Proposition 7.6]{SGA31}. Notwithstanding this, it will be shown in section \ref{ss_section_intégration_nil} that $J_{\mathfrak{u}}$ does actually integrate $\mathfrak{u}$ when the latter satisfies some maximality hypotheses (see Lemmas \ref{p_nil_integrable} and \ref{p_rad_integrable}).

\begin{proposition}
Let $k$ be a separably closed field and let $(G_i)_{i \in \{1, \cdots, n\}}$ and $G$ be smooth connected $k$-groups. For any $i \in \{1, \cdots, n\}$, consider a smooth morphism of $k$-groups $f_i: G_i \rightarrow G$. Then set: 

\begin{alignat*}{4}
f := \prod_{i=1}^n \: & f_i : \: & (\prod_{i=1}^n G_i) \: & \rightarrow \: & G\:&\\
\: & \: & (g_1, \cdots, g_n) \:& \mapsto \: & f_1\: &(g_1) \cdots f_n(g_n).
\end{alignat*}

\noindent and for any $N \in \mathbb{N}_{>0}$ define: 
\begin{alignat*}{3}
f_N := f \times \cdots \times f : \: & \left(\prod_{i=1}^n G_i \right)^{\times N} \: & \rightarrow \: & G \\
\: & (x_1, \cdots, x_N) \: & \mapsto \: & f(x_1) \cdots f(x_N)
\end{alignat*} given by taking $N$ times the morphism $f$. The following assertions are equivalent:
\begin{enumerate}
\item there exists an integer $N\geq 1$ for which the morphism $f_N$ is surjective and smooth over a non empty open subset of $(\prod_{i=1}^n G_i)^{\times N}$,
\item for $N\geq 1$ large enough the morphism $f_N$ is surjective and separable,
\item the Lie algebra of $G$ decomposes as a $k$-vector space as follows: 
\[\Lie(G) = \sum_{j=1}^n\Ad(h_j)
\left( \Lie(f_j(G_j))\right),\] 
\noindent where $h_j \in M(k)$ for $M = \langle f_i(G_i) \rangle_{i\in \{1, \cdots, n\}}$, the subgroup generated by the $f_i(G_i)$'s,
\item the group $G$ is generated by the images of the $G_j$'s on the big étale site.
\end{enumerate}
\label{etally_generated}
\end{proposition}

\begin{remark} 
If the equivalent conditions of Lemma \ref{etally_generated} are satisfied then in particular the $k$-group $G$ is generated by the images of the $G_j$'s for the $\fppf$-topology. 
\end{remark}

\begin{proof}
We show $\begin{tikzcd}[arrows=Rightarrow, column sep=1.3cm, row sep=1.3cm, every arrow/.append style={shift left=0.8ex}]
  \text{(i)} \arrow {r} 
    \arrow {d}
&\text{(ii)}   
    \arrow {d} 
\\
  \text{(iv)} 
    \arrow {r}
& \text{(iii)}    
\arrow [shift right=1.2ex] {ul}
\end{tikzcd}$. In order to avoid heavy notations we only focus on the case $n=2$ in the statement of the lemma, the general proof follows by induction.
 \begin{itemize}[leftmargin=2cm]
 	\item[(i) $\implies$ (ii)]  Let $N\geq 1$ be an integer such that $f_N : (G_1 \times G_2)^{\times N} \rightarrow G$ is smooth and surjective over a non empty open subset of $(\prod_{i=1}^n G_i)^{\times N}$. Denote by $U$ the image of this open set under $f_N$. It is open in $G$ as $f$ is open. One first needs to obtain the surjectivity on the whole product of $m$ terms (for $m$ large enough). Remark that since we are working with algebraic groups it is enough to consider $f_{2N}$ (thus $m = 2N$) rather than $f_N$ to obtain this property. This is so because the natural morphism $U(k) \cdot U(k) \rightarrow G(k)$ is surjective as $G$ is a $k$-algebraic group). 
 	
Now, as being separated is nothing but being generically smooth, one only needs to show that $f_{2N}$ is smooth on a dense open subset of $G$. It suffices to show that there exists $z \in (G_1\times G_2)^{\times 2N}$ such that $(df_{2N})_z$ is surjective because the source and the target of $f_{2N}$ are smooth varieties (see \cite[I, \S4, Corollaire 4.14]{DG}). The map $f_N$ being smooth over a non empty open subset of $(G_1 \times G_2)^{\times N}$, one can find an element $x \in (G_1 \times G_2)^{\times N}$ such that $(df_N)_x$ is surjective. This implies that so is $(df_{2N})_{(1,x)}$ and allows to conclude that $f_{2N}$ is smooth over an open neighborhood of $(1,x)$. 
 	
 	\item[(ii) $\implies$ (iii)]
Let $N \in \mathbb{N}_{>0}$ be such that the morphism $f_N$ is separable and surjective. These two assumptions together ensure that there exists an element \[h=\left(h_{1,i},h_{2,i}\right)_{i=1}^N \in (G_1(k) \times G_2(k))^{\times N}\] 
\noindent such that $(d f_N)_{h} : T\left((G_1 \times G_2)^N\right)_h \rightarrow T(G)_g$ is surjective for $g=f_N(h)$. Set \linebreak $g_{1,i} = h_{1,1}h_{1,2}\hdots h_{1,i}$ and $g_{2,i} = h_{2,1}h_{2,2}\hdots h_{2,i}$ and consider the map:
\begin{alignat*}{3}
\alpha : \: & (G_1 \times G_2)^{\times N} \: & \rightarrow \: & (G_1 \times G_2)^{\times N}\\
\: & (x_{1,1}, \hdots x_{2,N}) \: & \mapsto \: & (h_{1,1}x_{1,1}, \hdots, h_{2,N}x_{2,N}).
\end{alignat*}
\noindent This allows to translate $f_N$ to the origin as illustrated by the following diagram that can be shown to be commutative:
\[\begin{tikzcd}[arrows=rightarrow, column sep=1.3cm, row sep=1.3cm]
(G_1 \times G_2)^{\times N}   
\arrow{d}{\Ad^N(g_{j,i})}
    \arrow[r, "\alpha"]
&  (G_1 \times G_2)^{\times N}  
    \arrow[r, "f_N"] 
& G    
\arrow{d}{\rho(g^{-1})}
\\
  (G_1 \times G_2)^{\times N}
    \arrow {rr}{f_N''}
&   
& G     
\\  
\end{tikzcd}\]
\noindent where $\Ad^N(g_{j,i}): = \prod_{i=1}^N(\Ad(g_{1,i}), \Ad(g_{2,i}))$ is the diagonal conjugation by the $g_{j,i}$'s for $j \in \{1,2\}$ and $\rho(g^{-1})$ is the right multiplication. This in particular implies that the differential $d \left(f_N'' \circ \Ad^N(g_{j,i})\right)_e : T\left((G_1 \times G_2)^{\times N} \right)_e \rightarrow \Lie(G)=:\mathfrak{g}$ is surjective. Thus any $z \in \mathfrak{g}$ occurs as: 
\begin{alignat*}{3}
d(f_N'' \circ \left(\Ad^N(g_{j,i})\right))_e (x_{1,1},\hdots, x_{2,N}) \: & = \sum_{i=1}^N \left(\Ad(h_{1,i})(x_{1,i}) + \Ad(h_{2,i})(x_{2,i})\right)
\end{alignat*}
\noindent for an element $(x_{1,1},\hdots, x_{2,N}) \in T\left((G_1 \times G_2)^{\times N} \right)_e$, whence the desired equality of vector spaces. 

 	\item[$(iii) \implies (i)$] Let $x \in \mathfrak{g}$, by assumption there exist natural integers $n$ and $m$ such that: 
 	\begin{alignat*}{3}
 	x \: & = \sum_{i=1}^{n}\left(\Ad(h^{1,i})(x_{1,i}) + \Ad(h^{2,i})(x_{2,i})\right)\\
 	\: & = \sum_{i=1}^{n} \left(\Ad\left(\prod_{j=1}^m g^{1,i}_{1,j}g^{1,i}_{2,j}\right)(x_{1,i}) + \Ad\left(\prod_{l=1}^m g^{2,i}_{1,l}g^{2,i}_{2,l} \right) (x_{2,i})\right)
 	\end{alignat*}
\noindent where:
\begin{itemize}
\item[--] for $q \in \{1,2\}$, the $h^{q,i}$'s belong to $\langle f_1(G_1)(k), f_2(G_2)(k) \rangle$ hence decompose into products of $g^{q,i}_{r,j} \in f_r(G_r)(k)$ for $r \in \{1,2\}$ and $j \in \{1,\hdots, m\}$. Note that they may be equal to $1$;
\item[--] the $x_{q,i} = f(z_{q,i})$'s are elements of $\Lie(f_q(G_q))$.
\end{itemize}  

Recall that, as previously noticed (in the proof of the last implication), for any $g \in G$ the tangent space of $G$ at $g$ identifies with the Lie algebra of $G$. Hence the surjectivity of $(df_{N})_e$ (as for any $x \in \mathfrak{g}$, the $N$-tuple $(z_{1,i}, z_{2,i})$ is an antecedent for $(df_{N})_e$).
%
The derived morphism $(df_N)_e$ being surjective and $G$ and $(G_1 \times G_2)^{\times N}$ being smooth, the morphism $f_N$ is smooth over a non empty open subset $U \subset (G_1 \times G_2)^{N}$ (according to \cite[I, \S 4, Corollaire 4.14]{DG}). It remains to show the surjectivity of $f_N$ which is direct the natural morphism $U(k) \cdot U(k) \rightarrow G(k)$ being surjective (because $G$ is an algebraic group).

 	\item[(i) $\implies$ (iv)] For any $k$-algebra $R$ and any $g \in G(R)$ one needs to show that there exists an étale cover $S \rightarrow R$ on which $g$ writes $g_S = g_{1,1} g_{2,1} \cdots g_{1,N}g_{2,N}$, where $g_{1,i} \in f_1(G_1)(S)$ and $g_{2,i} \in f_2(G_2)(S)$ for $i \in \{1, \cdots, N\}$. 
 	This is therefore actually enough to prove the statement when $R$ is strictly henselian. One thus only needs to prove it on the residue field $\kappa$, as Hensel lemma holds true allowing to lift the desired property. The morphism $f_N$ is surjective and smooth over an open cover of $(G_1\times G_2)^{\times N}$ so its image is a dense open subset $U \subset G$, hence the result as $U(\kappa) \cdot U(\kappa) \rightarrow G(\kappa)$ is surjective.
 	
 	\item[(iv) $\implies$ (iii)] By assumption there exists an integer $N \geq 1$ for which the morphism $f_N$ is a covering (see \cite[VIB, Proposition 7.4 and Proposition 7.6]{SGA31}), hence its surjectivity. Any \linebreak $g \in  G(k[\epsilon_1, \epsilon'_1, \hdots, \epsilon_N, \epsilon'_{N}])$ thus writes $g  = \prod_{j=1}^N (h_{1,j} + \epsilon_j x_{1,j})(h_{2,j} + \epsilon'_j x_{2,j})$ for
 	
$h_{i,j} \in f_{i}(G_i)(k)$ with $i \in \{1,2\}$, and $x_{i,j} \in \Lie(f_i(G_i))$. Hence the map \[T\left((f_1(G_1) \times f_2(G_2))^{\times N}\right)_{h}\rightarrow T(G)_{f_N(h)}\] 
\noindent is surjective, for $h=(h_{1,j}, h_{2,j})_{j \in \{1, \cdots, N\}}$. We now run exactly the same reasoning as in the proof (ii) $\implies$ (iii): set $g_{1,i} = h_{1,1}h_{1,2}\hdots h_{1,i}$ and $g_{2,i} = h_{2,1}h_{2,2}\hdots h_{2,i}$. That leads to consider the map:
\begin{alignat*}{3}
\alpha : \: & (G_1 \times G_2)^{\times N} \: & \rightarrow \: & (G_1 \times G_2)^{\times N}\\
\: & (x_{1,1}, \hdots x_{2,N}) \: & \mapsto \: & (h_{1,1}x_{1,1}, \hdots, h_{2,N}x_{2,N}).
\end{alignat*}
\noindent This allows us to translate $f_N$ to the origin, as this can be read on the following diagram (which is commutative):
\[\begin{tikzcd}[arrows=rightarrow, column sep=1.3cm, row sep=1.3cm]
(G_1 \times G_2)^{\times N}   
\arrow{d}{\Ad^N(g_{j,i})}
    \arrow[r, "\alpha"]
&  (G_1 \times G_2)^{\times N}  
    \arrow[r, "f_N"] 
& G    
\arrow{d}{\rho(g^{-1})}
\\
  (G_1 \times G_2)^{\times N}
    \arrow {rr}{f_N''}
&   
& G,    
\\  
\end{tikzcd}\]
\noindent where $\Ad^N(g_{j,i}): = \prod_{i=1}^N(\Ad(g_{1,i}), \Ad(g_{2,i}))$ is the diagonal conjugation by the $g_{j,i}$'s for $j \in \{1,2\}$. This in particular implies that the differential:
\[d \left(f_N'' \circ \Ad^N(g_{j,i})\right)_e : T\left((G_1 \times G_2)^{\times N} \right)_e \rightarrow \Lie(G)=:\mathfrak{g}\] 
\noindent is surjective. Thus any $z \in \mathfrak{g}$ can be rewritten as: 
\begin{alignat*}{3}
d(f_N'' \circ \left(\Ad^N(g_{j,i})\right))_e (x_{1,1},\hdots, x_{2,N}) \: & = \sum_{i=1}^N \left(\Ad(h_{1,i})(x_{1,i}) + \Ad(h_{2,i})(x_{2,i})\right)
\end{alignat*}
\noindent for an element $(x_{1,1},\hdots, x_{2,N}) \in T\left((G_1 \times G_2)^{\times N} \right)_e$, whence the desired equality of vector spaces. 
\end{itemize}
\end{proof}

\begin{remark}
Let $G$ be a reductive $k$-group of finite presentation and let $H \subseteq G$ be the $k$-subgroup of $G$ generated by the $f_i(G_i)$'s as chosen in the above lemma. Note that $H$ is smooth and connected because so are the $G_i$'s. Under some extra assumptions such as:
\begin{enumerate}
\item the smoothness of all normalisers $N_G(\mathfrak{v})$ of all subspaces $\mathfrak{v}$ of $\mathfrak{g}$ (which is ensured under very strict conditions on $p$, as described in \cite[Theorem A]{HS}),
\item the smoothness of $N_G(\langle \Lie(f_i(G_i))\rangle_{i=1}^n)$,
\end{enumerate}
\noindent the third point of the above lemma also allows to conclude that $\Lie(G)$ is generated by the $\Lie(f_i(G_i))$'s as a restricted $p$-Lie algebra. Indeed one only needs to obtain the inclusion \linebreak $H \subseteq N_G(\langle \Lie(f_i(G_i))\rangle)$. The $f_i$'s being morphisms of groups one actually only needs to show that $f_i(G_i) \subset N_G(\langle \Lie(f_i(G_i))\rangle)$. Under the above assumptions the proof is the same as the one in characteristic $0$ (that can be found, for instance, in \cite[II, 7.6]{BORlag}). 

\end{remark}

The above remark provides some examples under which the Lie algebra of $H$ is the restricted $p$-Lie algebra generated by the Lie algebras of the $f_i$'s. The remark below however illustrates the necessity of assumptions made in  Proposition \ref{etally_generated}, by providing examples for which its conclusion does not hold true.

\begin{remark}
Let $G$ be an algebraic group over a separably closed field $k$ of characteristic $p>0$ and let $(f_i: G_i \rightarrow G)_{i=1, \cdots, n}$ be a family of $n$ smooth morphisms of $k$-groups, where the $G_i's$ are assumed to be smooth and connected. Assume that $G$ is generated by the $f_i(G_i)'s$. In general it is not true that $\Lie(G)$ is generated by the $\Lie(f_i(G_i))$'s, as shown on the two examples below:
\begin{enumerate}
\item assume $G = (\mathbb{G}_a)^2$. Set: 
\begin{alignat*}{3}
f_1: G_1:=\: &\mathbb{G}_a \: & \rightarrow \: & G\\
\: & x \: & \mapsto \: & (x,0),
\end{alignat*}
\begin{alignat*}{3}
f_2:G_2:=\: & \mathbb{G}_a \: & \rightarrow \: & G\\
\:& x \: & \mapsto \: & (x, x^p),
\end{alignat*} 
so $G = \langle f_1(G_1), f_2(G_2)\rangle$. Note that $\Lie(f_1(G_1)) = \Lie(f_2(G_2)) = k$, hence: \[\langle \Lie(f_1(G_1)), \Lie(f_2(G_2)) \rangle= k \neq \Lie(G);\]
\item 
the Lie algebra of $[G,G]$ (for $[G,G]$ the derived group of $G$) does not necessarily coincide with the derived Lie algebra $[\mathfrak{g}, \mathfrak{g}]$. For example if $G = \SL_p = [\GL_p, \GL_p]$ then $\ssl_p$ is nothing but the matrices of size $p \times p$ with trace zero, which does not coincide with $[\mathfrak{gl}_p, \mathfrak{gl}_p]$ due to the assumption on the characteristic.
\end{enumerate}
\label{contre-ex_gpes_engendres}
\end{remark}

The notion of $\phi$-infinitesimal saturation introduced here also allows us to extend theorems \cite[Théorème 1.7]{D1} and \cite[Theorem 2.5]{BDP} to $\phi$-infinitesimally reductive $k$-groups $N$ over an algebraically closed field $k$ of characteristic $p>0$ which is assumed to be separably good for $G$. This is the point of Theorem \ref{generalisation_Deligne}. Let us first remark that points (i) and (iii) of \cite[Lemme 2.3]{D1} are still valid in the aforementioned framework and allow us to reduce ourselves to show the result for connected $N$. More precisely:
\begin{lemma}
Let $G$ be a reductive group over an algebraically closed field $k$ of characteristic $p>0$ which is assumed to be separably good for $G$, and let $N \subset G$ be a subgroup of $G$. The following assertions hold true:
\begin{enumerate}
\item if $N$ is $\phi$-infinitesimally saturated in $G$ then so is $N^0$,
\item if the reduced part $N^0_{\red}$ of $N^0$ and its unipotent radical $\Rad_U(N^0_{\red})$ are normal subgroups of $N^0$ then they are normal in $N$.
\end{enumerate}
\label{analogue_2.3}
\end{lemma}

\begin{proof}
See \cite[Lemme 2.3]{D1} for a proof as the notion of $\phi$-infinitesimal saturation is nothing but a generalisation of those of infinitesimal saturation to the framework described above (see Remark \ref{sat_inf_coincide}).
\end{proof}

In what follows the $\phi$-infinitesimally saturated group $N$ is therefore assumed to be connected. In order to state and show the $\phi$-infinitesimal version of P. Deligne's result stated in the introduction of this article (see Theorem \ref{generalisation_Deligne}), one will need a fundamental result on maximal $k$-groups of multiplicative type, which is stated and showed in section \ref{A_preliminary_result} below.

\subsection{A preliminary result on maximal $k$-groups of multiplicative type}
\label{A_preliminary_result}
 \begin{corollary}[(Corollary of {\cite[Proposition A.2.11]{CGP}})]
Let $k$ be a field and $G$ be an affine smooth algebraic $k$-group. The maximal connected subgroups of multiplicative type of $G$ are the maximal tori of $G$.
\label{tm_connexe_max_cas_red}
\end{corollary}

\begin{proof}
Without loss of generality one can assume $G$ to be connected (as any maximal connected subgroup of $G$ is contained in the identity component $G^0$). Let $H \subset G$ be a maximal connected subgroup of multiplicative type. 

Note that, as explained in the proof of \cite[Corollaire 3.3]{BDP}, the connected centraliser of $H$ in $G$, denoted by $Z_G^0(H)$, is a smooth subgroup of $G$. This is an immediate consequence of the smoothness theorem for centralisers (see for example \cite[II, \S5, 2.8]{DG}): the group $G$ being smooth, the set of $H$-fixed points of $G$ (for the $H$-conjugation) is smooth over $k$. 

We proceed by induction on the dimension of $G$, the case of dimension $0$ being trivial.

If now the group $G$ is of strictly positive dimension then:
	\begin{enumerate}
		\item either the inclusion $Z_G^0(H) \subset G$ is strict and then $H$ is a maximal connected subgroup of $Z_G^0(H)$ of multiplicative type, thus $H$ is a $k$-torus (of $Z_G^0(H)$, hence of $G$) by induction;
		\item or $Z_G^0(H)= G$ and $H$ is central in $G$. Then, by \cite[Proposition A.2.11]{CGP} (applied to $G$) one has the following exact sequence:
		\begin{figure}[H]
	\begin{center}
\[\begin{tikzpicture} 

 \matrix (m) [matrix of math nodes,row sep=2em,column sep=1.8em,minimum width=2em,  text height = 1.5ex, text depth = 0.25ex]
  {
    1 & G_{t}& G & V & 1,\\};
  \path[-stealth]
  	(m-1-1) edge (m-1-2)
    (m-1-2) edge (m-1-3) 
   	(m-1-3) edge (m-1-4) 
   	(m-1-4) edge (m-1-5) ;
\end{tikzpicture}\]
\end{center}
\end{figure}
\noindent where $V$ is a unipotent smooth connected group and $G_t$ is the $k$-subgroup of $G$ generated by the $k$-tori of $G$. The subgroup $H \subseteq G$ is maximal and connected of multiplicative type in $G$. It thus fulfils the same conditions in $G_t$. The quotient $G/G_t = U$ is indeed unipotent, thus the subgroup of multiplicative type $H$ intersects $U$ trivially. It is therefore included in $G_t$. If $V \neq 1$ then $H$ is a  $k$-torus by induction. Otherwise one has $G_t = G$ and if $T$ is a $k$-torus of $G$ the subgroup $H\cdot T \subset G$ is connected of multiplicative type and contains $H$ so it is equal to $H$ (as $H$ is assumed to be maximal). Finally one actually has $T \subset H$ hence $G_t \subset H$ so we have shown that $H = G_t$. This in particular implies the smoothness of $H$ which turns out to be a $k$-torus.
\end{enumerate}
	
\end{proof}

\subsection{An infinitesimal version of Theorem \ref{generalisation_Deligne}}
\label{infinitesimal_version}
Let $H\subseteq N \subseteq G$ be a maximal connected subgroup of multiplicative type of the $\phi$-infinitesimally saturated subgroup $N$. The $k$-group $H$ is the direct product of a $k$-torus $T$ together with a diagonalisable $k$-group $D$. The latter is a product of subgroups of the form $\mu_{p^i}$, with $i \in \mathbb{N}$. Moreover the $k$-torus $T$ is nothing but the intersection $H\cap N_{\red}$ and it is a maximal torus of $N$ and $N_{\red}$ (according to Corollary \ref{tm_connexe_max_cas_red}). 

Let $Z:=Z^0_{N_{\red}}(T)$ be the connected centraliser of $T$ for the action of $N_{\red}$ and set $W = Z/T$. This is a unipotent subgroup of $N$, the reasoning is the same as the one of \cite[\S2.5]{D1}): according to \cite[XVII, Proposition 4.3.1 iv)]{SGA32} as the field $k$ is algebraically closed one only needs to show that this quotient has no subgroup of $\mu_p$-type. This is clear: if such a factor would exist its inverse image in $Z$ would be an extension of $\mu_p$ by $T$ in $N_{\red}$, hence of multiplicative type. This is absurd as the maximal connected subgroups of multiplicative type of a smooth algebraic group over a field are the maximal tori (by Corollary \ref{tm_connexe_max_cas_red}). Moreover 
\begin{itemize}
\item the groups $T$ and $Z$ being smooth so is $W$ according to \cite[II,\S5, n\degree 5 Proposition 5.3 (ii)]{DG}; 
\item the group $W$ is also unipotent according to what precedes. 
\end{itemize}
\noindent The field $k$ being perfect \cite[exposé XVII, Théorème 6.1.1]{SGA32} holds true and implies the exactness of the following exact sequence:
\begin{figure}[H]
\begin{center}
\[\begin{tikzpicture} 

 \matrix (m) [matrix of math nodes,row sep=2em,column sep=1.8em,minimum width=2em,  text height = 1.5ex, text depth = 0.25ex]
  {
    1 & T & Z & W & 1.\\};
  \path[-stealth]
  	(m-1-1) edge (m-1-2)
    (m-1-2) edge (m-1-3) 
    (m-1-3) edge (m-1-4)
    (m-1-4) edge (m-1-5) 
    (m-1-4) edge[bend left = -50] (m-1-3) ;
\end{tikzpicture}\]
\end{center}
\end{figure}
\noindent To summarise, we have an isomorphism $Z_{N_{\red}}(T)\cong T \times W$. Let $X$ be the reduced $k$-subscheme of $p$-nilpotent elements of $\mathfrak{n}^0 = W(\mathfrak{n}^0) = \Lie(Z_{N_{\red}}(T))$.

\begin{lemma}
The centraliser $Z^0_{N_{\red}}(T)$ is the subgroup of $N$ generated by $T$ and the morphism:
\begin{alignat*}{4}
\psi_X : \: & X \: & \times \: & \mathbb{G}_a \: & \rightarrow \: & G\\
\:& (x,\:& \: & t) \: & \mapsto \: & \phi(tx).
\end{alignat*}
\noindent It is normalised by $H$.
\label{action_H_sur_centralisateur}
\end{lemma}

\begin{proof}
Let $J_X \subseteq G$ be the subgroup of $N$ generated by the image of $\psi_X$. The subgroup $N$ being $\phi$-infinitesimally saturated, the $t$-power map induced by $\phi$ maps any $p$-nilpotent element of $\Lie(Z_{N}(T))$ to $N$. Thus $\psi_X$ factorises through $N$. Moreover as $J_X$ is the image of a reduced $k$-scheme it is reduced hence smooth (as $k$ is algebraically closed). Thus the inclusion $J_X \subset N_{\red}$ holds true. Finally, $\phi$ is $G$-equivariant because it is a Springer isomorphism. This implies that the image of $\psi_X$ commutes with any element of $T$. We just have shown that $J_X \subset Z_{N_{\red}}(T)$.
 
Let also $E_{T,J^0_X}$ be the subgroup generated by $T$ and $J^0_X$ as a $\fppf$-sheaf. It is: 
\begin{itemize}
	\item a smooth subgroup of $Z_{N_\red}(T)$ as $T \in Z^0_{N_{\red}}(T)$ is smooth and $J^0_X \subset Z^0_{N_{\red}}(T)$ by what precedes, 
	\item connected according to \cite[VIB. Corollaire 7.2.1]{SGA31}, the torus $T$ being geometrically connected and geometrically reduced. 
\end{itemize}
\noindent Thus $E_{T,J^0_X}$ is actually contained in the identity component of the reduced centraliser. At the Lie algebras level this leads to the following inclusion: 
\[\Lie(E_{T,J^0_X})  \subseteq \Lie(Z^0_{N_{\red}}(T))= \Lie(Z_{N_{\red}}(T)).\]

As $W$ is a unipotent subgroup of $Z^0_{N_{\red}}(T) \cong T \times W$ the Lie algebra $\mathfrak{w} : = \Lie(W)$ is a restricted $p$-nil $p$-subalgebra of $Z^0_{N_{\red}}(T)$, hence is contained in the reduced sub-scheme $X$. The latter is the set of $p$-nilpotent elements of $\Lie(Z_N(T))$ so it is contained in the set of all $p$-nilpotent elements of $\mathfrak{g}$. This set coincides with $\rad_p(\mathfrak{g})$ by Lemma \ref{p_rad_gpe_lisse_connexe}, which holds true as either $p\geq 3$ or if $p=2$ the conditions defined in Remark \ref{remark_autorise_car2} ii) are satisfied. As $p$ is not of torsion for $G$, Corollary \ref{corollaire_LMT} holds true and allows to embed $\rad_p(\mathfrak{g})$, thus $X$, into the Lie algebra of the unipotent radical of a Borel subgroup $B\subseteq G$. Remember that the differential at $0$ of the restriction of $\phi$ to this subalgebra satisfies $(\diff\phi)_0 = \id$. The group $J_X$ being generated by the image of $\psi_X$, this property ensures that the differential at $0$ of any $\phi(tx)$, for any $t \in \mathbb{G}_a$ and $x \in X$, belongs to $\mathfrak{j}_X:=\Lie(J_X)=\Lie(J_X)^0$. In other words one has the following inclusions $\mathfrak{w}\subseteq X \subseteq \Lie(J_X^0)$. Moreover the inclusion $T\subseteq E_{T,J^0_X}$ induces an inclusion of Lie algebras $\mathfrak{t} := \Lie(T) \subseteq \Lie(E_{T,J^0_X})= \Lie(E_{T,J^0_X})$.

As one has $Z_{N_{\red}}(T) \cong T \times W$, what precedes leads to the following inclusion: \[\Lie(Z^0_{N_{\red}}(T)) = \Lie(Z_{N_{\red}}(T)) \subseteq \Lie(E_{T,J^0_X}),\] 
thus to the equality $\Lie(Z^0_{N_{\red}}(T)) = \Lie(E_{T,J^0_X})$. As the groups involved here are smooth and connected this equality of Lie algebras lifts to the group level according to \cite[II, \S 5 n\degree 5.5]{DG}, whence the equality $Z^0_{N_{\red}}(T) = E_{T,J^0_X}$.

It then remains to show that $E_{T,J^0_X}$ is normalised by $H$. Recall that it is the subgroup generated by $T\subseteq H$ (which is normal in $H$) and $J^0_X$ (which is characteristic in $J_X$, see \cite[II, \S 5, n\degree 1.1]{DG}). Hence one only needs to show that $J_X$ is $H$-stable. First remark that $X$ is stabilised by $H$ because the latter stabilises $\Lie(Z_{N}(T))$ and the $p$-nilpotency is preserved by the adjoint action. The $G$-equivariance of $\phi$ (thus its $H$-equivariance) then allows to conclude: let $R$ be a $k$-algebra. For any $j \in J_X(R)$ and $h \in H(R)$ there is an $\fppf$-covering $S \rightarrow R$ such that $j_S = \psi_{X}(x_1,t_1) \cdots \psi_{X}(x_n,t_n)$ where $x_i \in X_R \otimes_R S $ and $s_i \in S$. But then one has:
	\[(\Ad(h)j)_S = \prod_{i=1}^n \Ad(h_S)\psi_{X}(x_i,s_i) = \prod_{i=1}^n \psi_{X}\left(\Ad(h_S)x_i,s_i\right) \in J_{X}(S)\cap G(R) = J_{X}(R),\]  
\noindent and by Yoneda Lemma $E_{T,J^0_X}$ is stable under the $H$-action. 
\end{proof}

\begin{lemma}
The restricted $p$-Lie algebra $\mathfrak{n}_{\red} := \Lie(N_{\red})$ is an ideal of $\mathfrak{n}$ acted on by $H$.
\label{Lie_N_red_stable_par_H}
\end{lemma}

\begin{proof}
According to the proof of \cite[Lemma 2.14]{BDP} the morphism of $k$-schemes $N_{\red}\times H \rightarrow N$ is faithfully flat. This being said $N$ appears as the $\fppf$-sheaf generated by $N_{\red}$ and $H$. Thus in order to show that $\mathfrak{n}_{\red}$ is actually $N$-stable one only needs to show that $\mathfrak{n}_{\red}$ is $H$-stable . The torus $T = H \cap N_{\red}$ acts on $N_{\red}$, respectively on $N$, leading to the following decompositions:
\[\mathfrak{n_{\red}} = \Lie(N_{\red}) = \Lie(Z_{N_{\red}}(T)) \oplus \bigoplus_{\alpha \in X(T)^*} \mathfrak{n}_{\red}^{\alpha},\]
\[\text{and } \Lie(N) = \Lie(Z_N(T)) \oplus \bigoplus_{\alpha \in X(T)^*} \mathfrak{n}^{\alpha},\]
\noindent where $X(T)^*$ stands for the group of non trivial characters of $T$. Any factor in the decomposition of $\Lie(N)$ is stable for $H$ as $T$ is normal in $H$ and we need to show that so is any factor of the decomposition of $\mathfrak{n}_{\red}$. Let us first study the positive weight spaces. The group $N$ being generated as a $\fppf$-sheaf by $N_{\red}$ and the subgroup of multiplicative type $H$ (whose Lie algebra is toral) the $p$-nilpotent elements of $\Lie(N)$ are the $p$-nilpotent elements of $\Lie(N_{\red})$. This being observed, as for any $\alpha \neq 0$ the weight space $\mathfrak{n}^{\alpha}$ has only $p$-nilpotent elements (because we consider the action of a torus here) the equality $\mathfrak{n}_{\red}^{\alpha}:= \mathfrak{n}^{\alpha} \cap \mathfrak{n}_{\red} = \mathfrak{n}^{\alpha}$ is satisfied, whence the desired $H$-stability. 

It remains to show that $\Lie(Z_{N_{\red}}(T))$ is $H$-stable. According to Lemma \ref{action_H_sur_centralisateur} the subgroup $H$ normalises $Z_{N_{\red}}(T)^0$, thus the stability of $\Lie(Z_{N_{\red}}(T)^0) = \Lie(Z_{N_{\red}}(T))$. 

According to what precedes $\mathfrak{n}_{\red}$ is stable for the action of $H$ on $\mathfrak{n}$, hence this subalgebra is invariant for the action of $N$. Reasoning  on the $R[\epsilon]$-points for any $k$-algebra $R$, one can shows that $\mathfrak{n_{\red}}$ is an ideal of $\mathfrak{n}$.
\end{proof}

The proof of the following lemma is the same as the proof of \cite[Lemma 2.22]{D1} because relaxing the hypotheses had no consequences on the involved arguments. We reproduce the proof here to ensure a consistency in notations.

\begin{lemma}[(P. Deligne, {\cite[2.22]{D1}})]
Let $V$ be the unipotent radical of $N_{\red}$. The action of $H$ on $\Lie(N_{\red}) = \mathfrak{n}_{\red}$ leaves $\Lie(V) := \mathfrak{v}$ invariant.
\label{Lie_rad_unip-stable_par_H}
\end{lemma}

\begin{proof}
The torus $T$ acts on $\mathfrak{n}_{\red}$ thus on $\mathfrak{v}$. The Lie algebras $\mathfrak{n}_{\red}$ and $\mathfrak{v}$ have a weight space decomposition for this action, namely $\mathfrak{n}_{\red} =\mathfrak{n}_{\red}^0 \oplus \bigoplus_{\alpha \in X(T)^*} \mathfrak{n}_{\red}^{\alpha}$ and $\mathfrak{v} = \mathfrak{v}^0 \oplus \bigoplus_{\alpha \in X(T)^*} \mathfrak{v}^{\alpha}$. According to the proof of Lemma \ref{Lie_N_red_stable_par_H} the decomposition of $\mathfrak{n}_{\red}$ is $H$-stable. It remains to show that so is any $\mathfrak{v}^{\alpha}$. 

Consider the following commutative diagram. As a reminder as $Z = T \times W$ and $T$ and $Z$ are normal in $N_{\red}$, so is the subgroup $W\subseteq N_{\red}$. Moreover $W$ is also unipotent smooth and connected so it is contained in $V:= \Rad_U(N_{\red})$:

\begin{figure}[H]
\begin{center}
\[\begin{tikzpicture} 

 \matrix (m) [matrix of math nodes,row sep=2em,column sep=1.8em,minimum width=2em,  text height = 1.5ex, text depth = 0.25ex]
  {
    1 & W & Z & T_Q,\\
    1 & V & N_{\red} & Q. \\};
  \path[-stealth]
  	(m-1-1) edge (m-1-2)
    (m-1-2) edge (m-1-3) 
    (m-1-3) edge (m-1-4) 
    (m-2-1) edge (m-2-2) 
    (m-2-2) edge (m-2-3) 
    (m-2-3) edge (m-2-4)
    (m-1-2) edge[right hook->] (m-2-2) 
    (m-1-3) edge[right hook->] (m-2-3)
    (m-1-4) edge[right hook->] (m-2-4) ;
    	
\end{tikzpicture}\]
\end{center}
\end{figure}
\noindent Let us first study the $H$-stability of the weight-zero part of $\mathfrak{v}$. The diagram above being cartesian one has $\mathfrak{v}^0 = \mathfrak{n}_{\red}^0 \cap \mathfrak{v} = \mathfrak{z} \cap \mathfrak{v} = \mathfrak{w}$. But $\mathfrak{w}$ is $H$-stable as the subgroups $T$ and $Z$ are (for $Z$ this has been shown in Lemma \ref{action_H_sur_centralisateur}) and the sequence is split.

Let us now focus on the positive weights. Let $\mathfrak{q}$ be the Lie algebra of the reductive quotient $N_{\red}/V$. The torus $T$ acts on this Lie algebra which writes $\mathfrak{q} = \mathfrak{q}^0 \oplus  \bigoplus_{\alpha \in X^*(T)} \mathfrak{q}^{\alpha}$. There are two possible situations: 
	\begin{itemize}
		\item either $\alpha$ is not a weight of $T$ on $\mathfrak{q}$. Then one has $\mathfrak{v}^{\alpha} = \mathfrak{n}_{\red}^{\alpha}$, whence the $H$-stability of $\mathfrak{v}^{\alpha}$;
		\item or $\alpha$ is a non trivial weight of $T$ on $\mathfrak{q}$. Then the weight spaces $\mathfrak{q}^{\alpha}$ and $\mathfrak{q}^{-\alpha}$ are of dimension $1$ (according to \cite[XIX, Proposition 1.12 (iii)]{SGA33}). As $p>2$ (because it is separably good for $G$), the pairing: \begin{alignat*}{3}
		\mathfrak{q}^{\alpha} \times \mathfrak{q}^{-\alpha} \:& \rightarrow \:& \mathfrak{q}^0 := \Lie(\:& T_Q) \\
		(X_{\alpha},X_{-\alpha}) \:& \mapsto \:& [X_{\alpha}, X_{-\alpha}]\:&
		\end{alignat*}		

induced by the bracket on $\mathfrak{q}$ is non-degenerate (see \cite[XXIII, Corollaire 6.5]{SGA33}), thus maps to a $1$-dimensional subspace $h_{\alpha}$.

Likewise, the bracket on $\mathfrak{n}_{\red}$ induces a non-degenerate pairing of $\mathfrak{n}_{\red}^{\alpha}$ and $\mathfrak{n}_{\red}^{-\alpha}$, and one has the following commutative diagram: 
\begin{figure}[H]
\begin{center}
\[\begin{tikzpicture} 

 \matrix (m) [matrix of math nodes,row sep=2em,column sep=1.8em,minimum width=2em,  text height = 1.5ex, text depth = 0.25ex]
  {
    \mathfrak{n}_{\red}^{\alpha} \times \mathfrak{n}_{\red}^{-\alpha} & \mathfrak{n}_{\red}^0,\\
    \mathfrak{q}^{\alpha} \times \mathfrak{q}^{-\alpha} & \mathfrak{n}_{\red}^0/\mathfrak{w} \cong \mathfrak{t_q=t/w}. \\};
  \path[-stealth]
  	(m-1-1) edge (m-1-2)
    (m-1-1) edge (m-2-1) 
    (m-2-1) edge (m-2-2) 
    (m-1-2) edge (m-2-2) 
     ;
    	
\end{tikzpicture}\]
\end{center}
\end{figure}
\noindent Denote by $d$ the image of the pairing of $\mathfrak{n}_{\red}^{\alpha}$ and  $\mathfrak{n}_{\red}^{-\alpha}$ composed with the projection \linebreak $\mathfrak{n}_{\red}^0 \rightarrow  \mathfrak{n}_{\red}^0/\mathfrak{w}$. According to what precedes this is a line of $\mathfrak{n}_{\red}^0/\mathfrak{w}$.\\
The situation can be summarized in the commutative diagram below:
\begin{figure}[H]
\begin{center}
\[\begin{tikzpicture} 

 \matrix (m) [matrix of math nodes,row sep=2em,column sep=1.8em,minimum width=2em,  text height = 1.5ex, text depth = 0.25ex]
  {
   0 & \mathfrak{v}^{\pm \alpha} & \mathfrak{n}_{\red}^{\pm\alpha} &  \mathfrak{q}^{\pm\alpha},\\
  & & \Hom(\mathfrak{n}_{\red}^{\mp\alpha},d).  & \\};
  \path[-stealth]
  	(m-1-1) edge (m-1-2)
  	(m-1-2) edge (m-1-3)
	(m-1-3) edge (m-1-4)
    (m-1-3) edge (m-2-3) 
    (m-1-4) edge (m-2-3) 
     ;
    	
\end{tikzpicture}\]
\end{center}
\end{figure}
\noindent In other words one has $\mathfrak{v}^{\pm \alpha} = \ker\left(\mathfrak{v}^{\pm \alpha} \rightarrow \Hom(\mathfrak{n}_{\red}^{\mp\alpha},d)\right)$ and $\mathfrak{v}^{\pm \alpha}$ is a sub-representation of the representation defined by the action of $H$ on $\mathfrak{n}_{\red}$, thus it is $H$-stable.  
	\end{itemize}	 
\end{proof}

Combining Lemmas \ref{Lie_N_red_stable_par_H} and \ref{Lie_rad_unip-stable_par_H} one can show an infinitesimal version of \cite[Théorème 2.5]{BDP}, namely:

\begin{proposition}
Let $G$ be a reductive group over an algebraically closed field $k$ of characteristic $p>0$ which is assumed to be separably good for $G$. Let $\phi : \mathcal{N}_{\red}(\mathfrak{g}) \rightarrow \mathcal{V}_{\red}(G)$ be a Springer isomorphism for $G$. If $N \subseteq G$ is a $\phi$-infinitesimally saturated subgroup, then:
\begin{enumerate}
	\item the Lie algebra $\mathfrak{n}_{\red}$ is an ideal of $\mathfrak{n}$,
	\item the Lie algebra of the unipotent radical of $N_{\red}$ is an ideal of $\mathfrak{n}$. 
\end{enumerate}
\label{generalisation_infinitesimale_Deligne}
\end{proposition}

\begin{proof}
The first point is provided by Lemma \ref{Lie_N_red_stable_par_H}. The second point follows from a direct application of Lemma \ref{Lie_rad_unip-stable_par_H} combined with \cite[Lemma 2.14]{BDP}: the subgroup $N$ being generated as an $\fppf$-sheaf by $H$ and $N_{\red}$, one only needs to show that $\radu(N_{\red})$ is $H$-stable. This has been shown by the aforementioned lemma. A reasoning on  $R[\epsilon]$-points for any $k$-algebra $R$ then leads to obtain that $\radu(N_{\red})$ is an ideal of $\mathfrak{n}$. 
\end{proof}

 	\subsection{Proof of Theorem \ref{generalisation_Deligne}}
We can now prove Theorem \ref{generalisation_Deligne}.

We start by showing that $N_{\red}$ is a normal subgroup of $N$. The latter being generated by $H$ and $N_{\red}$ as an $\fppf$-sheaf, one actually needs to show that $N_{\red}$ is $H$-stable. The reasoning follows the proof of Lemma \ref{Lie_N_red_stable_par_H}: we consider the subgroup $E_{Z^0, J_{\mathfrak{n}^{\alpha}}}$ generated by $Z^0 := Z^0_{N_{\red}}(T)$ and $J_{\mathfrak{n}^{\alpha}}$ for $\alpha \in X(T)^*$. Recall that the $J_{\mathfrak{n}^{\alpha}}$'s are themselves the subgroups generated as $\fppf$-sheaves by the image of the morphisms 
\begin{alignat*}{3}
\psi_{\alpha} :  W(\:&\mathfrak{n}^{\alpha}) \:&\times \: &  \mathbb{G}_a \:& \rightarrow \:& G \\
\:& (x,\:& \: & t) \:& \mapsto \:& \phi(tx).
\end{alignat*}
Note that $\psi_{\alpha}$ is well-defined for any $\alpha \in X(T)^*$: 
\begin{itemize}
\item any weight space $\mathfrak{n}^{\alpha}$ consists in $p$-nilpotent elements because we consider the action of a torus, 
\item any weight space $\mathfrak{n}^{\alpha}$ is geometrically reduced and geometrically connected as it is a vector space. 
\end{itemize}
\noindent Thus the groups $J_{\mathfrak{n}^{\alpha}}$ are smooth and connected (this last point is ensured by \cite[VIB, Corollaire 7.2.1]{SGA31}).
 
The arguments of the proof of Lemma \ref{action_H_sur_centralisateur} apply and allow to show that the $k$-subgroup $E_{Z^0, J_{\mathfrak{n}^{\alpha}}}$ is contained in $N$ (this subgroup being $\phi$-infinitesimally saturated), and even in $N_{\red}$ as it is smooth. Moreover recall that $p$ is not of torsion for $G$ and that for any non-zero weight the corresponding weight space is $p$-nil. Therefore, they are all embeddable into the Lie algebra of the unipotent radical of a Borel subgroup $B\subseteq G$. The weight spaces   
$\mathfrak{n}^{\alpha}$ are all contained in $\Lie(E_{Z^0, J_{\mathfrak{n}^{\alpha}}}) =: \mathfrak{e}$ because the differential at $0$ of the restriction of $\phi$ to the Lie algebra of the unipotent radical of any Borel subgroup is the identity. The Lie algebra $\Lie(Z^0)$ also satisfies this inclusion as $Z^0 \subset E_{Z^0, J_{\mathfrak{n}^{\alpha}}}$. 

To summarize we have shown that $\mathfrak{n}_{\red} = \Lie(Z^0) \oplus \bigoplus_{\alpha \in X(T)^*} \mathfrak{n}^{\alpha} \subseteq \mathfrak{e}$. The groups involved here being smooth and connected the equality of Lie algebras lifts to an equality of groups (see \cite[II, \S5, n\degree 5.5]{DG}) hence the identity $E_{Z^0, J_{\mathfrak{n}^{\alpha}}}= N^0_{\red}$. 

Thus the problem restricts to showing the $H$-stability of $E_{Z^0, J_{\mathfrak{n}^{\alpha}}}$. By Lemma \ref{action_H_sur_centralisateur} the centraliser $Z^0$ is $H$-invariant, so one only has to show the $H$-stability of the $J_{\mathfrak{n}^{\alpha}}$'s. As $H$ normalises $T$ and as $\phi$ is $G$-equivariant any $\mathfrak{n}^{\alpha}$ is $H$-invariant. Hence $N_{\red}$ is a normal subgroup of $N$.

Recall that in the preamble of section \ref{infinitesimal_version} we have explained that $H$ is actually equal to the product $T \times D$ (for $D$ a $k$-diagonalisable group). To prove that $N/N_{\red}$ is of multiplicative type we show that it is isomorphic to the group $D$. As:
\begin{itemize}
\item the group $H$ normalises $N_{\red}$ (which is normal in $N$)
\item the equality $HN_{\red} = N$ is satisfied as well as the following isomorphism $N_{\red} \cong H \cap N_{\red}$, \end{itemize}
\noindent one has an isomorphism of $\fppf$-sheaves which turns out to be an isomorphism of algebraic groups $H/H_{\red} \cong N/N_{\red} \cong D$.

To end the proof of the first point of the theorem it remains to show that the unipotent radical of $N_{\red}$, denoted by $V$, is normal in $N$. Once again, the $\fppf$-formalism reduces the problem to showing the $H$-invariance of $V$, the unipotent radical $\Rad_U(N_{\red})$ being normal in $N_{\red}$. The reasoning follows the proof of the normality of $N_{\red}$ in $N$: we consider the subgroup $E_{W,J_{\mathfrak{v}}^{\alpha}}$ generated by $W$ and $J_{\mathfrak{v}^{\alpha}}$ for $\alpha \in X(T)^*$. As $W$ and $J_{\mathfrak{v}^{\alpha}}$ are normal in $N_{\red}$, the subgroup $E_{W,J_{\mathfrak{v}^{\alpha}}}$ is a unipotent smooth connected normal subgroup of $N$, thus it is contained in the unipotent radical of $N_{\red}$. 

Moreover for any non zero weight, the corresponding weight space can be embedded into the Lie algebra of the unipotent radical of a Borel subgroup (as $p$ is not of torsion for $G$ and the considered weight-space is $p$-nil). Once again we make use of the properties of the differential of $\phi$ at $0$ to conclude that $\mathfrak{v}= \mathfrak{w} \oplus \bigoplus_{\alpha \in X(T)^*} \mathfrak{v}^{\alpha}$ is contained in $\Lie(E_{W,J_{\mathfrak{v}}^{\alpha}})$. This implies the equality $V=  E_{W,J_{\mathfrak{v}}^{\alpha}}$ for the same reasons as above. This equality being satisfied the result follows from stability properties established in the proof of Lemma \ref{Lie_rad_unip-stable_par_H}. Indeed we have shown that then $W$ as well as any $J_{\mathfrak{v}}^{\alpha}$, for non trivial $\alpha$, are $H$-stable. Combining this with the $G$-equivariance of $\phi$ leads to the conclusion that $V$ is a normal subgroup of $N$.

It remains to show the last point of Theorem \ref{generalisation_Deligne} which is a generalised version of \cite[Theorem 1.7 iii)]{D1} (see also \cite[Theorem 2.5 ii)]{BDP}). A careful reading of the proof of this latter shows that it does not depend on the additional assumptions made by the author (that, in practice, reduce the range of allowed characteristics). The arguments are hence the same as the one provided by P. Deligne in the framework of \cite[2.25]{D1} (see also \cite[Corollary 2.15]{BDP}) and the proof is reproduced here only for sake of clarity.

The reduced part $N_{\red} \subseteq N$ is now assumed to be reductive. We show that the connected component of the identity $M^0$ of $M= \ker\left(H\rightarrow \Aut(N_{\red})\right)$ is the central connected subgroup of multiplicative type we are seeking. It is clearly of multiplicative type as it is a closed subgroup of $H$ (see \cite[IV \S 1 Corollaire 2.4 a)]{DG}). Thus we need to show that it is central and that $M^0 \times N_{\red} \rightarrow N$ is an epimorphism. The first assertion is clear as:
\begin{itemize}
 \item the connected group $M^0$ centralises $N_{\red}$,
 \item and $N$ is generated by $H$ and $N_{\red}$ as a $\fppf$-sheaf (as shown previously).
\end{itemize} 
\noindent To show that $M^0 \times N_{\red} \rightarrow N$ is an epimorphism, one proves that $N$ is generated by $M^0$ and $N_{\red}$ as a $\fppf$-sheaf. We already know that $N$ is generated by $N_{\red}$ and $H$. To conclude we show:
\begin{itemize}
\item that $M$ is generated by $M_{\red}\subset N_{\red}$ and $M^0$,
\item that $H$ is generated by $M$ and $T$.
\end{itemize}  

The assertion for $M$ is the consequence of structural properties of groups multiplicative type: the field $k$ being algebraically closed, any group of multiplicative type is diagonalisable. Hence $M$ is isomorphic to a product of $\mathbb{G}_m$, $\mu_{q}$ and $\mu_{p^i}$ (for $(q,p) =1$) (see the proof of \cite[VIII, Proposition 2.1]{SGA31}). Its reduced component being smooth of multiplicative type, the order of its torsion part is coprime with $p$ (see \cite[VIII, Proposition 2.1]{SGA31}). Hence $M/M_{\red}$ is a product of groups of the form $\mu_{p^i}$ for $i \in \mathbb{N}$. Conversely, the quotient $M/M^0$ is a product of $\mu_{q}$ with $(p,q) = 1$, hence the result.

One still has to show that $H$ is generated by $M$ and $T$. Recall that we have shown previously that $N_{\red}$ is stable under the action of $H$-conjugation on $N$. Note that this action fixes $T$, hence we have the following diagram (according to \cite[XXIV, Proposition 2.11]{SGA33}, the group $N_{\red}$ being reductive):
\begin{figure}[H]
\begin{center}
\[\begin{tikzpicture} 

 \matrix (m) [matrix of math nodes,row sep=2em,column sep=1.8em,minimum width=2em,  text height = 1.5ex, text depth = 0.25ex]
  {
   H & \Aut(N_{\red}),\\
   T^{\Ad} & N_{\red}^{\Ad}.\\
 };
  \path[-stealth]
  	(m-1-1) edge (m-1-2)
  	(m-1-1) edge (m-2-1)
	(m-2-1) edge (m-2-2)
   (m-2-2) edge (m-1-2)
     ;
    	
\end{tikzpicture}\]
\end{center}
\end{figure}

The action of $H$ on $N_{\red}$ thus factors through $T^{\Ad}$. So we have the following exact sequence:   
\begin{figure}[H]
\begin{center}
\[\begin{tikzpicture} 

 \matrix (m) [matrix of math nodes,row sep=2em,column sep=1.8em,minimum width=2em,  text height = 1.5ex, text depth = 0.25ex]
  {
   1 & M & H & T^{\Ad} & 1.\\
 };
  \path[-stealth]
  	(m-1-1) edge (m-1-2)
  	(m-1-2) edge (m-1-3)
	(m-1-3) edge (m-1-4)
   (m-1-4) edge (m-1-5)
     ;
    	
\end{tikzpicture}\]
\end{center}
\end{figure}
\noindent Hence $H$ is generated as a $\fppf$-sheaf by  $M$ and $T^{\Ad}$, whence by $M$ and $T$ as $T^{\Ad}$ is a quotient of $T$. 

\section{Integration of some maximal $p$-nil $p$-subalgebras $\mathfrak{g}$}
	\label{ss_section_intégration_nil}

Let us start with the very specific case which has motivated our interest in the questions studied in this article:  assume $\mathfrak{u}\subseteq \mathfrak{g}$ to be a restricted $p$-subalgebra which is the set of $p$-nilpotent elements of $\rad(N_{\mathfrak{g}}(\mathfrak{u}))$. Note that $N_{\mathfrak{g}}(\mathfrak{u})$ is a restricted $p$-Lie algebra (according to Lemma \ref{lie_normalisateurs} as it derives from an algebraic $k$-group, namely $N_G(\mathfrak{u})$). Moreover $\mathfrak{u}$ is a restricted $p$-nil $p$-subalgebra of $\mathfrak{g}$.

\begin{lemma}
	Let $G$ be a reductive group over a field $k$ of characteristic $p>0$ which is assumed to be separably good for $G$ and let $\mathfrak{u} \subseteq \mathfrak{g}$ be a subalgebra. If $\mathfrak{u}$ is the set of $p$-nilpotent elements of the radical of its normaliser in $\mathfrak{g}$, denoted by $N_{\mathfrak{g}}(\mathfrak{u})$, the subalgebra $\mathfrak{u}$ is integrable by $J_{\mathfrak{u}}$.
	\label{p_nil_integrable}
	\end{lemma}

	\begin{proof}
According to Lemma \ref{alg_lie_groupe_engendre}, there is a unipotent smooth connected subgroup $J_{\mathfrak{u}}\subset G$ such that the inclusion  $\mathfrak{u}\subseteq \mathfrak{j_u}:= \Lie(J_{\mathfrak{u}})$ holds true. Moreover, as according to Lemma \ref{normalisateurs_inclusion} one has $J_{\mathfrak{u}} \subseteq N_G(J_{\mathfrak{u}}) \subseteq N_G(\mathfrak{j_u})$, at the Lie algebra level the following inclusions are satisfied:
\[\mathfrak{u} \subset \mathfrak{j_u} \subseteq \Lie(N_{G}(J_{\mathfrak{u}})) \subseteq N_{\mathfrak{g}}(\mathfrak{j_u}).\]

Assume the inclusion $\mathfrak{u} \subset \mathfrak{j_u}$ to be strict, then $\mathfrak{u}$ is a proper subalgebra of its normaliser in $\mathfrak{j_u}$ (this is a corollary of Engel Theorem, see for example \cite[\S4 n\degree 1 Proposition 3]{BOU1}). In other words one has $\mathfrak{u} \subsetneq N_{\mathfrak{j_u}}(\mathfrak{u}):= \mathfrak{j_u}\cap N_{\mathfrak{g}}(\mathfrak{u}).$ But according to Lemma \ref{inclusion_normalisateurs_J_u_et_u} the group $J_{\mathfrak{u}}$ is normalised by $N_G(\mathfrak{u})$, hence $N_{\mathfrak{j_u}}(\mathfrak{u}) \subseteq N_{\mathfrak{g}}(\mathfrak{u})$ is an ideal of $N_{\mathfrak{g}}(\mathfrak{u})$. It is:
\begin{itemize}
\item a restricted $p$-algebra (as it derives from an algebraic group according to Lemma \ref{lie_normalisateurs}), 
\item a restricted $p$-ideal (as the restriction of the $p$-structure of $N_{\mathfrak{g}}(\mathfrak{u})$ coincides with the one inherited from $N_{J_\mathfrak{u}}(\mathfrak{u})$), 
\item and even a $p$-nil $p$-ideal (as $\mathfrak{j_u}$ is $p$-nil according to Lemma \ref{rad_unip_p_nil}). 
\end{itemize}
\noindent In particular, it is a solvable ideal of $N_{\mathfrak{g}}(\mathfrak{u})$ whence the inclusion $N_{\mathfrak{j_u}}(\mathfrak{u}) \subseteq \rad(N_{\mathfrak{g}}(\mathfrak{u}))$. To summarize: the set $\mathfrak{u}$ of $p$-nilpotent elements of $\rad(N_{\mathfrak{g}}(\mathfrak{u}))$ is contained in a $p$-nil ideal of this radical (namely $N_{\mathfrak{j_u}}(\mathfrak{u})$) hence is equal to the latter. This contradicts the strictness of the inclusion, whence the equality $\mathfrak{u} = \mathfrak{j_u}$. This in particular means that there exists a unipotent smooth connected subgroup $J_{\mathfrak{u}} \subseteq G$ such that $\Lie(J_{\mathfrak{u}}) = \mathfrak{u}$. Thus $\mathfrak{u}$ is integrable.
 	\end{proof}

Let now $\mathfrak{h} \subseteq \mathfrak{g}$ be a subalgebra, and denote by $\mathfrak{u}$ the $p$-radical of $N_{\mathfrak{g}}(\mathfrak{h})$. The $p$-radical of $N_{\mathfrak{g}}(\mathfrak{h})$ being a restricted $p$-nil $p$-ideal, the work done in section \ref{section_Deligne}
allows to associate to $\mathfrak{u}$ a unipotent, smooth, connected subgroup $J_{\mathfrak{u}} \subset G$. 

\begin{lemma}
The subgroup $N_G(\mathfrak{h})$ normalises $J_\mathfrak{u}$.
\label{analogue_normalisation_rad_p}
\end{lemma}

	\begin{proof}
One only needs to apply verbatim the proof of Lemma \ref{inclusion_normalisateurs_J_u_et_u} as by assumption $\mathfrak{u}$ is an ideal of $N_{\mathfrak{g}}(\mathfrak{h})$.
	\end{proof}
	
	\begin{lemma}
	Let $G$ be a reductive group over a field $k$ of characteristic $p>0$ which is assumed to be separably good for $G$ and let $\mathfrak{h} \subseteq \mathfrak{g}$ be a subalgebra such that the normaliser $N_{G}(\mathfrak{h})$ is $\phi$-infinitesimally saturated. If $\mathfrak{u}:=\rad_p(N_{\mathfrak{g}}(\mathfrak{h}))$ then the subalgebra $\mathfrak{u}$ is integrable. 
	\label{p_rad_integrable}
	\end{lemma}
	
	\begin{proof}
	Recall that according to Lemma \ref{alg_lie_groupe_engendre} one has the inclusion 
$\mathfrak{u}\subseteq \mathfrak{j_u}$. Moreover $N_G(\mathfrak{h})$ being $\phi$-infinitesimally saturated, the group $J_{\mathfrak{u}}$ is a subgroup of $N_G(\mathfrak{h})$. At the Lie algebra level this leads to the following inclusions $\mathfrak{u} \subset \mathfrak{j_u} \subseteq \Lie(N_{G}(\mathfrak{h})) = N_{\mathfrak{g}}(\mathfrak{h}).$

Assume the inclusion $\mathfrak{u} \subsetneq \mathfrak{j_u}$ to be strict. Then $\mathfrak{u}$ is a proper subalgebra of its normaliser in $\mathfrak{j_u}$ (according to \cite[\S4 n\degree 1 Proposition 3]{BOU1}). In other words one has \[\mathfrak{u} \subsetneq N_{\mathfrak{j_u}}(\mathfrak{u}):= \mathfrak{j_u}\cap N_{\mathfrak{g}}(\mathfrak{u})=\mathfrak{j_u} \cap N_{\mathfrak{j_u}}(\mathfrak{h}).\] But according to Lemma \ref{analogue_normalisation_rad_p} the subgroup $J_{\mathfrak{u}}$ is normalised by $N_G(\mathfrak{h})$, hence $N_{\mathfrak{j_u}}(\mathfrak{u}) \subseteq N_{\mathfrak{g}}(\mathfrak{h})$ is an ideal of $N_{\mathfrak{g}}(\mathfrak{h})$. The same arguments as the ones developed in the proof of Lemma \ref{p_nil_integrable} allow us to show that it is a restricted $p$-nil $p$-ideal of $N_{\mathfrak{g}}(\mathfrak{h})$ such that $N_{\mathfrak{j_u}}(\mathfrak{u}) \subseteq \rad(N_{\mathfrak{g}}(\mathfrak{h}))$. This leads to the equality $N_{\mathfrak{j_u}}(\mathfrak{h})=\mathfrak{u}$ as $\mathfrak{u}$ is nothing but the set of $p$-nilpotent elements of $\rad(N_{\mathfrak{g}}(\mathfrak{h}))$. This contradicts the strictness of the inclusion, whence the equality $\mathfrak{u} = \mathfrak{j_u}$. In particular $\mathfrak{u}$ is integrable.
	\end{proof}
	
	\begin{remarks}
Let us better explicit the above condition of $\phi$-infinitesimal saturation with the two following remarks.
	\begin{enumerate}
	\item In the particular case when $\mathfrak{h} = \mathfrak{u}$, namely when $\mathfrak{u} := \rad_p(\N_{\mathfrak{g}}(\mathfrak{u}))$ is the $p$-radical of its normaliser in $\mathfrak{g}$, the $\phi$-infinitesimal saturation assumption is superfluous as in this case the inclusion $J_{\mathfrak{u}}\subseteq N_G(J_{\mathfrak{u}})$ is clear.
	\item The condition of $\phi$-infinitesimal saturation of normalisers might seem to be extremely restrictive. Let us stress out that there exists non-trivial examples of $\phi$-infinitesimally saturated normalisers: any parabolic subgroup satisfies this condition (according to Lemma \ref{parab_phi_inf_sat}) and in characteristic $p>2$ such subgroup appears to be the normaliser of its Lie algebra. Moreover, if $p>\h(G)$ one can show that the normaliser for the adjoint action of $G$ of any restricted $p$-nil $p$-subalgebra is $\exp$-infinitesimally saturated (or infinitesimally saturated).
	 \end{enumerate}
	\end{remarks}

\section{Added in proof: technical results on normalisers and centralisers}

\label{appendice_Lie}
The formalism used in this section is developed in \cite[II, \S 4]{DG}. We especially refer the reader to \cite[II, \S 4, 3.7]{DG} for notations. Let $A$ be a ring and $G$ be an affine $A$-group functor. As a reminder: 
	
\begin{enumerate}
	
	\item if $R$ is an $A$-algebra $R$, we denote by $R[t]$ the algebra of polynomials in $t$ and by $\epsilon$ the image of $t$ via the projection $R[t] \rightarrow R[t]/(t^2)=:R[\epsilon]$.
\noindent We associate to $G$ a functor in Lie algebras denoted by $\mathfrak{Lie}(G)$ and which is the kernel of the following exact sequence: 
	
	\begin{figure}[H]
	\begin{center}
\[\begin{tikzpicture} 

 \matrix (m) [matrix of math nodes,row sep=2em,column sep=1.8em,minimum width=2em,  text height = 1.5ex, text depth = 0.25ex]
  {
    1 & \mathfrak{Lie}(G)(R)& G(R[\epsilon]) & G(R) & 1.\\};
  \path[-stealth]
  	(m-1-1) edge (m-1-2)
    (m-1-2) edge (m-1-3) 
   	(m-1-3) edge node [below] {$p$}(m-1-4) 
   	(m-1-4) edge (m-1-5) 
   	(m-1-4) edge[bend left = -50]
   			node [above] {$i$}(m-1-3);
\end{tikzpicture}\]
\end{center}
\end{figure}
For any $y \in \mathfrak{Lie}(G)(R)$ we denote by $\e^{\epsilon y}$ the image of $y$ in $G(R[\epsilon])$. In what follows the notation $\mathfrak{Lie}(G)(R)$ refers both to the kernel of $p$ as well and to its image in $G(R[\epsilon])$. The Lie-algebra of $G$ is given by the $k$-algebra $\mathfrak{Lie}(G)(A)$ and is denoted by $\Lie(G) := \mathfrak{g}$. According to \cite[II, \S 4, n\degree 4.8, Proposition]{DG} when $G$ is smooth or when $A$ is a field and $G$ is locally of finite presentation over $A$, the equality $\Lie(G) \otimes_A R =  \mathfrak{Lie}(G)(A)\otimes_A R = \mathfrak{Lie}(G)(R) = \Lie(G_R)$ holds true for any $A$-algebra $R$ (these are sufficient conditions). When the aforementioned equality is satisfied the $A$-functor $\mathfrak{Lie}(G)$ is representable by $W(\mathfrak{g})$, where for any $A$-module $M$ and any $A$-algebra $R$ we set $W(M)(R) := M \otimes_A R$;
\item for any $A$-algebra $R$ we use the additive notation to describe the group law of $\mathfrak{Lie}(G)(R)$;
\item the $A$-group functor $G$ acts on $\mathfrak{Lie}(G)$ as follows: for any $A$-algebra $R$ the induced morphism is the following:
\begin{alignat*}{3}
\Ad_R : \: & G_R\: &\rightarrow \: & \Aut(\mathfrak{Lie}(G))(R),\\ 
\: & g \: & \mapsto \: & \Ad_R(g) : \mathfrak{Lie}(G)(R) \rightarrow \mathfrak{Lie}(G)(R) : x \mapsto i(g)xi(g)^{-1}.
\end{alignat*}
When $G$ is smooth (in particular when $\mathfrak{Lie}(G)$ is representable) the $G$-action on $\mathfrak{Lie}(G)$ defines a linear representation $G \rightarrow \GL(\mathfrak{g})$ (see \cite[II, \S 4, n\degree 4.8, Proposition]{DG}).
\end{enumerate}

	\subsection{Centralisers}
	
\begin{lemma}
Let $A$ be a ring and set $S = \Spec(A)$. If $G$ is a smooth affine $S$-group scheme, the equality $\Lie(Z_G(\mathfrak{h})) = Z_{\mathfrak{g}}(\mathfrak{h})$ is satisfied for any subspace $\mathfrak{h}\subset \mathfrak{g}$.
\label{lie_centralisateurs}
\end{lemma}

\begin{proof}
By definition one has:

\begin{alignat*}{3}
\Lie(Z_G(\mathfrak{h}))\: & = \mathfrak{g} \cap Z_G(\mathfrak{h})(A[\epsilon])\\
\: & = \{g \in \mathfrak{g} \mid \Ad(g_{A[\epsilon]})(x) = x, \forall x \in \mathfrak{h}(A[\epsilon])\}.\\
\end{alignat*}
The last identity can be rewritten as $e^{\epsilon g}e^{\epsilon' x}e^{-\epsilon g}e^{-\epsilon' x} = e^{\epsilon \epsilon'[g,x]} = 1$ in $G(A[\epsilon, \epsilon'])$, whence the vanishing of the Lie bracket $[g,x]$ (which is a condition in $G(A[\epsilon])$). This leads to the following equality: 
\[\Lie(Z_G(\mathfrak{h})) = \{g \in \mathfrak{g} \mid [g,x] = 0, \forall x \in \mathfrak{h}(A[\epsilon])\} = Z_ {\mathfrak{g}}(\mathfrak{h}).\]

\end{proof}

\begin{remarks}
Let us emphasize some very particular behaviours of the center:
\begin{enumerate}
\item Let $Z_G(\mathfrak{h})_{\red}$ be the reduced part of the centraliser. Even when $k$ is an algebraically closed field, the equality $\Lie(Z_G(\mathfrak{h})_{\red}) = Z_{\mathfrak{g}}(\mathfrak{h})$ is a priori not satisfied (see for example \cite[2.3]{Jantzen2004}). 
\item Let $S:= \Spec(A)$ be an affine scheme and $G$ be a $S$-group scheme. Assume $Z_G$ to be representable (this condition is in particular satisfied when $G$ is locally free and separated (see \cite[II, \S1, n\degree 3.6 c), Théorème]{DG}). As mentioned in \cite[II, 5.3.3]{SGA31} the algebra $\Lie(Z_G):= \mathfrak{Lie}(Z_G)(A)$ is a subalgebra of $\mathfrak{z_g}$.

According to \cite[XII Théorème 4.7 d) and Proposition 4.11]{SGA32} when $G$ is smooth affine of connected fibers and of zero unipotent rank over $S$, the center of $G$ is the kernel of the adjoint representation $\Ad : G \rightarrow \GL(\mathfrak{g})$. Under these assumptions the equality $\Lie(Z_G)= \mathfrak{z_g}$ holds true. Indeed the following exact sequence of algebraic groups:

\begin{figure}[H]
\begin{center}
\[\begin{tikzpicture} 

 \matrix (m) [matrix of math nodes,row sep=2em,column sep=4.8em,minimum width=2em,  text height = 1.5ex, text depth = 0.25ex]
  {
    1 & Z_G & G & \GL(\mathfrak{g}),\\
  };
  \path[-stealth]
  	(m-1-1) edge (m-1-2)
    (m-1-2) edge (m-1-3) 
    (m-1-3) edge node[above] {$\Ad$} (m-1-4) 
    ;
\end{tikzpicture}\]
\end{center}
\end{figure}
\noindent induces by derivation an exact sequence (see \cite[II, \S 4, n\degree 1.5]{DG}):
\begin{figure}[H]
\begin{center}
\[\begin{tikzpicture} 

 \matrix (m) [matrix of math nodes,row sep=2em,column sep=5.4em,minimum width=2em,  text height = 1.5ex, text depth = 0.25ex]
  {
    0 & \Lie(Z_G) & \mathfrak{g} & \End(\mathfrak{g}).\\
  };
  \path[-stealth]
  	(m-1-1) edge (m-1-2)
    (m-1-2) edge (m-1-3) 
    (m-1-3) edge node[above] {$\ad := \Lie(\Ad) $} (m-1-4) 
    ;  	
\end{tikzpicture}\]
\end{center}
\end{figure}
\noindent The desired equality follows as by definition $\mathfrak{z_g}:= \ker(\ad)$. Let us emphasise that this in particular applies to any reductive $S$-group $G$ and to any parabolic subgroup $P\subseteq G$ (as any Cartan subgroup of $P$ is a Cartan subgroup of $G$).
\end{enumerate}
\label{centre_noyau_ad}
\end{remarks}

	\subsection{Normalisers}
	
Let $S=\Spec(A)$ be an affine scheme and $G$ be a smooth $S$-group scheme of finite presentation. In what follows $H \subseteq G$ is a closed locally free subgroup. Let us stress out that under these conditions the normaliser $N_G(H)$ is representable by a closed group-sub-functor of $G$ according to \cite[II, \S 1 n\degree 3, Théorème 3.6 b)]{DG}. Moreover, if $H$ is smooth, the aforementioned theorem provides the representability of $N_G(\mathfrak{Lie}(H))= N_G(\mathfrak{h})$ as then $\mathfrak{Lie}(H)$ is representable by $W(\mathfrak{h})$ which is locally free.

\begin{lemma}
If $H \subseteq G$ is a closed subgroup then the inclusion $N_G(H) \subseteq N_G(\mathfrak{Lie}(H))$ is satisfied. In particular if $H$ is smooth this leads to the inclusion $N_G(H)(R) \subseteq N_G(\mathfrak{h}_R)$ for any $A$-algebra $R$.
\label{normalisateurs_inclusion}
\end{lemma}

\begin{proof}
Let us remind that $G$ acts on $\mathfrak{Lie}(G)$ via the adjoint representation. Namely for any $A$-algebra $R$ one has: 
\begin{alignat*}{3}
\Ad_R : \: & G_R\: &\rightarrow \: & \GL(\mathfrak{Lie}(G))(R),\\ 
\: & g \: & \mapsto \: & \Ad_R(g) : \mathfrak{Lie}(G)(R) \rightarrow \mathfrak{Lie}(G)(R) : x \mapsto i(g)xi(g)^{-1}.
\end{alignat*}

\noindent Let $g \in N_G(H)(R):= \{g' \in G(R) \mid \Ad(g')(H\otimes_A R) = H\otimes_A R\}$ (see for example \cite[II, \S 1, n\degree 3.4 Definition]{DG}). In particular, for any $x \in \mathfrak{Lie}(H)(R)$, one has: 
\[\Ad(g)(x) = i(g)xi(g)^{-1} \in H(R[\epsilon]) \cap \mathfrak{Lie}(H)(R),\] 
\noindent hence the inclusion $N_G(H) \subseteq N_G(\mathfrak{Lie}(H))$.

If now $H$ is smooth then $\mathfrak{Lie}(H)$ is representable by a $A$-functor of Lie algebras and one has: 
\[\mathfrak{Lie}(H)(A)\otimes_A R = \mathfrak{Lie}(H)(R) = \Lie(H_R):= \mathfrak{h}_R\] 
\noindent for any $A$-algebra $R$.
\end{proof}

\begin{lemma}
Let $\mathfrak{h} \subseteq \mathfrak{g}$ be a Lie subalgebra. Then one has $\Lie(\N_G(\mathfrak{h})) = \N_{\mathfrak{g}}(\mathfrak{h})$.
\label{lie_normalisateurs}
\end{lemma}

\begin{proof} 
By definition one has that:
\begin{alignat*}{3}
\Lie(N_G(\mathfrak{h}))\: & = \mathfrak{g} \cap N_G(\mathfrak{h})(A[\epsilon])\\
\: & = \{g \in \mathfrak{g} \mid \Ad(g_{A[\epsilon]})(x) \in \mathfrak{h}_{A[\epsilon]}, \forall x \in \mathfrak{h}(A[\epsilon])\}.\\
\end{alignat*}
The last relation writes: 
\[\Ad(e^{\epsilon g})e^{\epsilon' x} = e^{\epsilon'x}e^{\epsilon'\epsilon[g,x]} = e^{\epsilon' (x + \epsilon [g,x])}\in \mathfrak{h}_R \cap G(A[\epsilon,\epsilon']),\]
\noindent in $G(A[\epsilon, \epsilon'])$, because $\epsilon^2 =0$. In other words one has:
\begin{alignat*}{3}
\Lie(N_G(\mathfrak{h})) \: & = \{g \in \mathfrak{g} \mid x + \epsilon[g,x] \in \mathfrak{h}_{A[\epsilon]}, \forall x \in \mathfrak{h}(A[\epsilon])\}\\
\: & = \{g \in \mathfrak{g} \mid \epsilon[g,x] \in \mathfrak{h}_{A[\epsilon]}, \forall x \in \mathfrak{h}(A[\epsilon])\}\\ 
\: & = \{g \in \mathfrak{g} \mid [g,x] \in \mathfrak{h}_{A[\epsilon]}, \forall x \in \mathfrak{h}(A[\epsilon])\}\\
\: &= N_{\mathfrak{g}}(\mathfrak{h}). 
\end{alignat*}
\end{proof}

The second part of the following lemma is shown in the proof of \cite[Proposition 3.5.7]{CGP} when $k$ is a separably closed field. The study of the proof shows that one actually only needs $H(k)$ to be Zariski-dense in $H$ for the result to hold true. Let us stress out that this is especially verified when:
	\begin{enumerate}
	\item the field $k$ is perfect and the subgroup $H$ is connected (see \cite[Corollary 18.2]{BORlag}),
	\item the field $k$ is infinite and the subgroup $H$ is reductive (see \cite[Corollary 18.2]{BORlag}),
	\item the subgroup $H$ is unipotent smooth connected and split. Indeed, under these assumptions $H$ is isomorphic to a product of $\mathbb{G}_a$s. These conditions are especially satisfied when $k$ is perfect and $H$ is unipotent smooth and connected (which is a special case of (i)).
	\end{enumerate}
	
\begin{lemma}
Let $H\subseteq G$ be a closed and smooth subgroup. Then:
\begin{enumerate}
\item in general only the inclusion $\Lie(N_G(H)) \subseteq N_{\mathfrak{g}}(\mathfrak{h})$ holds true,
\item if $H(k)$ is Zariski-dense in $H$ then
\[\Lie(N_G(H)) = \lbrace x \in \mathfrak{g} \mid \Ad(h)(x) - x \in \mathfrak{h} \ \forall h \in H(k)\rbrace.\]
\end{enumerate} 
\label{lie_norm_inclusion}
\end{lemma}

\begin{proof}
The inclusion $N_G(H) \subseteq N_{G}(\mathfrak{Lie}(H))$ is provided by Lemma \ref{normalisateurs_inclusion}. Combining this together with the equality obtained in Lemma \ref{lie_norm_inclusion} one obtains:
\[\Lie(N_G(H)) \subseteq \Lie(N_G(\mathfrak{h})) = N_{\mathfrak{g}}(\mathfrak{h}).\] As already mentioned, the second assertion of the lemma is shown in \cite[Proposition 3.5.7]{CGP}. 
\end{proof}

\begin{remarks}
The first point of the above lemma provides a strict inclusion of Lie algebras in the general case. This is actually a positive characteristic phenomenon (see \cite[10.5 Corollary B]{HUM3} and the remark that follows Corollary B):
	\begin{enumerate}
		\item when $k$ is of characteristic $0$ the aforementioned inclusion is always an equality (see \cite[13. Exercise 1]{HUM3}),
		\item when $k$ is of characteristic $p>0$, the inclusion may be strict as shown on the following example (see \cite[10 Exercise 4]{HUM3}): assume $p=2$. Set $G = \SL_2$ and consider the Borel subgroup $B$ of upper triangular matrices. The group $B$ being parabolic it is its self normaliser. In other words one has $N_G(B) = B$. However, at the Lie algebra level one has $\Lie(N_G(B)) = \mathfrak{g}$ (as $k$ is of characteristic $2$). Indeed $\ssl_2$ is generated by $\begin{pmatrix}
		1&0\\
		0&1
		\end{pmatrix}, \ \begin{pmatrix}
		0&1\\
		0&0
		\end{pmatrix}, \ \begin{pmatrix}
		0 & 0\\
		1 &0
		\end{pmatrix}$
and one only needs to show that the bracket of the following two matrices $\begin{pmatrix}
0 & 0\\
1 & 0 
\end{pmatrix} $ and $\begin{pmatrix}
0 & 1\\
0 & 0 
\end{pmatrix}$ still belongs to $\mathfrak{b}$. One has: $\left[\begin{pmatrix}
0 & 1\\
0 & 0 
\end{pmatrix},  \begin{pmatrix}
0 & 0\\
1 & 0 
\end{pmatrix}\right] = \id \in \mathfrak{b}$.
	\end{enumerate}
\end{remarks}

\underline{\textbf{Acknowledgements :}}
The author would like to thank Philippe Gille for all the fruitful discussions they had which go beyond the framework of this paper, his availability and his multiple reviews and corrections of this article; Benoît Dejoncheere for his supportive help, his several reviews and useful advice; Matthieu Romagny for conversations on the nilpotent scheme; as well as both reviewers of her PhD manuscript: Anne-Marie Aubert and Vikraman Balaji for their corrections and remarks.  Finally many thanks should go to the anonymous referee of this article for raising many interesting points that definitely helped to improve its content. Any critical remark must be exclusively addressed to the author of this paper.

\bibliographystyle{amsalpha}
\bibliography{bib}
\end{document}